\documentclass[12pt]{article}
\usepackage{amsmath}
\usepackage{amssymb}
\usepackage{amsthm}
\usepackage{latexsym}
\usepackage{color,fancyhdr,lastpage,graphicx}
\usepackage{subfigure, epsfig, epsf}
\usepackage{xspace}
\usepackage{setspace}
\usepackage{bbm}
\usepackage{esint,stmaryrd}
\usepackage{mathrsfs} 
\usepackage[toc]{appendix}
\usepackage{mathtools}

\theoremstyle{plain}

\newtheorem{lemma}{Lemma}[section]

\newtheorem{theorem}[lemma]{Theorem}
\newtheorem{proposition}[lemma]{Proposition}

\theoremstyle{remark}

\definecolor{blue}{rgb}{0.0,0.02,0.67}
\def\bb{\begin{color}{blue}}
\def\bw{\begin{color}{white}}
\def\bg{\begin{color}{green}}
\def\br{\begin{color}{red}}
\def\bbr{\begin{color}{brown}}
\def\eg{\end{color}}
\def\ew{\end{color}}
\def\er{\end{color}}
\def\eb{\end{color}}
\def\jj{\j}
\def\jj{}

\def\PP{\mathbb{P}}

\def\H{\mathbb{H}}

\def\ve{{\varepsilon}}
\def\le{\leqslant}
\def\leq{\leqslant}

\def\tn{\interleave}

\def\ge{\geqslant}
\def\geq{\geqslant}

\def\E{{\mathbb E}}
\def\O{{\Omega}}

\def\R{{\mathbb R}}
\def\C{{\mathbb C}}
\def\N{{\mathbb N}}

\def\Z{{\mathbb Z}}

\def\a{{\alpha}}
\def\b{{\beta}}
\def\d{{\delta}}
\def\D{{\Delta}}
\def\t{{\tau}}

\def\g{{\gamma}}

\def\s{{\sigma}}
\def\l{{\lambda}}
\def\L{{\Lambda}}
\def\th{{\theta}}
\def\Th{{\Theta}}

\def\cR{{\cal R}}

\def\cC{{\cal C}}

\def\cF{{\cal F}}
\def\cH{{\cal H}}

\def\cW{{\cal W}}
\def\cM{{\cal M}}
\def\cN{{\cal N}}

\def\bw{{\text{\boldmath$w$}}}

\def\q{\quad}

\def\re{\operatorname{Re}}
\def\im{\operatorname{Im}}
\def\cp{\operatorname{cap}}

\def\<{\langle}
\def\>{\rangle}
\def\ra{\rightarrow}

\def\sse{\subseteq}
\def\sm{\setminus}

\def\cadlag{{c{\`a}dl{\`a}g }}

\begin{document}
\bibliographystyle{plain}


\begin{center}
\LARGE \textbf{Scaling limits for planar aggregation with subcritical fluctuations}

\vspace{0.2in}

\large {\bfseries 
James Norris\footnote{Statistical Laboratory, Centre for Mathematical Sciences, Wilberforce Road, Cambridge, CB3 0WB, UK. 
 Email:  j.r.norris@statslab.cam.ac.uk}, 
Vittoria Silvestri\footnote{University of Rome La Sapienza, Piazzale Aldo Moro 5, 00185, Rome, Italy. \\ 
Email: silvestri@mat.uniroma1.it},   
Amanda Turner\footnote{Department of Mathematics and Statistics, Lancaster University, Lancaster LA1 4YF, UK. \\
Email: a.g.turner@lancaster.ac.uk}
}


\end{center}
\begin{abstract} 
We study scaling limits of a family of planar random growth processes in which clusters grow by the successive aggregation of small particles. 
In these models, clusters are encoded as a composition of conformal maps and the location of each successive particle is distributed according to the density of harmonic measure on the cluster boundary, raised to some power. 
We show that, when this power lies within a particular range, the macroscopic shape of the cluster converges to a disk, but that as the power approaches the edge of this range the fluctuations approach a critical point, which is a limit of stability. The methodology developed in this paper provides a blueprint for analysing more general random growth models, such as the Hastings-Levitov family. 
\end{abstract}

\setcounter{tocdepth}{2} 
\tableofcontents


\section{Introduction}
\label{INT}
We study a family of planar random growth processes in which clusters grow by the successive aggregation of particles. 
Clusters are encoded as a composition of conformal maps, following an approach first introduced by  Carleson and Makarov \cite{C+M} and Hastings and Levitov \cite{HL}. 
The specific models that we study fall into the class of Laplacian growth models in which the growth rate of the cluster boundary is determined by the density of harmonic measure of the boundary as seen from infinity. 
In our case, the location of each successive particle is distributed according to the density of harmonic measure raised to some power. 
Our set-up is closely related to that of the Hastings-Levitov family of models,  HL($\a$), $\a \in [0,\infty)$ \cite{HL}, which includes off-lattice versions of the physically occurring dielectric-breakdown models \cite{NPW}, 
in particular the Eden model for biological growth \cite{Eden} and diffusion-limited aggregation (DLA) \cite{W+S}.
Our family of models shares with the HL($0$) model the unphysical feature that new particles are distorted by the conformal map
which encodes the current cluster.
However, in subsequent work \cite{NST2}, we show that these models share common behaviour with the HL($\a$) models when $\a \neq 0$, 
so the present paper serves to develop methods applicable to these more physical models.

We establish scaling limits of the growth processes in the small-particle scaling regime where the size of each particle converges to zero as the number of particles becomes large. 
We show that, when the power of harmonic measure is chosen within a particular range, 
the macroscopic shape of the cluster converges to a disk, 
but that as the power approaches the edge of this range the fluctuations approach a critical point, which is a limit of stability. 
This phase transition in fluctuations can be interpreted as the beginnings of a macroscopic phase transition, 
from disks to non-disks.

\subsection{Description of the model}
\label{DESM}
Our clusters will grow from the unit disk by the aggregation of many small particles.
Let
$$
K_0=\{z\in\C:|z|\le1\},\q D_0=\{z\in\C:|z|>1\}. 
$$
We fix a non-empty subset $P$ of $D_0$ and set
\[
K=K_0\cup P,\q D=D_0\sm P.
\]
We assume that $P$ is chosen so that $K$ is compact and simply connected.
Then we call $P$ a basic particle.

We will call a conformal map $F$, defined on $D_0$ and having values in $D_0$, a basic map if it is univalent
and satisfies
$$
F(\infty)=\infty,\q F'(\infty)\in(1,\infty).
$$
By the Riemann mapping theorem, there is a one-to-one correspondence between basic particles and basic maps given by
$$
P=\{z\in D_0:z\not\in F(D_0)\}.
$$
For convenience, we will assume throughout that $F$ has a continuous extension to the unit circle.
It is well understood geometrically when this holds.
The map $F$ has the form
\begin{equation}
\label{LAUR}
F(z)=e^c\left(z+\sum_{k=0}^\infty a_kz^{-k}\right)
\end{equation} 
for some $c>0$ and sequence $(a_k:k\ge0)$ in $\C$. The value $e^c$ is called the logarithmic capacity of the cluster $K$.
We define the capacity of the particle $P$ (or, interchangeably, of the map $F$) by
$$
\cp(P)=\log F'(\infty) = c.
$$

For the purpose of this introduction, we will assume that we have chosen a family of basic particles $(P^{(c)}:c\in(0,\infty))$, 
such that $\cp(P^{(c)})=c$. 
Figure \ref{fig:particles} shows four representative particles from some families we have in mind.
\begin{figure}[h!]
  \centering
    \includegraphics[width=0.85\textwidth]{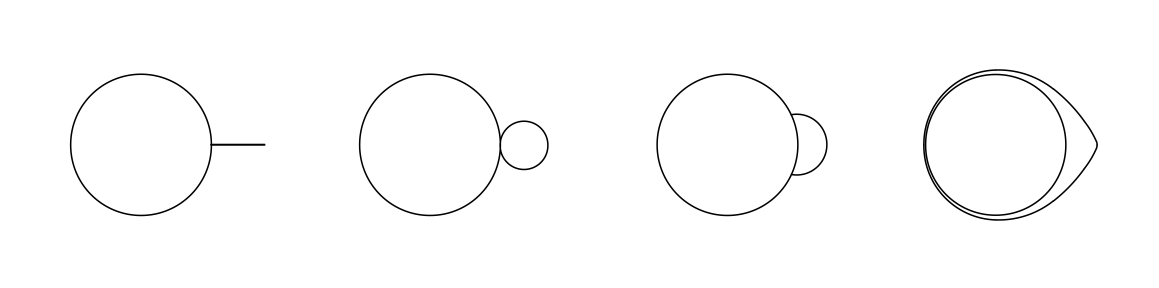}
    \caption{Examples of basic particles.} \label{fig:particles}
\end{figure}
Write $(F^{(c)}:c\in(0,\infty))$ for the family of associated basic maps.
Given a sequence of attachment angles $(\Theta_n:n\ge1)$ and capacities $(c_n:n\ge1)$, set 
$$
F_n(z)=e^{i\Th_n}F^{(c_n)}(e^{-i\Th_n}z).
$$
Define a process $(\Phi_n:n\ge0)$ of conformal maps on $D_0$ as follows:
set $\Phi_0(z)=z$ and for $n\ge1$ define recursively
\begin{equation}
\label{phidef}
\Phi_n=\Phi_{n-1}\circ F_n=F_1\circ\dots\circ F_n.
\end{equation}
Then $\Phi_n$ encodes a compact set $K_n\sse\C$, given by
$$
K_n=K_0\cup\{z\in D_0:z\not\in\Phi_n(D_0)\}
$$
and $\Phi_n$ is the unique conformal map $D_0\to D_n$ such that
$$
\Phi_n(\infty)=\infty,\q \Phi_n'(\infty)\in(1,\infty)
$$
where $D_n=\C\sm K_n$. It is straightforward to see that
$$
\cp(K_n)=\log\Phi_n'(\infty)=c_1+\dots+c_n
$$
and that $K_n$ may be written as the following disjoint union
$$
K_n=K_0\cup(e^{i\Th_1}P^{(c_1)})\cup\Phi_1(e^{i\Th_2}P^{(c_2)})\cup\dots\cup\Phi_{n-1}(e^{i\Th_n}P^{(c_n)}).
$$
We think of the compact set $K_n$ as a cluster, 
formed from the unit disk $K_0$ by the addition of $n$ particles.

By choosing the sequences $(\Theta_n:n\ge1)$ and $(c_n:n\ge1)$ in different ways, we can obtain a wide variety of growth processes.
In the aggregate Loewner evolution (ALE) model with parameters $\a\in\R$, $\eta\in\R$, $c\in(0,\infty)$ and $\sigma\in[0,\infty)$, 
which was introduced in \cite{STV} for slit particles, and abbreviated as ALE($\a,\eta$), we set
\begin{equation}
\label{hn}
h_n(\th)=\frac{|\Phi_{n-1}'(e^{\s+i\th})|^{-\eta}}{Z_{n}},\q
Z_n=\fint_0^{2\pi}|\Phi_{n-1}'(e^{\s+i\th})|^{-\eta}d\th
=\frac1{2\pi}\int_0^{2\pi}|\Phi_{n-1}'(e^{\s+i\th})|^{-\eta}d\th
\end{equation}
and we take 
$$
c_n=c|\Phi_{n-1}'(e^{\s+i\Th_n})|^{-\a}
$$
with $(\Th_n:n\ge1)$ a sequence of random variables whose distribution given by
$$
\PP(\Th_n\in B|\cF_{n-1})
=\fint_0^{2\pi}1_B(\th)h_n(\th)d\th
$$
where $\cF_n=\s(\Th_1,\dots,\Th_n)$.
In this paper, we will consider only the case where $\a=0$, 
which takes as data a single basic map $F=F^{(c)}$ and a choice of $\eta\in\R$ and $\s\in[0,\infty)$.
For simplicity, we refer to this model here as the ALE$(\eta)$ model with basic map $F$ and regularization parameter $\s$.

If, on the other hand, we were to take $\eta=\s=0$ and fix $\a\in[0,\infty)$, 
then we would obtain the HL($\a$) model considered by Hastings and Levitov \cite{HL}.
The parameters $\a$ and $\eta$ play a similar role in adjusting the `local growth rate of capacity' as a function of the current cluster shape. 
Indeed, in the subsequent paper \cite{NST2} 
 we show that, 
modulo a deterministic time-change and under the same restrictions on the parameter $\sigma$ as will be used in this paper, 
the scaling limit of ALE($\a,\eta$) depends primarily on the sum $\a+\eta$ provided that $\a+\eta\le1$. 
This means that ALE($\eta$) and regularized HL($\alpha$) have qualitatively similar behaviour when $\a=\eta$. 
Moreover, the range of the attachment densities considered in ALE($\eta$) corresponds exactly to those used to define the
dielectric-breakdown models, 
so the full family ALE($\a,\eta$) is of wider interest than HL($\a$) alone. 
See \cite{STV} for a comprehensive discussion of other models related to ALE. 

One of the challenges of studying HL($\a$) when $\a\neq0$ is that the capacity of the cluster $K_n$ is random and could be quite badly behaved. 
It is therefore a priori unclear how to tune parameters in order to obtain non-trivial scaling limits. 
One way in which ALE($\eta$) is simpler is that the capacity of the cluster $K_n$ is always $cn$, 
where $c=\log F'(\infty)$. 
Nevertheless, the models have much in common, 
and it has turned out that the framework developed here for ALE($\eta$) provides useful ideas for the analysis of HL($\a$). 
In this paper we will focus on the case where $\eta\in(-\infty,1]$. We will establish scaling limits and fluctuations for ALE($\eta$) in the small-particle regime, 
where simultaneously $c\to0$, $\s\to0$ and $n\to\infty$ with $n$ tuned so that $nc \to t$, for some fixed $t \in \mathbb{R}$, thereby giving clusters of macroscopic capacity. 

\subsection{Review of related work}
Much effort has been devoted to the analysis of lattice-based random growth models. These are models in which, at each step, a lattice site adjacent
to the current cluster is added, chosen according to a distribution determined by the current cluster.
Examples include the Eden model \cite{Eden}, diffusion limited aggregation (DLA) \cite{W+S} and the family of dielectric-breakdown models \cite{NPW}. 
Around 20 years ago, Carleson and Makarov \cite{C+M} and Hastings and Levitov \cite{HL} introduced an alternative approach in the planar case, 
which allows the formulation of a discrete particle model directly in the continuum 
by encoding clusters in terms of conformal maps, as described in the preceding subsection.
In \cite{C+M}, the authors obtained a growth estimate for a deterministic analogue of DLA which is formulated in terms of the Loewner equation. 
In \cite{HL}, 
the HL($\a$) model was studied numerically and experimental evidence was shown for a phase transition in behaviour at $\alpha=1$: 
when $\alpha<1$, clusters appeared to converge to disks; on the other hand, when $\alpha>1$, a turbulent growth regime emerged, in which clusters behaved randomly at large scale. 
Hastings and Levitov argued that HL(1) is a candidate for an off-lattice version of the Eden model, and HL(2) corresponds to DLA. 
Establishing the existence of this phase transition rigorously is one of the main open problems in this area.

In \cite{STV}, Sola, Turner and Viklund showed the existence of a phase transition in the ALE($\eta$) model.
They showed that, for $\eta>1$, if particles are taken to be slits, 
and the regularisation parameter $\sigma$ is sufficiently small then, in the small-particle limit, 
the clusters themselves grow from the unit disk by the emergence of a radial slit, at a random angle.

This behaviour is qualitatively different to the known behaviour of ALE(0), that is to say HL(0), in the same scaling regime.
In \cite{NT2}, 
Norris and Turner showed that the HL(0) clusters converge to disks with internal branching structure given by the Brownian web. 
More recently, 
Silvestri \cite{Sil} analysed the fluctuations in HL(0) and showed that these converge to a log-correlated fractional Gaussian field. 
Several other papers consider modifications of the HL(0) model \cite{JST12,JST15,RZ}.

In this paper, we approach the question of the phase transition in ALE$(\eta)$ at $\eta=1$ from the opposite direction to that in \cite{STV}  by showing convergence to a disk for ALE($\eta$) for all $\eta\le1$, provided that $\s$ does not converge to zero too fast.
Further, we prove convergence of the associated fluctuations to an explicit limit, which depends on $\eta$, 
and which would exhibit unstable behaviour if one took $\eta>1$.
Our results apply in a different regime to that considered in \cite{STV}. 
We require that the regularization parameter $\s\gg c^{1/2}$ (and sometimes more), which enables us to show that, 
for each $\eta\le 1$, the disk limit and the fluctuations hold universally for a wide class of particle shapes. 
By contrast, in \cite{STV} the parameter $\s\ll c$ and the results rely heavily on the slit particle being non-differentiable at its tip. 

\subsection{Statement of results}
Our main results will be proved under the technical assumption \eqref{partcond} below, 
which we will show in Appendix \ref{sec:particle} to be satisfied for small particles of any given shape.
This assumption expresses that the basic particle $P$ is concentrated near the point $1$ on the unit circle 
in a certain controlled way.
Let $F$ be a basic map of capacity $c\in(0,1]$, in the sense of Subsection \ref{DESM}, 
that is to say, a univalent conformal map from $\{|z|>1\}$ into $\{|z|>1\}$ such that $F(z)/z\to e^c$ as $z\to\infty$.
We say that $F$ has regularity $\L\in[0,\infty)$ if, for all $|z|>1$,
\begin{equation}
\label{partcond}
\left|\log\left(\frac{F(z)}z\right)-c\frac{z+1}{z-1}\right|\le\frac{\L c^{3/2}|z|}{|z-1|(|z|-1)}.
\end{equation} 
Here and below we choose the branch of the logarithm so that $\log(F(z)/z)$ is continuous on $\{|z|>1\}$ with limit $c$ at $\infty$.
Our results will concern the limit $c\to0$ with $\L$ fixed, but are otherwise universal in the choice of particle.
We will show that, for $\eta\in(-\infty,1]$, in this limit, 
provided the regularisation parameter $\s$ does not converge to 0 too fast, 
the cluster $K_n$ converges to a disk of radius $e^{cn}$, and the fluctuations, suitably rescaled,
converge to the solution of a certain stochastic partial differential equation.

\begin{theorem}
\label{UFA}
Let $\eta\in(-\infty,1]$, $\L\in[0,\infty)$ and $\ve\in(0,1/2)$ be given.
Let $(\Phi_n:n\ge0)$ be an ALE$(\eta)$ process with basic map $F$ and regularization parameter $\s$. 
Assume that $F$ has capacity $c$ and regularity $\L$, and that $e^\s\ge1+c^{1/2-\ve}$.
For all $\eta\in(-\infty,1)$, $m\in\N$ and $T\in[0,\infty)$,
there is a constant $C=C(\eta,\ve,\L,m,T)<\infty$ with the following property.
There is an event $\O_1$ of probability exceeding $1-c^m$ on which,
for all $n\le T/c$ and all $|z|\ge1+c^{1/2-\ve}$,
\begin{equation} \label{eq:erroretal1}
|\Phi_n(z)-e^{cn}z|
\le C\bigg(c^{1/2-\ve}+\frac{c^{1-\ve}}{(e^\s-1)^2}\bigg).
\end{equation}
Moreover, in the case where $\eta=1$, provided $\ve\in(0,1/5)$ and $e^\s\ge1+c^{1/5-\ve}$,
there is also a constant $C=C(\ve,\L,m,T)<\infty$ with the following property.
There is an event $\O_1$ of probability exceeding $1-c^m$ on which,
for all $n\le T/c$ and all $|z|\ge1+c^{1/5-\ve}$,
$$
|\Phi_n(z)-e^{cn}z|
\le C \left( c^{1/2-\ve}\bigg(\frac{|z|}{|z|-1}\bigg)^{1/2} 
+ \frac{c^{1-\ve}}{(e^\s -1)^3}\right).
$$
\end{theorem} 

We remark that Theorem \ref{UFA} can be recast in terms of a regularized particle $P^{(\s)}$ given by
$$
P^{(\s)}=\{z\in D_0:e^\s z\not\in F(e^\s z)\}.
$$
Note that $P^{(\s)}$ also has capacity $c$ and is associated to the
conformal map
$$
F^{(\s)}(z)=e^{-\s}F(e^\s z).
$$
Let $(\Phi_n^{(\s)}:n\ge0)$ be an ALE process with basic map $F^{(\s)}$ and regularization parameter $0$.
Then
$$
\Phi_n^{(\s)}(z)=e^{-\s}\Phi_n(e^\s z)
$$
for an ALE process $(\Phi_n:n\ge0)$ with basic map $F$ and regularization parameter $\s$.
Hence, 
if we replace $\Phi_n$ by $\Phi_n^{(\s)}$ in Theorem \ref{UFA}, then 
under the same restrictions on $\s$,
the same estimates are valid
but now for all $|z|\ge1$ and without regularization in the density of attachment angles.

The simulations on the left side of Figure \ref{alesimulations}
illustrate the conjectured phase transition in macroscopic shape from disks to non-disks at $\eta=1$. 
The simulations on the right show the sensitivity of the fluctuations of the level lines 
$\th\mapsto\Phi_n(re^{i\th})$ in ALE(0) to taking $r-1\approx c^{1/2}$ versus $r-1\gg c^{1/2}$. 
This provides evidence that the speed at which $\s\to0$ as $c\to0$ in ALE$(\eta)$ significantly affects cluster behaviour. 
\begin{figure}[p!]
\begin{centering}
    \subfigure[$\eta=0.5$]
      {\includegraphics[width=0.31 \textwidth]{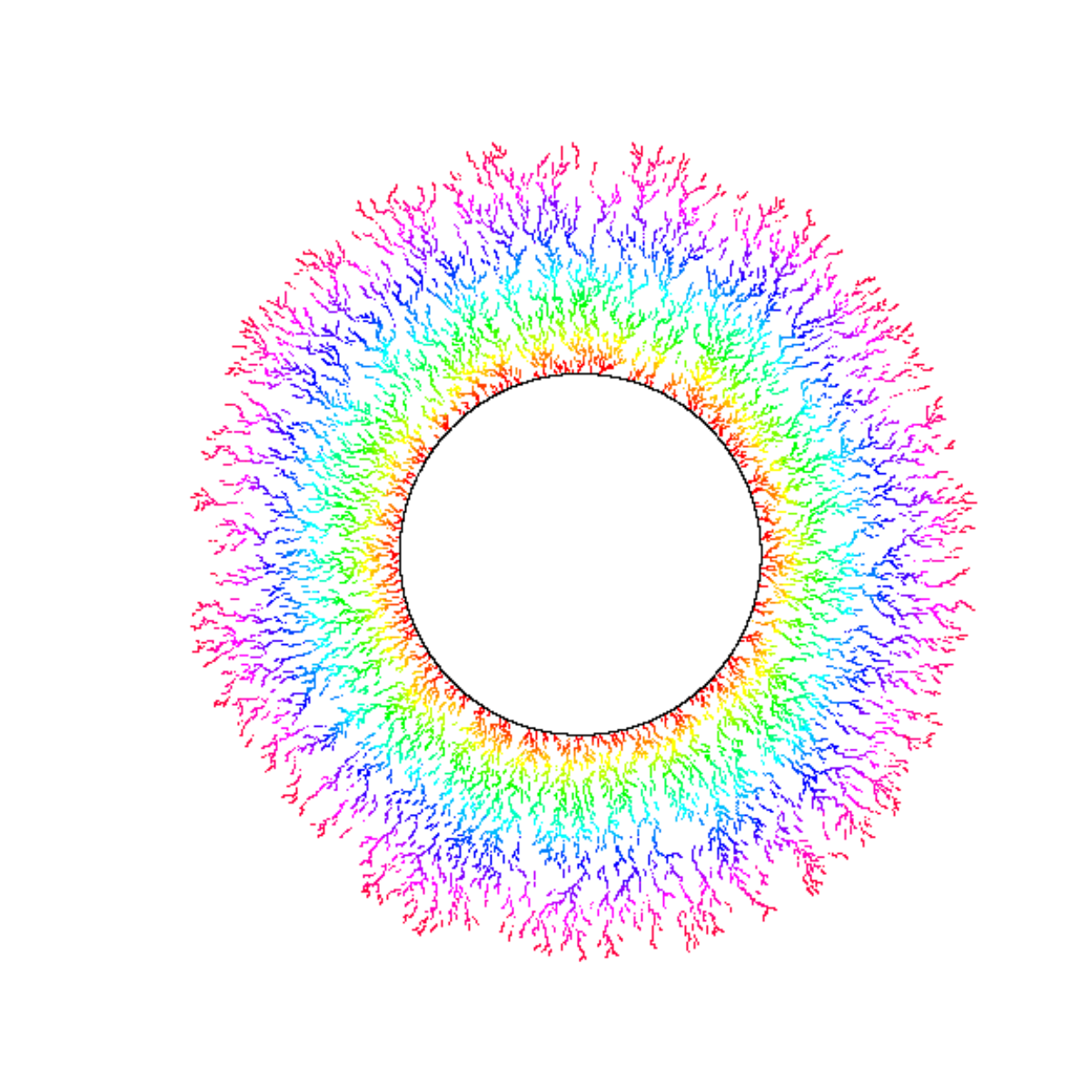}}
    \hspace{3cm}
        \subfigure[$r=1$]
      {\includegraphics[width=0.31 \textwidth]{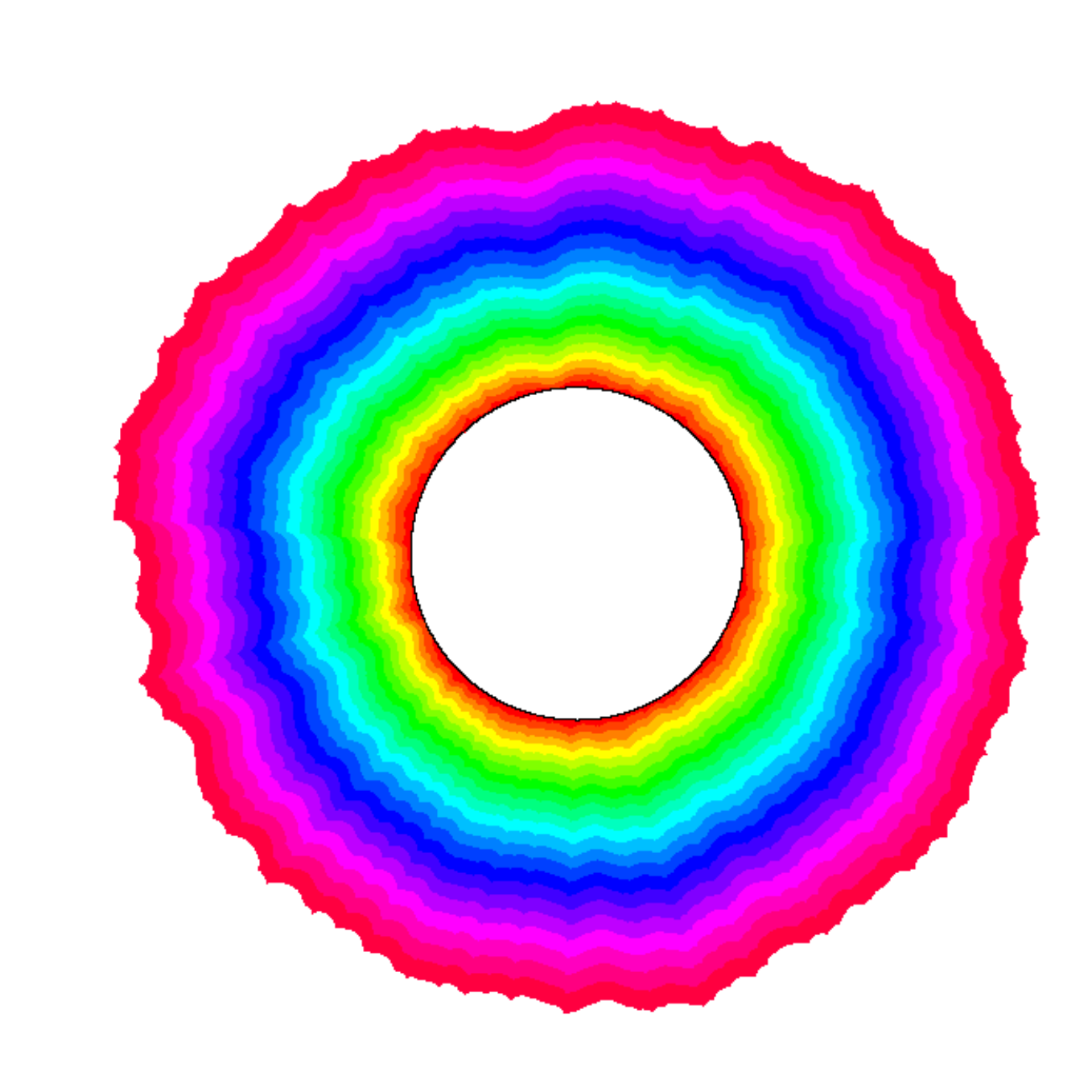}}
    \\
    \subfigure[$\eta=1$]
      {\includegraphics[width=0.31 \textwidth]{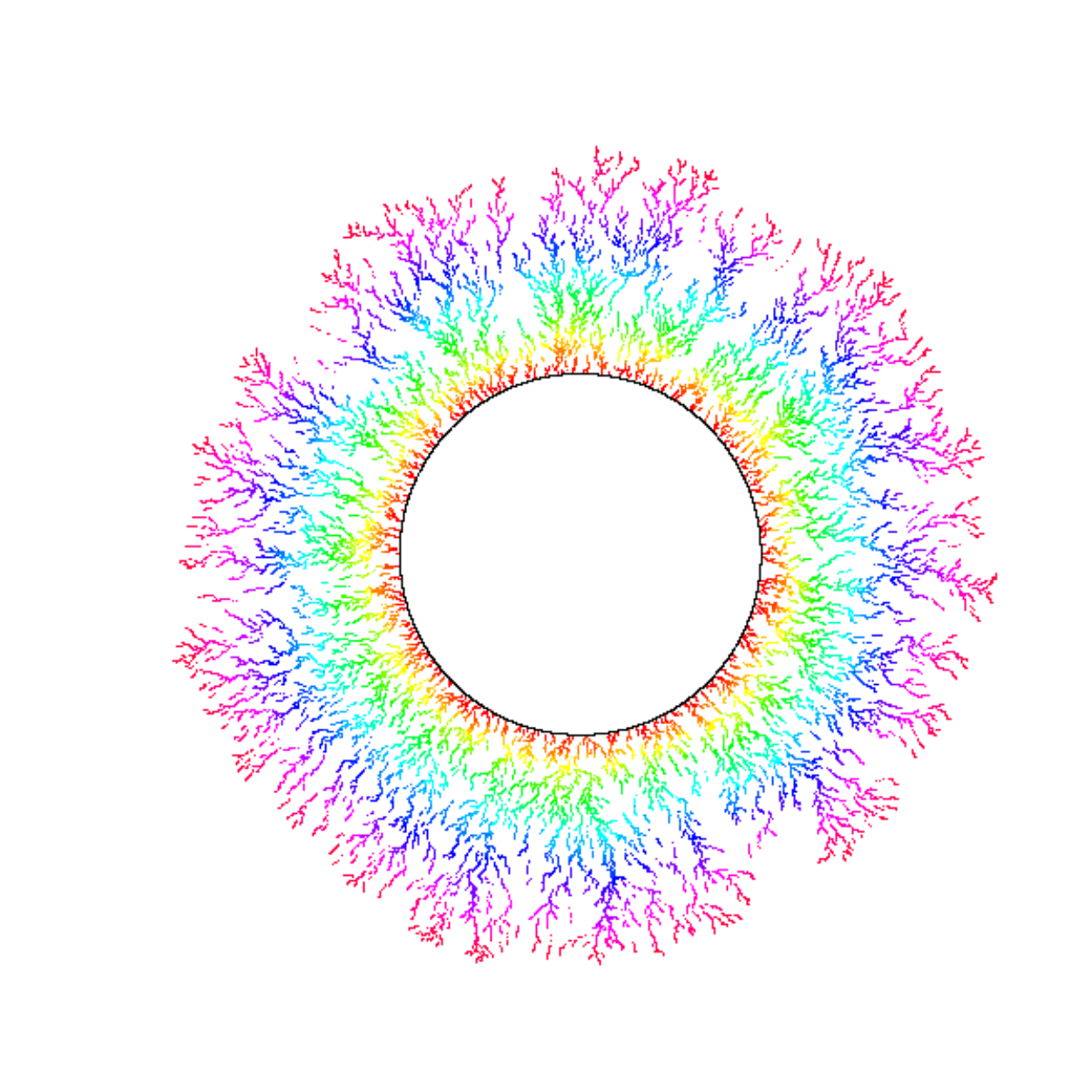}}
		\hspace{3cm}
	\subfigure[$r=1+c^{1/2}$]
      {\includegraphics[width=0.31 \textwidth]{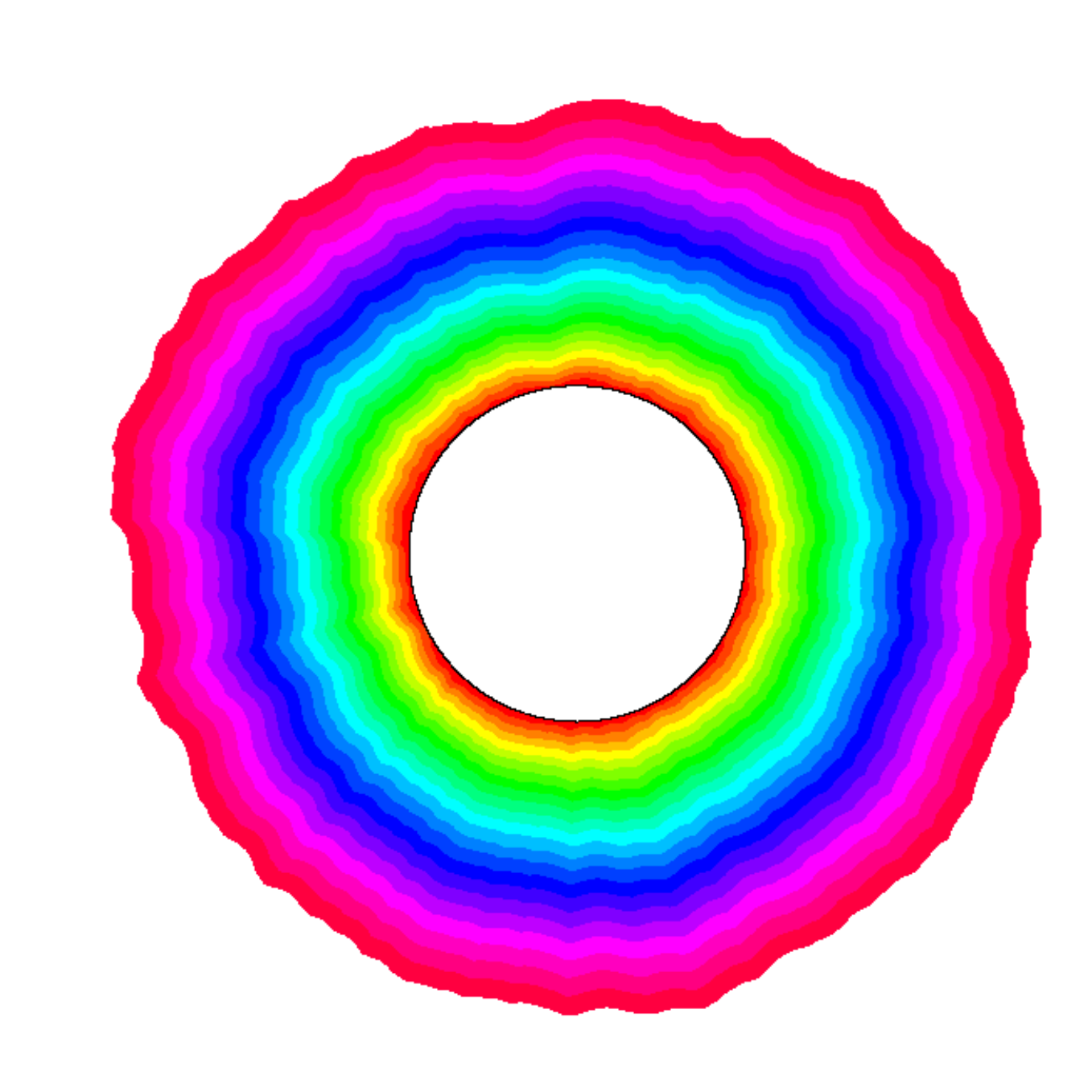}}
    \\
    \subfigure[$\eta=1.5$]
      {\includegraphics[width=0.31 \textwidth]{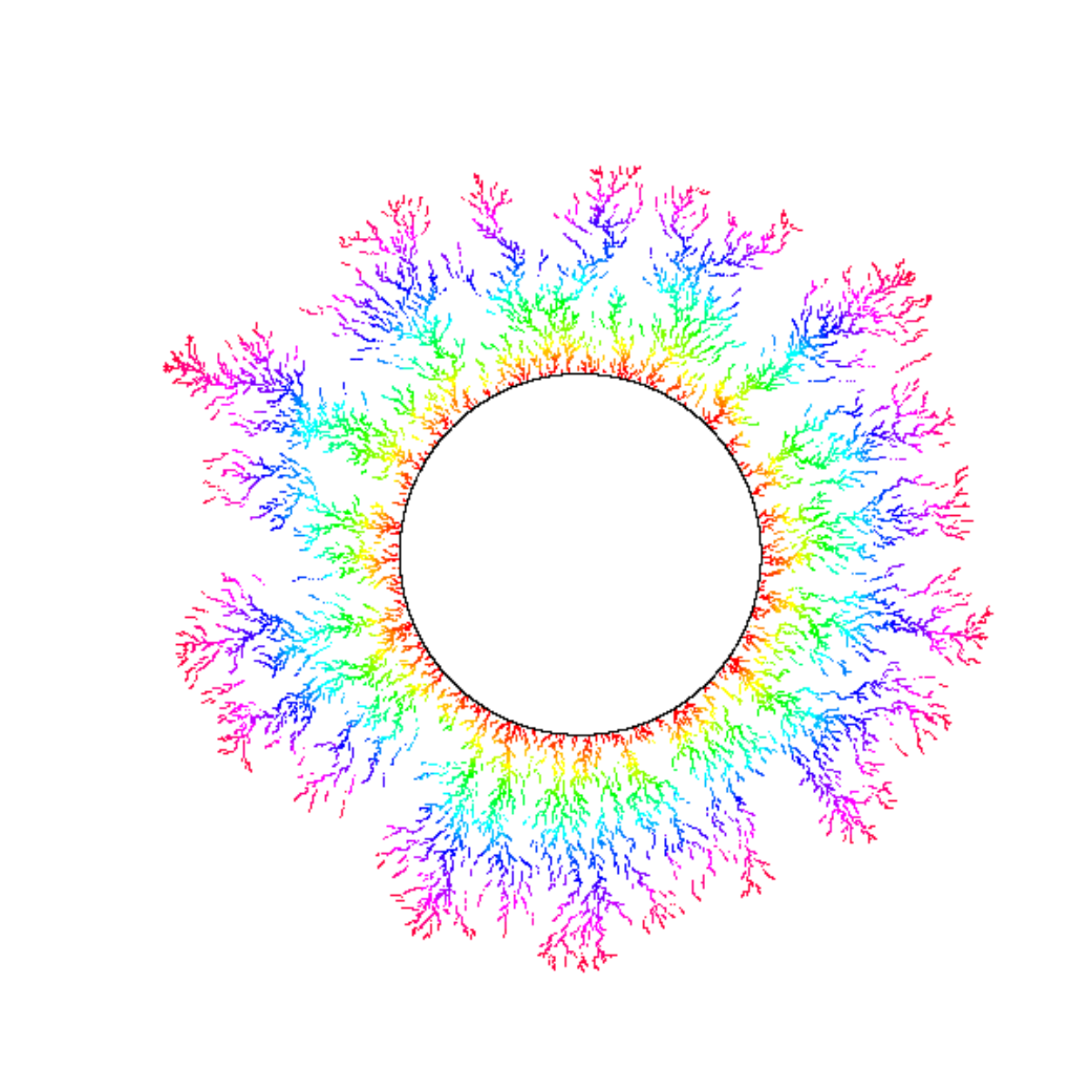}}
    \hspace{3cm}
    \subfigure[$r=1+c^{1/4}$]
      {\includegraphics[width=0.31 \textwidth]{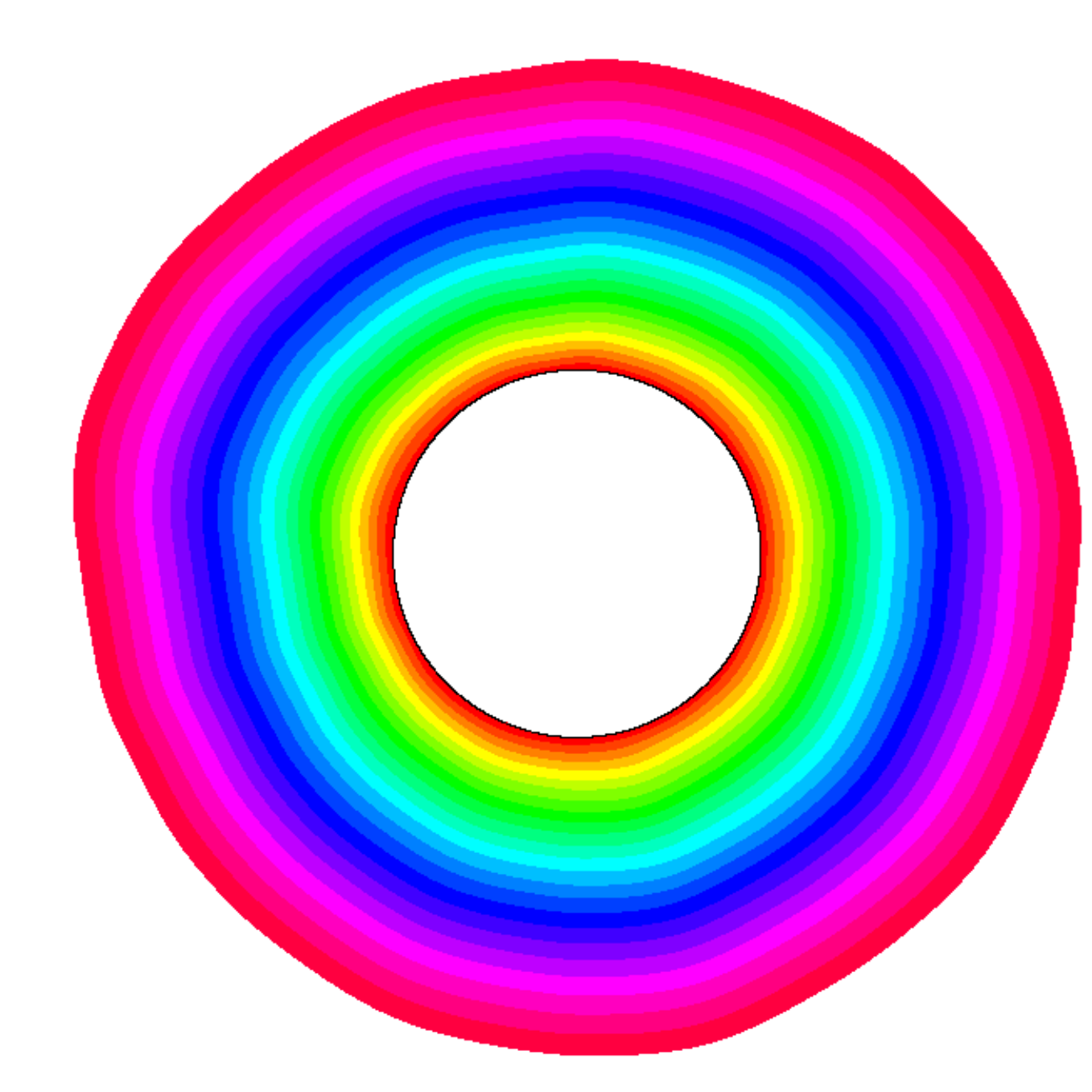}}
      
      \end{centering}
  \caption{\textsl{Left: ALE$(\eta)$ clusters with slit particles where $c=10^{-4}$, $\s=0.02$, and $n=8,000$. 
  Right: Level lines of the form $\Phi_n(r e^{i\theta})$ in an ALE(0) cluster with spread out particles (Figure 1, far right) for $c=10^{-4}$ and $n=10,000$. Colour variation is used to denote time evolution.}}
\label{alesimulations}
\end{figure}

We also establish the following characterization of the limiting fluctuations, 
which shows in particular that they are universal within the class of particles considered.

\begin{theorem} 
\label{th:fluctuations_intro}
Let $\eta\in(-\infty,1]$, $\L\in[0,\infty)$ and $\ve\in(0,1/6)$ be given.
Let $(\Phi_n:n\ge0)$ be an ALE$(\eta)$ process with basic map $F$ and regularization parameter $\s$. 
Assume that $F$ has capacity $c$ and regularity $\L$.
Assume further that 
$$
\s\ge
\begin{cases}
c^{1/4-\ve},&\text{ if $\eta\in(-\infty,1)$},\\
c^{1/6-\ve},&\text{ if $\eta=1$}.
\end{cases}
$$
Set $n(t)=\lfloor t/c\rfloor$.
Then, in the limit $c\to0$ with $\s\to0$, uniformly in $F$,
\[ 
(e^{-cn(t)}\Phi_{n(t)}(z)-z)/\sqrt c\to\cF(t,z)
\]
in distribution on $D([0,\infty),\cH)$, 
where $\cH$ is the set of analytic functions on $\{|z|>1\}$ vanishing at $\infty$, 
equipped with the metric of uniform convergence on compacts,
and where $\cF$ is given by the following stochastic PDE driven by the analytic extension $\xi$ in $D_0$ of space-time white noise on the unit circle, 
\begin{equation} 
\label{PDEfluct}
d\cF(t,z)=(1-\eta)z\cF'(t,z)dt-\cF(t,z)dt+\sqrt2d\xi(t,z).
\end{equation}
\end{theorem}
The space $\cH$ and the meaning of this PDE are discussed in more detail in Section \ref{sec:fluctuations}.
For $\eta=0$ we recover the fluctuation result in \cite{Sil}. 
The solution to the above stochastic PDE is an Ornstein--Uhlenbeck process in $\cH$.
This process converges to equilibrium as $t\to\infty$. 
When $\eta<1$, the equilibrium distribution is given by the analytic extension in $D_0$ of a log-correlated Gaussian field defined on the unit circle. 
In the case $\eta=0$, this is known as the augmented Gaussian Free Field. 
When $\eta=1$, the equilibrium distribution is the analytic extension of complex white noise on the unit circle. 
The equation \eqref{PDEfluct} can be interpreted as a family of independent equations for the Laurent coefficients of $\cF(t,.)$, given in \eqref{OU}.
These equations may be considered also for $\eta>1$ 
but now the equation for the $k$th Laurent coefficient shows exponential growth of solutions at rate $(\eta-1)k$,
so there is no solution to \eqref{PDEfluct} in $\cH$, 
indicating a destabilization of dynamics as $\eta$ passes through $1$.

Although we have stated our theorems above for $\eta\in(-\infty,1]$, in many of our arguments we restrict to the case $\eta\in[0,1]$. 
The proofs are largely similar when $\eta<0$ except in the way that we decompose the operator in Section \ref{sec:evol}. 
We remark on the correct decomposition in the case $\eta<0$ at the relevant point.

\subsection{Remarks on context and scope of results}
The process of conformal maps $(\Phi_n:n\ge0)$ is Markov and takes values in an infinite-dimensional vector space.
In the limit considered, where $c\to0$, the jumps of this process become small, while we speed up the discrete time-scale
to obtain a non-trivial limiting drift.
So we are in the domain of fluid limits for Markov processes.
The analysis of such limits, and of the renormalized fluctuations around them, is well understood in finite dimensions.
However, while the formal lines of this analysis transfer readily to infinite dimensions, its detailed implementation is not so clear,
not least because it is necessary to choose a norm, which should be well adapted to the dynamics, 
and the limiting drift will in general be a non-linear and unbounded operator.

In the case at hand, there are a number of special features which are important to the analysis.
First, while the limiting dynamics is not in equilibrium, it is an explicit steady state, which allows us to handle
convergence of the Markov process in terms of linearizations around this steady state:
we find that the difference $\Phi_n(z)-e^{cn}z$ may usefully be expressed by an interpolation in time, in which each term describes the
error introduced by a single added particle.
Second, the map $\Phi_n$ is determined by its restriction to the unit circle $(\Phi_n(e^{i\th}):\th\in[0,2\pi))$ and the action of each jump, besides 
being small, also becomes localized in $\th$ in the limit $c\to0$.
This is one of the features contributing to the explicit form found for the limiting fluctuations.
Third, we have at our disposal, not only the usual tools of stochastic analysis, but also a range of tools from complex analysis,
including distortion estimates, and $L^p$-estimates for multiplier operators, which turn out to mesh well with $L^p$-martingale inequalities.

We have tried to optimise, as far as our present techniques allow, the constraints in our results on the regularization parameter $\s$.
In the case $\eta<1$, we establish the disk limit for $\s\gg c^{1/2}$.
Indeed, for $\eta<1$, in the limit considered, we show that the derivative of the fluctuations at radius $e^\s$, 
which controls the scale of $h_n(\th)-1$, is at most of order $c^{1/2}/(e^\s-1)$.
Therefore, to leading order, the distribution of each attachment angle is approximately uniform and the bulk dynamics of our process resemble that of HL$(0)$. As seen in Proposition \ref{PEST}, the scale of individual particles is $c^{1/2}$, so for $\s\sim c^{1/2}$ the fluctuations of $e^{-c}f'(e^{\s + i\th})$ around 1 are scale-invariant. With that choice of $\sigma$ we would expect to see macroscopic variations of $h_n(\th)$, so the attachment distributions would no longer be well approximated by the uniform distribution. We therefore believe our constraint on $\s$ is close to optimal within this regime and it remains a challenging open problem to allow $\s\sim c^{1/2}$. When $\eta = 1$, on the other hand, we show that the derivative of the fluctuations at radius $e^\s$ is at most of order $c^{1/2}/(e^\s-1)^{3/2}$. The break-down in the uniform approximation may therefore well happen for larger $\s$ than $\s\sim c^{1/2}$ and the form of the fluctuations is suggestive of $\s \sim c^{1/3}$. Although we need a stronger regularization for the fluctuation result (cf.~Theorem \ref{th:fluctuations_intro}), we find that the fluctuations develop variations on all spatial scales, 
so the modification of dynamics from HL$(0)$ to ALE$(\eta)$, even with the averaging enforced by our choice of regularization,
results in a feedback which affects the limiting evolution, and which identifies the case $\eta=1$ as critical.


\subsection{Organisation of the paper}
The structure of the paper is as follows. 
In Section \ref{sec:HL0}, we give a simplified proof of convergence to a disk in the case $\eta=0$, corresponding to HL(0). 
This is followed by an overview of the proof when $\eta\neq0$. 
In Section \ref{sec:estimates}, we decompose the increment $\Phi_n(z)-\Phi_{n-1}(e^cz)$ as a sum of martingale difference and
drift terms, which we expand to leading order in $c$ with error estimates.
In Section \ref{sec:evol} we obtain the evolution equation and decomposition for the fluctuations. 
The remainder of the paper analyses this equation. 
Specifically, in Section \ref{sec:estdecom} we use the estimates from Section \ref{sec:estimates} to obtain bounds on the terms arising in the
decomposition of the differentiated fluctuations.
These bounds are then used in Section \ref{sec:diskconv} to obtain our disk limit Theorem \ref{UFA}.
Finally the fluctuation limit Theorem \ref{th:fluctuations_intro} is derived in Section \ref{sec:fluctuations}.

Some necessary but technical estimates are deferred to appendices. In Appendix \ref{sec:particle} we show that our main assumption \eqref{partcond} is satisfied for small particles of any given shape. Appendix \ref{sec:preliminaries} contains the estimates for multiplier operators used in the paper. In Appendices \ref{sec:comp_est} and \ref{sec:proof_est} we derive the specific estimates on ALE($\eta$) used in our main results.


\def\j{
\bb
\subsection{Notes on literature}
Lawler `Conformally Invariant Processes in the Plane' has a similar estimate in the unit disk but for the mapping-out function. 
I have not seen how to transfer this to $F$.
Lawler seems to have used Duren `Univalent Functions'.

\subsection{Details of capacity estimates}
Write $F_\d$ for the map associated to $P_\d$, as defined in the proof of Proposition \ref{PEST} and write
$$
\log(F_\d(z)/z)=u_\d(z)+iv_\d(z).
$$
The conformal map 
$$
f(z)=i(z-1)/(z+1)
$$ 
takes $D_0$ to the upper half-plane $\H$, 
mapping $P_\d$ to the half-disk $S_\d=\{z\in\H:|z|\le b\}$, where
$$
b=\sin\th_\d/(1+\cos\th_\d).
$$
Then the map 
$$
g(z)=\frac{z+b^2/z}{1-b^2}
$$
takes $\H\sm S_\d$ to $\H$ fixing $i$ and $\infty$.
Hence we have
$$
F_\d=f^{-1}\circ g^{-1}\circ f.
$$
Now if $e^{ia}=f^{-1}(g(b))$ then $a$ is given by
$$
a=2\int_0^{2b/(1-b^2)}\frac{dx}{1+x^2}.
$$
and $u_\d(e^{i\th})=0$ whenever $|\th|\in[a,\pi]$.
Moreover, by some straightforward estimation, we can show that 
$$
\fint_0^{2\pi}|v_\d(e^{i\th})|d\th\le C\d^2(1+\log(1/\d))
$$
for all $\d\in(0,1]$ for some constant $C<\infty$.

Here are the details.
On mapping $D_0$ to the upper half-plane $\H$ by $f(z)=i(z-1)/(z+1)$, 
the set $P_\d\cap D_0$ maps to the half-disk $S$ centred at $0$ of radius $b=\sin\th_\d/(1+\cos\th_\d)$.
Since $\sin\th_\d\le\d$ and $\cos\th_\d\ge1/2$, we have $b\le2\d/3$ (and $b\sim\d/2$ as $\d\to0$).
We map $\H\sm S$ to $\H$ by $\phi(z)=(z+b^2/z)/(1-b^2)$.
We have scaled $\phi$ so that $\phi(i)=i$.
Then
$$
F_\d=f^{-1}\circ\phi^{-1}\circ f.
$$
Set $e^{ia}=F_\d^{-1}(e^{i\th_\d})=f^{-1}(\phi(b))$.  
Then
$$
a=2\int_0^{2b/(1-b^2)}\frac{dx}{1+x^2}\le\frac{4b}{1-b^2}\le C\d
$$
(and $a\sim2\d$ as $\d\to0$).
For $\th\in[a,\pi]$ and $x\in[b,\infty)$ such that $x=\phi^{-1}(f(e^{-i\th}))$, we have 
\begin{align*}
|v_\d(e^{i\th})|
&=|\arg(F_\d(e^{i\th}))-\arg(e^{i\th})|
\le|F_\d(e^{i\th})-e^{i\th}|\\
&=|f^{-1}(x)-f^{-1}(\phi(x))|
\le|(f^{-1})'(x)||\phi(x)-x|
=\frac2{x^2+1}\frac{b^2(x+1/x)}{1-b^2}.
\end{align*}
On the other hand, we have $|v_\d(e^{i\th})|\le a$ for all $\th\in[0,a]$.
Hence
$$
\fint_0^{2\pi}|v_\d(e^{i\th})|d\th
\le\frac{4b^2}{\pi(1-b^2)}\int_b^\infty\frac{(x+1/x)dx}{(1+x^2)^2}+\frac{a^2}\pi
\le C\d^2\left(1+\log(1/\d)\right).
$$

\subsection{Stronger estimates on $\a$}
Let $P$ be as in Proposition \ref{PEST} and assume that $\d\le1$.
Then $e^{i\pi}$ is not a limit point of $P$ so $F(e^{i(\pi+\a)})=e^{i\pi}$ for some $\a\in\R$.
Then $u(e^{i(\pi+\a)})=0$ and we can and do choose $\a$ so that $\a+v(e^{i(\pi+\a)})=0$.
Set
$$
\th^+=\sup\{\th\le\pi+\a:u(e^{i\th})>0\},\q
\th^-=\inf\{\th\ge\pi+\a:u(e^{i\th})>0\}-2\pi.
$$
Then $\th^-\le\th^+$.
Define
$$
\tilde u(z)=u(e^{i\a}z),\q
\tilde v(z)=\a+v(e^{i\a}z),\q
\tilde F(z)=z\exp(u(z)+iv(z)).
$$
Then $\tilde F$ maps $D_0$ to $D_0\sm P$ fixing $\infty$ and $e^{i\pi}$, and $\tilde v(e^{i\pi})=0$.
Note that also $v_\d(e^{i\pi})=0$ by symmetry.
Since $P\sse P_\d$, harmonic measure arguments now show that
$$
|\th^\pm-\a|\le a
$$
and
$$
|\tilde v(e^{i\th})|\le|v_\d(e^{i\th})|
$$
whenever $|\th|\in[a,\pi]$.
On the other hand, for $|\th|\le a$, we have $\tilde F(e^{i\th})\in P_\d$, so $|\tilde v(e^{i\th})|\le $.
But
$$
\fint_0^{2\pi}v(e^{i\th})d\th=0
$$
so
$$
|\a|
=\left|\fint_0^{2\pi}\tilde v(e^{i\th})d\th\right|
\le\dots
\le C\d^2\left(1+\log\left(\frac1\d\right)\right).
$$
Hence, there is a constant $C<\infty$ such that $|\th^\pm|\le C\d$.
In fact, for all $\ve>0$ there is a $\d(\ve)>0$ such that $|\th^\pm|\le(2+\ve)\d$ whenever $\d\in(0,\d(\ve)]$.

\subsection{Rotation and smoothing of particles}
Consider a basic particle $P$ and associated map $F$ satisfying \eqref{PREC}.
For each $\th\in(0,2\pi)$, we can rotate to obtain a new particle $P_\th=e^{i\th}P$ whose associated map is given by
$$
F_\th(z)=e^{i\th}F(e^{-i\th}z).
$$
Then
$$
\left|\log\left(\frac{F_\th(z)}z\right)-c\frac{\g z+e^{i\th}}{\g z-e^{i\th}}\right|\le\frac{Ac^{3/2}}{|z-e^{i\th}|(|z|-1)}.
$$
It is straightforward to check that this is not compatible with \eqref{PREC}, which may be considered as identifying
$1$ as the unique point on the unit circle at which $P$ is centred.
On the other hand, given $r\in(1,\infty)$, we can obtain a new particle $P_r$, which also has capacity $c$, by setting
$$
P_r=\{z\in D_0:rz\not\in F(rD_0)\}.
$$
The map $F_r$ associated with $P_r$ is given by
$$
F_r(z)=r^{-1}F(rz)
$$
and satisfies \eqref{PREC} $\g_r=r\g$ and $\L_r=\L/r$.

\subsection{Range of $\gamma$}
The range $\g\in[1,\infty)$ in \eqref{PREC} is natural, as the following discussion shows.
Consider the possible validity of \eqref{PREC} for a particle $P$ and some general $\g\in\C$.
By rotating the particle as above, we reduce to the case $\g\in[0,\infty)$.
In the limit of interest to us later, we will have $\L\sqrt c<1$ eventually.
But then the case $\g\in[0,1)$ cannot arise. 
For if $\g=1-x$ with $x\in(0,1]$ then we can take $z=1+x$ in \eqref{partcond} and use $|F(z)|\ge|z|$ to see that $\log|F(z)/z|\ge0$ and so
$$
c\frac{2-x^2}{x^2}=-c\frac{\g z+1}{\g z-1}\le\frac{\L c^{3/2}}{|z-1|(|z|-1)}<\frac c{x^2}
$$
which is impossible.

\subsection{Notes for the estimates on capacity in Proposition \ref{PEST}}
For the lower bound, we will consider the case where $1+\d\in P$.
For $\d>0$, set
$$
q_\d
=\PP_\infty(B\text{ hits }(1,1+\d]\text{ before leaving }D_0)
$$
where $B$ is a complex Brownian motion.
Then, by a calculation of conformal maps,
$$
q_\d
=2\int_0^{\d/(2\sqrt{1+\d})}\frac{dx}{\pi(1+x^2)}.
$$
Hence $q_\d\le\d/\pi$ and $q_\d\sim\d/\pi$ as $\d\to0$, and
for $\d\le1$ we have the lower bound
$$
q_\d\ge\frac8{9\sqrt2}\frac\d\pi.
$$
By Beurling's projection theorem, this provides a lower bound of $\d/C$ on the harmonic measure
from $\infty$ carried on the boundary outside $\{|z|\ge1+\d/2\}$.
Then
$$
c=\fint_0^{2\pi}u(e^{i\th})d\th\ge\log(1+\d/2)\d/C
$$
giving the claimed lower bound. 

For the upper bound, we will consider the case where $|z-1|\le\d$ for all $z\in P$ for some $\d\in(0,1]$.
Then $\cp(P)\le\cp(\{z\in D_0:|z-1|\le\d\})\le C\d^2$,
where the final inequality comes from explicit calculation of the map $F_\d$.
In fact this shows that 
$$
\cp(P_\d)=\log\left(\frac{1+b^2}{1-b^2}\right)
$$
where $b\le2\d/3$ and so $\cp(P_\d)\le8\d^2/5$
(and $b\sim\d/2$ and so $\cp(P_\d)\sim\d^2/2$ as $\d\to0$).

\subsection{Notes on $v$ and $v_\d$}
By a harmonic measure argument
$$
\tilde v(e^{i\th})\le\tilde v(e^{i\th'})\le0.
$$ 
Since $P_\d$ is invariant under reflection in the real axis, and $F_\d$ is uniquely determined by $P_\d$, 
we have $F_\d(e^{i\pi})=e^{i\pi}$ and, with obvious notation, $\a_\d=0$.
Consider the particle $P_*=\tilde F^{-1}(P_\d\sm P)$ and let $\tilde F_*$ be the unique conformal map $D_0\to D_0\sm P_*$
such that $\tilde F_*(\infty)=\infty$ and $\tilde F_*(e^{i\pi})=e^{i\pi}$.
Then
$$
F_\d=\tilde F\circ\tilde F_*.
$$
Determine $\th_\d^+\in[\th^+,\pi]$ by $\tilde F(e^{i\th_\d^+})=e^{i\th_\d}$.
Given $\th\in[\th_\d^+,\pi]$, there is a unique $\th_*\in[\th,\pi]$ such that $\tilde F_*(e^{i\th_*})=e^{i\th}$.
Then $\th+\tilde v(e^{i\th})=\th_*+v_\d(e^{i\th_*})$.
Hence
$$
v_\d(e^{i\th})\le v_\d(e^{i\th_*})\le\tilde v(e^{i\th})\le 0
$$
and so 
$$
|\tilde v(e^{i\th})-v_\d(e^{i\th})|\le|v_\d(e^{i\th})|\le\th_\d.
$$
On the other hand, for $\th\in[0,\th_\d^+]$, we have $\tilde F(e^{i\th})\in P_\d$, so 
$$
|\tilde v(e^{i\th})-v_\d(e^{i\th})|\le2\th_\d.
$$
Analogous estimates apply for $\th\in[\pi,2\pi]$.
Hence we see that $|\tilde v(e^{i\th})-v_\d(e^{i\th})|\le2\th_\d$ for all $\th\in[0,2\pi]$.
Then
$$
|\a|
=\left|\fint_0^{2\pi}\tilde v(e^{i\th})d\th\right|
=\left|\fint_0^{2\pi}(\tilde v-v_\d)(e^{i\th})d\th\right|
\le\fint_0^{2\pi}|\tilde v-v_0|(e^{i\th})d\th
\le2\th_\d.
$$
If just going for this estimate, can simplify above.

Estimates on $v_0$. 
The same monotonicity argument used above work if we replace $P_0$ by the 
larger set 
$$
\tilde P_0=\{z\in D_0:|z-x_0|\le\rho_0\}
$$
where $\rho_0$ and $x_0$ are determined so that the convex boundary of $\tilde P_0$ cuts the unit circle orthogonally and
at the same two points $e^{\pm i\th_0}$ as does the convex boundary of $P_0$.
Then $x_0>1$ and $\rho_0>\rho$, so $P_0\sse\tilde P_0$.
From now on we work with this new set and drop the tildes.
We will obtain estimates on $v_0$ (new) by using the fact that $F_0=H^{-1}\circ S\circ H$, 
where $H$ maps $\{|z|>1\}$ to the upper half-plane $\H$, given by
$$
H(z)=i\frac{z-1}{z+1}
$$
and $S$ maps $\H$ to $\H\sm\{|z|>r\}=H(D\sm P_0)$ fixing $i$ and $\infty$, given by
$$
S^{-1}(z)=\frac{z+r^2/z}{1-r^2},\q S(z)=(1-r^2)z-\frac{r^2}{(1-r^2)z}+O(r^4).
$$
Note that 
$$
r=|H(e^{i\th})|=\frac\rho{|e^{i\th}+1|}\le\rho
$$
and 
$$
H'(z)=\frac{2i}{(z+1)^2}
$$
so $|H'(z)|\ge1/2$ when $|z|=1$.
Determine $\th_1\in[0,\pi]$ by $F_0(e^{i\th_1})=e^{i\th_0}$.
Then $\th_1=$
and
$$
\frac1{2\pi}\int_{\th_1}^\pi|v_0(e^{i\th})|d\th
\le2\int_{-\infty}^{2r/(1-r^2)}\frac{r^2}{x(1-r^2)}\frac1{\pi(1+x^2)}dx
\le Ar^2\le A\rho^2.
$$
\eb
}\jj

\section{HL(0) and overview of the proof of Theorem \ref{UFA}}\label{sec:HL0}
In this section we give a quick argument for the scaling limit of HL(0) (which is the same as ALE(0)), 
where the attachment angles $(\Theta_n:n\ge1)$ are independent and uniformly distributed. 
Then we discuss the structure of the proof of Theorem \ref{UFA}, some aspects of which follow the argument used for HL(0).  

For a measurable function $f$ on $\{|z|>1\}$, for $p\in[1,\infty)$ and $r>1$, we will write
\begin{equation} \label{prnorm}
\|f\|_{p,r}=\left(\fint_0^{2\pi}|f(re^{i\th})|^pd\th\right)^{1/p},\q
\|f\|_{\infty,r}=\sup_{\th\in[0,2\pi)}|f(re^{i\th})|.
\end{equation}
In the case where $f$ is analytic and is bounded at $\infty$, we have, for $\rho\in(1,r)$,
\begin{equation}\label{PINF}
\|f\|_{p,r}\le\|f\|_{p,\rho},\q 
\|f\|_{\infty,r}\le\left(\frac\rho{r-\rho}\right)^{1/p}\|f\|_{p,\rho}.
\end{equation}
The notation $\|\cdot\|_p$ will be reserved for the $L^p (\PP)$-norm on the probability space. 

\subsection{Disk limit for $\eta=0$} \label{sec:HL0proof}
We now show that HL($0$) converges to a disk in the small-particle limit. 
A weaker form of this result was shown in \cite{NT2} by fluid limit estimates on the Markov processes $(\Phi_n^{-1}(z):n\ge0)$.
Here, we will use a new method, based on estimating directly the conformal maps $\Phi_n$.
This both gives a simpler argument and leads to a stronger result.

\begin{theorem} 
\label{th:HL0}
Let $(\Phi_n:n\ge0)$ be an HL$(0)$ process with basic map $F$.
Assume that $F$ has capacity $c\in(0,1]$ and regularity $\L\in[0,\infty)$.
Then, for all $p\in[2,\infty)$, there is a constant $C=C(\L,p)<\infty$ such that, for all $r>1$ and $n\ge0$, we have
\[
\Big\|\sup_{|z|\ge r}|\Phi_n(z)-e^{cn}z|\Big\|_p 
\le Ce^{cn}\sqrt c\left(\frac r{r-1}\right)^{1+1/p}. 
\]
\end{theorem}
We remark that by taking $p$ large enough it is possible to deduce that, for all $\ve\in(0,1/2)$ and $T\ge0$, we have
\[ 
\sup_{n\le T/c,\,|z|\ge 1+c^{1/2-\ve}}|\Phi_n(z)-e^{cn}z|\to0 
\]
in probability as $c\to0$. 
As this is spelled out more generally in Section \ref{sec:highprob}, we omit the details at this stage. 
Indeed, on applying Theorem \ref{UFA} to HL(0), say with $\s=1$, we obtain the stronger estimate
\[ 
\sup_{n\le T/c,\,|z|\ge 1+c^{1/2-\ve}}|\Phi_n(z)-e^{cn}z|\leq C c^{1/2-\ve} 
\]
with high probability as $c\to0$.
This improvement can be traced to the iterative argument used in the proof of Proposition \ref{DFA}.

\begin{proof}[Proof of Theorem \ref{th:HL0}] 
It will suffice to consider the case where $r\ge1+\sqrt c$.
Set
\begin{equation}
\label{eq:Delta_def}
\Delta_n(\th,z)=\Phi_{n-1}(e^{i\th}F(e^{-i\th}z))-\Phi_{n-1}(e^cz),\q \Delta_n(z)=\Delta_n(\Th_n,z).
\end{equation}
Note that $\Phi_{n-1}(e^cz)$ is the map we would obtain after $n$ steps if we substituted $F_n(z)$ by $e^cz$ in \eqref{phidef}. As we aim to show that $\Phi_n(z)$ is close to $e^{cn}z$, $\Delta_n(z)$ can be understood as the error due to the $n$th particle.
We can write $\Phi_n$ as a telescoping sum
\begin{equation}
\label{HL0decomp}
\Phi_n(z)-e^{cn}z=\sum_{j=1}^n\Delta_j(e^{c(n-j)}z). 
\end{equation}
The functions $F$ and $\Phi_{j-1}$ are analytic in $\{|z|>1\}$ and $F(z)/z\to e^c$ as $z\to\infty$, so the function
$$
w\mapsto(\Phi_{j-1}(wF(z/w))-\Phi_{j-1}(e^cz))/w
$$
is analytic in $\{0<|w|<|z|\}$ and extends analytically to $\{|w|<|z|\}$.
Hence, almost surely, by Cauchy's theorem,
$$
\E(\Delta_j(z)|\cF_{j-1})
=\fint_0^{2\pi}\Delta_j(\th,z)d\th
=\frac1{2\pi i}\int_{|w|=1}(\Phi_{j-1}(wF(z/w))-\Phi_{j-1}(e^cz))\frac{dw}w=0.
$$
There is a constant $C=C(\L)<\infty$ such that, for all $|z|>1+\sqrt c/2$,
$$
|F(z)-e^cz|\le Cc\frac{|z|}{|z-1|}.
$$
Since $\Phi_{j-1}$ is univalent on $\{|z|>1\}$ and $\Phi_{j-1}(z)/z\to e^{c(j-1)}$ as $z\to\infty$,
by a standard distortion estimate,
for all $|z|=r>1$,
\begin{equation*}\label{DISTN}
|\Phi_{j-1}'(z)|\le e^{c(j-1)}\frac r{r-1}.
\end{equation*}
Hence, for $|z|=r>1+\sqrt c/2$, we have
$$
|\Delta_j(\th,z)|
\le Cce^{cj}\frac{r^2}{(r-1)|e^{-i\th}z-1|}
$$
and so
\begin{align*}
\E(|\Delta_j(z)|^2|\cF_{j-1})
&=\fint_0^{2\pi}|\Delta_j(\th,z)|^2d\th\\
&\le Cc^2e^{2cj}\left(\frac r{r-1}\right)^2\fint_0^{2\pi}\frac{r^2d\th}{|e^{-i\th}z-1|^2}
\le Cc^2e^{2cj}\left(\frac r{r-1}\right)^3.
\end{align*}
Burkholder's inequality (see Section \ref{sec:martingales}) applies to the sum of martingale differences \eqref{HL0decomp}, to give that for all $p\in[2,\infty)$ there is a constant $C=C(\Lambda , p)<\infty$, such that
$$
\|\Phi_n(z)-e^{cn}z\|_p^2
\le C\sum_{j=1}^n\|\E(|\Delta_j(e^{c(n-j)}z)|^2|\cF_{j-1})\|_{p/2}
+ C c^2 e^{2cn} \left( \frac r{r-1} \right)^4 . 
$$
Hence, for $|z|\ge1+\sqrt c/2$,
$$
\|\Phi_n(z)-e^{cn}z\|_p^2
\le Cc^2\sum_{j=1}^ne^{2cj}\left(\frac{e^{c(n-j)}r}{e^{c(n-j)}r-1}\right)^3
+ C c^2 e^{2cn} \left( \frac r{r-1} \right)^4 \le Cce^{2cn}\left(\frac r{r-1}\right)^2,
$$
where we used an integral comparison for the last inequality.
Set 
\[ 
\tilde{\Phi}_n(z)=e^{-cn}\Phi_n(z)-z.
\] 
and write $\rho=(r+1)/2$.
Then, for $|z|\ge1+\sqrt c$, we have $\rho\ge1+\sqrt c/2$, so
\[\begin{split} 
\Big\|\sup_{|z|\ge r}|\tilde\Phi_n(z)|\Big\|_p^p 
&=\E\left(\|\tilde\Phi_n\|_{\infty,r}^p\right) 
\le C\left(\frac r{r-1}\right)\E\left(\|\tilde\Phi_n\|_{p,\rho}^p\right)\\ 
&=C\left(\frac r{r-1}\right)\fint_0^{2\pi}\|\tilde\Phi_n(\rho e^{i\th})\|_p^pd\th 
\le Cc^{p/2}\left(\frac r{r-1}\right)^{p+1}
\end{split}\]
and the claimed estimate follows.
\end{proof}

\subsection{Overview of the proof of Theorem \ref{UFA}}
We now discuss how the above strategy can be adapted to the case where $\eta\in(-\infty,1]$. 
Write 
\[ 
\Phi_n(z)-e^{cn}z=\sum_{j=1}^n\Delta_j(e^{c(n-j)}z) 
\]
with $\Delta_j(z)=\Phi_j(z)-\Phi_{j-1}(e^cz)$ as in \eqref{eq:Delta_def}. 
We split $\Delta_j(z)$ as the sum of a martingale difference
\begin{equation}
\label{eq:B_def}
B_j(z)=\Delta_j(z)-\E(\Delta_j(z)|\cF_{j-1})
\end{equation}
and a drift term (which vanished in the case $\eta=0$)
\begin{equation}
\label{eq:A_def}
A_j(z)=\E(\Delta_j(z)|\cF_{j-1}). 
\end{equation}
Set $\tilde\Phi_n(z)=e^{-cn}\Phi_n(z)-z$ as above.
We start by identifying the leading term in the drift, showing that 
\begin{equation}
\label{eq:An_dec} 
A_j(z)=-c\eta e^{cj}z\tilde\Phi_{j-1}'(e^\s z)+R_j(z)
\end{equation}
where $R_j(z)$ is small provided $\|\tilde\Phi'_{j-1}\|_{\infty,e^\s}$ is sufficiently small. 
This gives the following decomposition
\[ 
\begin{split} 
\tilde\Phi_n(z)
&=e^{-c}\tilde\Phi_{n-1}(e^cz)-c\eta z\tilde\Phi_{n-1}'(e^\s z)+e^{-cn}B_n(z)+e^{-cn}R_n(z)\\
&=P\tilde\Phi_{n-1}(z)+e^{-cn}B_n(z)+e^{-cn}R_n(z)
\end{split}
\]
where $P$ is the operator which acts on analytic functions on $\{|z|>1\}$ by
\begin{equation}
\label{eq:P_def}
Pf(z)=e^{-c}f(e^cz)-c\eta zf'(e^\s z). 
\end{equation}
The reader is alerted to the fact that, while we used $P$ to denote our basic 
particle in Sections \ref{INT} and Appendix \ref{sec:particle}, in the rest of the
paper, $P$ will refer to this operator instead. 
Solving the recursion we end up with 
\begin{equation} 
\label{curlyP}
\tilde\Phi_n(z)=\sum_{j=1}^n e^{-cj}P^{n-j}B_j(z)+\sum_{j=1}^ne^{-cj}P^{n-j}R_j(z). 
\end{equation}
Note that for $\eta=0$ the operator $P$ has the simple form $Pf(z)=e^{-c}f(e^cz)$ and we recover~\eqref{HL0decomp}. 
We treat the general case $\eta\in(-\infty,1]$ by observing that $P$ acts diagonally on the Laurent coefficients, 
thus is a Fourier multiplier operator, 
which we can bound in $\|\cdot\|_{p,r}$-norm by means of the Marcinkiewicz multiplier theorem 
(see Appendix \ref{sec:operators}). 

The proof strategy for the disk theorem then goes as follows. 
For $\d=\d(c)$ small, to be specified, introduce the stopping time 
\begin{equation}
\label{eq:N0_def}
N(\d)=\min\{n\ge0:\|\tilde\Phi_n'\|_{\infty,e^\s}>\d\}.
\end{equation}
Then for all $n\le N(\d)$ the angle density $h_n$ defined in \eqref{hn} is approximately uniform. 
This, together with the multiplier theorem, 
can be used to bound both the martingale term (the first term in \eqref{curlyP}) 
and the remainder term (the second term in \eqref{curlyP}), thus leading to a bound for the map $\tilde\Phi_n$.
At this point it remains to show that we can pick $\d_0$ such that $N(\d_0)\ge\lfloor T/c\rfloor$ with high probability to conclude the proof. 
To this end, it turns out to be convenient to work instead with the differentiated dynamics 	
\[\Psi_n(z)=z\tilde\Phi_n'(z)\]
for which a decomposition similar to \eqref{curlyP} holds (see \eqref{DECOMP} below). 
We use it to show that $\|\Psi_n1_{\{n\le N_0\}}\|_{p,r}$ is small in $L^p(\PP)$ (see Proposition \ref{DFA}), where we have set $N_0 = N(\d_0)$ to ease the notation slightly. 
The analyticity of $\Psi_n$ then allows us to make this bound into a high probability statement on the supremum norm of $\Psi_n1_{\{n\le N_0\}}$, 
at the price of taking $p$ large enough (see Proposition \ref{UDFA}). 
By showing that this bound is smaller than $\d$ for all $n\le N_0$, 
we deduce that in fact we must have $N_0\ge\lfloor T/c\rfloor$, thus concluding the proof. 

\def\j{
\bb
\subsection{Distortion estimate}
Suppose that $f$ is a univalent function on $\{|z|>1\}$, having Laurent expansion
$$
f(z)=z+\sum_{k=0}^\infty a_kz^{-k}.
$$
By Gronwall's area theorem, 
$$
\sum_{k=1}^\infty k|a_k|^2\le1.
$$
So, by Cauchy--Schwarz,
$$
|f'(z)-1|=\left|\sum_{k=1}^\infty ka_kz^{-k-1}\right|\le\left(\sum_{k=1}^\infty k|z|^{-2k-2}\right)^{1/2}=\frac1{|z|^2-1}.
$$
\eb
}\jj

\subsection{Choice of state variables}
The sequence of conformal maps $(\Phi_n)_{n\ge0}$ is a Markov process. 
This allows an approach to the desired scaling limits using martingale estimates.
Above, we introduced the analytic function $\Psi_n$ on $\{|z|>1\}$ given by
$$
\Psi_n(z)=D\tilde\Phi_n(z),
$$
where we set $Df(z)= zf'(z)$ and $\tilde\Phi_n$ is the time-rescaled process of fluctuations given by
$$
\tilde\Phi_n(z)=e^{-cn}\Phi_n(z)-z.
$$
Then the process $(\Psi_n)_{n\ge0}$ is also Markov and it proves more convenient to use this as our primary state variable. In doing this, we forget the limiting values $(\Phi_n(\infty))_{n\ge0}$, so we see the clusters only up to an unknown displacement.
Otherwise, the use of $(\Psi_n)_{n\ge0}$ may be considered as a particular choice of coordinates for the sequence of clusters.
The function $\Phi_n$ has a Laurent expansion in $\{|z|>1\}$ of the form
$$
\Phi_n(z)=e^{cn}\left(z+\sum_{k=0}^\infty a_n(k)z^{-k}\right)
$$
so $\Psi_n$ has expansion 
$$
\Psi_n(z)=-\sum_{k=1}^\infty ka_n(k)z^{-k}.
$$
In the final section of the paper, we will characterise the limit distribution of the rescaled fluctuations, by analysing the Laurent coefficients. 

\section{Expansions to first order and error estimates} 
\label{sec:estimates}

In this section we identify the leading order behaviour of several quantities of interest and gather together bounds on the error terms which hold while the differentiated fluctuation process $(\tilde\Phi'_n)_{n\ge0}$ is well-behaved. Our main objective is to justify \eqref{eq:An_dec}.

Fix $\d_0\in(0,1/8]$ and consider the stopping time $N_0 = N(\delta_0)$ where $N(\delta)$ is defined in \eqref{eq:N0_def}.
Several of our estimates will be made under the assumption that $n \leq N_0$. In fact, in this section, we only use that $|\tilde\Phi'_{j}(e^{\s+i\th})|\le\d_0\le1/8$ when $j=n-1$. However, we will need this to hold for all $j \leq n-1$ in the remainder of the paper and it simplifies notation to make the assumption here. This assumption guarantees that $h_n$, defined in \eqref{hn}, can be bounded above and below by absolute constants.  Bounding very crudely,
\begin{equation}\label{h_uniform}
\frac12 \le h_{n}(\th)\le \frac32 \q \mbox{so} \q |h_{n}(\th)-1|\le \frac12.
\end{equation}
A more refined analysis shows that, for all $n \leq N_0$,
\begin{equation}
\label{eq:hn_refined}
\left | h_{n}(\th)-1+\eta\re\tilde\Phi'_{n-1}(e^{\s+i\th}) \right | \leq C \delta_0^2
\end{equation}
where $C = C(\eta)$ is a constant depending only on the value of $\eta$. As the precise computation consists of elementary manipulations, it is deferred to Appendix \ref{sec:comp_est} (see \eqref{HN} and \eqref{eq:e5_est}).

Recall the definitions of $\Delta_n(\th,z)$ and $\Delta_n(z)$ from \eqref{eq:Delta_def} and the definitions of $A_n(z)$ and $B_n(z)$ from \eqref{eq:A_def} and \eqref{eq:B_def}. Then  
$$
A_n(z)=\fint_0^{2\pi}\Delta_n(\th,z)h_{n}(\th)d\th.
$$
Furthermore,
$A_n$ and $B_n$ are analytic in $\{|z|>1\}$ and, almost surely,
$$
\E(B_n(z)|\cF_{n-1})=0.
$$
As we showed in the proof of Theorem \ref{th:HL0}, by Cauchy's theorem,
$$
\int_0^{2\pi}\Delta_n(\th,z)d\th=0
$$
so
\begin{equation}
\label{eq:An_h1}
A_n(z)=\fint_0^{2\pi}\Delta_n(\th,z)(h_{n}(\th)-1)d\th.
\end{equation}

We now identify the leading order terms in $\Delta_n(z)$ and $A_n(z)$, in the limit $c\to0$. Where the computations add little to the intuition, these are also deferred to Appendix \ref{sec:comp_est}.

Given $\th\in[0,2\pi)$ and $|z|>1$, define, for $s\in[0,1]$,
\begin{equation}\label{Gdef}
\begin{split} 
F_s(z)
&=z\exp\left((1-s)c+s\log\frac{F(z)}z\right),\\
F_{s,\th}(z)
&=e^{i\th}F_s(e^{-i\th}z)=z\exp\left((1-s)c+s\log\frac{F(e^{-i\th}z)}{e^{-i\th}z}\right).
\end{split} 
\end{equation}
Note that $F_{0,\th}(z)=e^cz$ and $F_{1,\th}(z)=e^{i\th}F(e^{-i\th}z)$.
Note also that $|F_{s,\th}(z)|\ge|z|$ for all $s\in[0,1]$ and
$$
\frac d{ds}\log F_{s,\th}(z)
=\log\frac{F(e^{-i\th}z)}{e^{-i\th}z}-c.
$$
Then
\begin{align}
\notag
\Delta_n(\th,z)
&=\Phi_{n-1}(e^{i\th}F(e^{-i\th}z))-\Phi_{n-1}(e^cz) \\ 
\notag
&=\int_0^1D\Phi_{n-1}(F_{s,\th}(z))\,\frac{d}{ds}\log F_{s,\th}(z)ds\\
\notag
&=\frac{2c\b e^{cn}z}{e^{-i\th}z-1}
+\left(\log\frac{F(e^{-i\th}z)}{e^{-i\th}z}-c\right)
\int_0^1(D\Phi_{n-1}(F_{s,\th}(z))-e^{cn}z)ds\\ 
&\q\q\q\q\q\q\q\q\q\q\qquad\qquad+e^{cn}z\left(\log\frac{F(e^{-i\th}z)}{e^{-i\th }z}-c-\frac{2c\b}{e^{-i\theta}z-1}\right),
\label{PFMB}
\end{align}
where $\b$ is defined in Proposition \ref{PPC} in the appendix.
It will be convenient to set
\begin{equation}
\label{eq:mn_def}
m_n(\th,z)
=\frac{2c\b e^{cn}z}{ze^{-i\th}-1}
\end{equation}
and
\begin{equation}
\label{eq:wn_def}
w_n(\th,z) = \Delta_n(\th,z) - m_n(\th,z).
\end{equation}
Note that $w_n(\th,\infty)=0$ and for all $|z|\ge1+\sqrt c$ 
\begin{equation}\label{eq:wn_bound}
|w_n(\th,z)|\le\frac{Cce^{cn}}{|e^{-i\th}z-1|}\int_0^1|\Psi_{n-1}(F_{s,\th}(z))|ds+ \frac{C e^{cn}c^{3/2}|z|}{|e^{-i\th}z-1|(|z|-1)}. 
\end{equation}
for some constant $C=C(\eta,\L)<\infty$ (see \eqref{E7B} and \eqref{E6B}).

Using \eqref{eq:An_h1}, \eqref{eq:hn_refined}, \eqref{PFMB} and that $|\beta - 1| \leq \Lambda \sqrt{c}/2$ (cf.~Proposition \ref{PPC}), the leading term of $A_n(z)$ is
\begin{equation}
\label{eq:Ln_def}
L_n(z) = -\fint_0^{2\pi}\eta\re\tilde\Phi'_{n-1}(e^{\s+i\th})\frac{2c e^{cn}z}{ze^{-i\th}-1}d\th
=-c\eta e^{cn}z\tilde\Phi'_{n-1}(e^\s z),
\end{equation}
where the equality follows by Cauchy's integral formula. To be precise, set
\begin{equation}
\label{eq:Rn_def}
R_n(z) = A_n(z) - L_n(z). 
\end{equation}
Then, by the argument in Appendix \ref{sec:comp_est}, for $n \leq N_0$ and $|z|=r$ with $r\ge1+\sqrt c$,
\begin{align}
\label{eq:Rn_est}
\notag |R_n(z)-R_n(\infty)|
&\le 
\frac{Cce^{cn}\d_0}r\left(1+\log\left(\frac r{r-1}\right)\right)\left(\d_0+\sqrt c\left(\frac r{r-1}\right)\right)\\
&\q\q
+Cc^{3/2}c^{cn}|\Psi_{n-1}(e^\s z)|
+Cce^{cn}\d_0\int_0^1\fint_0^{2\pi}\frac{|\Psi_{n-1}(F_{s,\th}(z))|}{|ze^{-i\th}-1|}d\th ds
\end{align}
and
$$
|R_n(\infty)|\le Cce^{cn}\d_0^2
$$
for some constant $C=C(\eta,\Lambda)<\infty$ (possibly different to the constant $C$ obtained earlier). By the maximum principle, it follows that provided one takes $\delta_0 \geq \sqrt{c}/(e^\s - 1)$ and $r \geq e^\s \geq 1 + \sqrt{c}$, 
\[
|R_n(z)| \leq Cce^{cn} \delta_0^2 \log\left(\frac r{r-1}\right).
\]
From this bound, it can be easily seen that $R_n(z)$ is small if $\|\tilde\Phi'_{n-1}\|_{\infty,e^\s}$ is sufficiently small, which is what we wanted to show. However, the assumption that $r \geq e^\s$ is too restrictive for our needs, so in subsequent analysis we revert to the more general estimate \eqref{eq:Rn_est}.

\section{Linear evolution equation for the fluctuations}
\label{sec:evol}

In this section, our objective is to justify the expansion \eqref{curlyP}. In fact, we obtain an analogous expansion which makes it clearer which terms determine the leading order fluctuations.

In Section \ref{UFA} we decomposed $\Delta_n(z)=\Phi_n(z)-\Phi_{n-1}(e^cz)$ as a sum of a martingale difference $B_n(z)$ and drift $A_n(z)$, and in the previous section we justified writing
$$
A_n(z) = L_n(z)+R_n(z).
$$
In view of \eqref{eq:wn_def}, it is convenient to split the martingale difference $B_n$ as a sum of analytic functions
$$
B_n(z)=M_n(z)+W_n(z)
$$
where
$$
M_n(z)=m_n(\Th_n,z)-\fint_0^{2\pi}m_n(\th,z)h_n(\th)d\theta
$$
and
$$
W_n(z)=w_n(\Th_n,z)-\fint_0^{2\pi}w_n(\th,z)h_n(\th)d\theta.
$$
We will see that $M_n$ is the main term: its explicit form allows for precise estimates, 
and it determines the Gaussian fluctuations. 
On the other hand, $W_n$ is accessible less directly, but is of smaller order, so can also be handled adequately.
Then, using \eqref{eq:Rn_def},
\begin{equation*} 
\Phi_n(z)=\Phi_{n-1}(e^cz)+M_n(z)+L_n(z)+W_n(z)+R_n(z)
\end{equation*}
so we obtain the linear evolution equation
\begin{equation} 
\label{TIPH}
\tilde\Phi_n(z)=P\tilde\Phi_{n-1}(z)+e^{-cn}M_n(z)+e^{-cn}W_n(z)+e^{-cn}R_n(z) 
\end{equation}
where $P$ is as in \eqref{eq:P_def}.
Note that $P$ acts diagonally on the Laurent coefficients, with multipliers
$$
p(k)=e^{-c(k+1)}+c\eta ke^{-\s(k+1)},\q k\ge0.
$$
In the case $\eta\in[0,1]$, we factorize $P$ by writing
\begin{equation} \label{eq:p0}
p(k)=e^{-c}e^{-c(1-\eta)k}p_0(k).
\end{equation}
It is straightforward to check then that, for all $k$,
\begin{equation}\label{POE}
0\le p_0(k+1)\le p_0(k)\le 1.
\end{equation}
In order to adapt our argument to the case $\eta\in(-\infty,0)$, 
we would modify the equation defining $p_0(k)$ to
$$
p(k)=e^{-c(k+1)}p_0(k).
$$
The subsequent argument is very similar so we will not give further details for this case.

Write $P_0$ for the multiplier operator acting on analytic functions on $\{|z|>1\}$ by
$$
\widehat{P_0f}(k)=p_0(k)\hat f(k)
$$
and note that
$$
Pf(z)=e^{-c}P_0f(e^{c(1-\eta)}z).
$$
We iterate \eqref{TIPH} to obtain
\begin{equation}
\label{barDECOMP}
\tilde\Phi_n(z)=\tilde\cM_n(z)+\tilde\cW_n(z)+\tilde\cR_n(z)
\end{equation}
where
\begin{equation*} 
\begin{split} 
\tilde\cM_n(z)&
=\sum_{j=1}^ne^{-cj}P^{n-j}M_j(z)
=e^{-cn}\sum_{j=1}^nP_0^{n-j}M_j(e^{c(1-\eta)(n-j)}z),\\
\tilde\cW_n(z)&=e^{-cn}\sum_{j=1}^nP_0^{n-j}W_j(e^{c(1-\eta)(n-j)}z),\\
\tilde\cR_n(z)&=e^{-cn}\sum_{j=1}^nP_0^{n-j}R_j(e^{c(1-\eta)(n-j)}z).
\end{split}
\end{equation*}
Then, on differentiating,
\begin{equation}\label{DECOMP}
\Psi_n(z)=\cM_n(z)+\cW_n(z)+\cR_n(z)
\end{equation}
where
\begin{align*}
\cM_n(z)=e^{-cn}\sum_{j=1}^nP_0^{n-j}DM_j(e^{c(1-\eta)(n-j)}z),\\
\cW_n(z)=e^{-cn}\sum_{j=1}^nP_0^{n-j}DW_j(e^{c(1-\eta)(n-j)}z),\\
\cR_n(z)=e^{-cn}\sum_{j=1}^nP_0^{n-j}DR_j(e^{c(1-\eta)(n-j)}z).
\end{align*}
We will focus initially on bounding the terms in the decomposition \eqref{DECOMP} of the differentiated fluctuations $\Psi_n$.
We  will refer to $\cM_n$, $\cW_n$ and $\cR_n$ as the principal martingale term, 
the second martingale term and the remainder term respectively. 
Later, we will return also to the undifferentiated decomposition \eqref{barDECOMP}.

\subsection{Norms}
We conclude this section by describing the normed spaces on which we will obtain our bounds.

Recall from \eqref{prnorm} the definition of $\|f\|_{p,r}$ for a measurable function $f$ on $\{|z|>1\}$.
For a random such function $\Phi$, we will write
$$
\tn\Phi\tn_{p,r}=\left(\E\fint_0^{2\pi}|\Phi(re^{i\th})|^pd\th\right)^{1/p}.
$$
Thus
\begin{equation*} \label{norms}
\tn\Phi\tn_{p,r}=\|\|\Phi\|_{p,r}\|_p = 
\left( \fint_0^{2\pi} \| \Phi (re^{i\th}) \|_p^p d\th \right)^{1/p} 
\end{equation*}
where $\|\cdot\|_p$ denotes the $L^p(\PP)$-norm on the probability space.

Note that, for all $n\ge0$, 
the boundedness and monotonicity seen in \eqref{POE} allows an application of the Marcinkiewicz multiplier theorem (see Appendix \ref{sec:operators}),
with $m_k=p_0(k)^n$ and $M=1$ to see that for all $p\in(1,\infty)$ and all $r>1$, 
there is a constant $C=C(p)<\infty$ such that
\begin{equation}\label{PIE}
\|P_0^nf\|_{p,r}\le C\|f\|_{p,r}.
\end{equation}
Some further operator estimates which will be used in the subsequent analysis are stated in Appendix \ref{sec:operators}.

\section{Estimation of terms in the decomposition of the differentiated fluctuations}
\label{sec:estdecom}
In this section we collect estimates for the principal martingale term, the second martingale term and remainder term. 

We first estimate the principal martingale term $\cM_n(z)$ in the decomposition \eqref{DECOMP} of the differentiated fluctuation process, which is given by
$$
\cM_n(z)=e^{-cn}\sum_{j=1}^nP_0^{n-j}DM_j(e^{c(1-\eta)(n-j)}z).
$$

\begin{lemma}\label{lemma:mbound}
For all $p\in[2,\infty)$, there is a constant $C=C(p)<\infty$ such that
\begin{equation}
\label{MNEST}
 \tn\cM_n1_{\{n\le N_0\}}\tn_{p,r}^2 
\le C\left ( \frac{c^{2-2/p}r^{2-2/p}}{(r-1)^{4-2/p}} + c^2\sum_{j=1}^ne^{-2c(n-j)}\frac{r_{n-j}}{(r_{n-j}-1)^3} \right ), 
\end{equation}
where $r_n=re^{c(1-\eta)n}$. 

It follows that if $r \geq 1 + c^{1/2 - \ve}$ for some $\ve \in (0, 1/2)$,
\begin{equation*}
\label{MPA}
 \tn\cM_n1_{\{n\le N_0\}}\tn_{p,r} 
\le 
\begin{dcases}
\frac{C\sqrt c}{r-1}, &\q \eta < 1; \\
\frac{C\sqrt c}r\left(\frac r{r-1}\right)^{3/2}, &\q \eta=1.
\end{dcases}
\end{equation*}
\end{lemma}
\begin{proof}
By Burkholder's inequality (cf.\ Theorem \ref{thB}), for all $p\in[2,\infty)$, there is a constant $C=C(p)<\infty$ such that
\begin{align*}
&\|\cM_n(z)1_{\{n\le N_0\}}\|_p^2 \\
&\q \q \q \le Ce^{-2cn}\left ( \| \max_{1 \leq j \leq n} X_{j,n}(e^{c(1-\eta)(n-j)}z)1_{\{j\le N_0\}}\|_p^2 + \sum_{j=1}^n\|\tilde Q_{j,n}(e^{c(1-\eta)(n-j)}z)1_{\{j\le N_0\}}\|_{p/2} \right ),
\end{align*}
where
$$
X_{j,n}(z)
=|P_0^{n-j}DM_j(e^{c(1-\eta)(n-j)}z)| \quad \mbox{ and } \quad \tilde Q_{j,n}(z) = \E(|P_0^{n-j}DM_j(z)|^2|\cF_{j-1}).
$$
So, on taking the $\|\cdot\|_{p/2,r}$-norm,
\begin{align}
\label{MEE2}
\tn\cM_n1_{\{n\le N_0\}}\tn_{p,r}^2
\le Ce^{-2cn}\left (  \tn \max_{1 \leq j \leq n } X_{j,n}1_{\{j\le N_0\}}\tn_{p,r}^2 + \sum_{j=1}^n
\tn \tilde Q_{j,n}1_{\{j\le N_0\}}\tn_{p/2,r_{n-j}} \right ).
\end{align}
Recall from \eqref{eq:mn_def} that
$$
m_j(\th,z)
=\frac{2c\b e^{cj}z}{ze^{-i\th}-1}
=2c\b e^{cj}\sum_{k=0}^\infty z^{-k}e^{i\th(k+1)}.
$$
Observe that
$$
P_0^{n-j}Dm_j(\th,z)
=2c\b e^{cj}\sum_{k=0}^\infty p_0(k)^{n-j}(-k) z^{-k}e^{i\th(k+1)}.
$$
Hence, almost surely,
$$
\E(|P_0^{n-j}DM_j(z)|^2|\cF_{j-1})
\le\E(|P_0^{n-j}Dm_j(\Theta_j , z)|^2|\cF_{j-1})
=\fint_0^{2\pi}|P_0^{n-j}Dm_j(\th,z)|^2h_{j}(\th)d\th
$$
and, for $|z|=r$,
$$
\fint_0^{2\pi}|P_0^{n-j}Dm_j(\th,z)|^2d\th
=4c^2|\b|^2e^{2cj}\sum_{k=0}^\infty p_0(k)^{2(n-j)}k^2r^{-2k}
\le4c^2|\b|^2e^{2cj}\sum_{k=0}^\infty k^2r^{-2k}.
$$
For $j\le N_0$, we have $h_{j}(\th)\le3/2$, so we obtain, for $|z|=r$, almost surely, 
$$
\tilde Q_{j,n}(z)
\le 6c^2|\b|^2e^{2cj}\frac r{(r-1)^3}
$$
where we have used
$$
\sum_{k=0}^\infty k^2r^{-2k}
=\frac{r^2(r^2+1)}{(r-1)^3(r+1)^3}
\le\frac r{(r-1)^3}.
$$
Hence, for $|z|=r$, almost surely,
$$ 
\sum_{j=1}^n\tilde Q_{j,n}(e^{c(1-\eta)(n-j)}z)1_{\{j\le N_0\}} 
\le Cc^2|\b|^2\sum_{j=1}^ne^{2cj}\frac{r_{n-j}}{(r_{n-j}-1)^3}.
$$
Moreover, 
\begin{align*}
X_{j,n}(z) 
&= \left | P_0^{n-j}Dm_j(\Theta_j,e^{c(1-\eta)(n-j)}z)-\E(P_0^{n-j}Dm_j(\Theta_j,e^{c(1-\eta)(n-j)}z)|\cF_{j-1}) \right | \\
&\leq \left | P_0^{n-j}Dm_j(\Theta_j,e^{c(1-\eta)(n-j)}z) \right |+ \E(|P_0^{n-j}Dm_j(\Theta_j,e^{c(1-\eta)(n-j)}z)|^2 |\cF_{j-1})^{1/2},
\end{align*}
and
\begin{align*}
\left \| P_0^{n-j}Dm_j(\Theta_j,z) \right \|_{p,r}^{p}
&= \fint_0^{2 \pi} | P_0^{n-j}Dm_j(\Theta_j,re^{i \th}) |^p d\th 
= \fint_0^{2 \pi} | P_0^{n-j}Dm_j(0,re^{i (\th-\Theta_j)}) |^p d\th \\
&=\left \| P_0^{n-j}Dm_j(0,z) \right \|_{p,r}^{p} 
\leq C \left \| Dm_j(0,z) \right \|_{p,r}^{p} 
\leq \frac{Cc^p |\beta|^p e^{pcj}r^{p-1}}{(r-1)^{2p-1}}.
\end{align*}
Hence, 
\begin{align*}
\| \max_{1\leq j \leq n} X_{j,n}1_{\{j\le N_0\}}\|_{p,r}^p 
& \leq \sum_{j=1}^n \| X_{j,n}1_{\{j\le N_0\}}\|_{p,r}^p 
\leq \frac{Cc^{p-1} |\beta|^p e^{pcn}r^{p-1}}{(r-1)^{2p-1}}.
\end{align*}
Finally, we take the $\|.\|_{p/2,r}$-norm and substitute into \eqref{MEE2} to obtain \eqref{MNEST}.

Now suppose $r\ge1+c^{1/2 - \ve}$ for some $\ve \in (0,2)$ and $p \geq 1 + 1/(2 \ve)$.
If $\eta<1$, by using an integral comparison in \eqref{MNEST} we obtain
\begin{align}
\notag \tn\cM_n1_{\{n\le N_0\}}\tn_{p,r}^2
&\le Cc^2\bigg(\frac r{(r-1)^3}+\sum_{j=1}^{n-1}e^{-2cj}\frac{r_j}{(r_j-1)^3}\bigg)
 +   \frac{C c^{2-2/p}r^{2-2/p}}{(r-1)^{4-2/p}}   \\
\label{MPA2}
&\le Cc\left(\frac 1{(r-1)^2}+\int_0^\infty\frac{cre^{c(1-\eta)\t}}{(re^{c(1-\eta)\t}-1)^3}d\t\right)
=\frac{C c}{(r-1)^2}
\end{align}
where we used the assumption on $p$ in the second inequality, and absorbed a factor of $2+1/(2-2\eta)$ in the final constant $C$.
Hence
\begin{equation*}
 \tn\cM_n1_{\{n\le N_0\}}\tn_{p,r} 
\le \frac{C\sqrt c}{r-1}. 
\end{equation*}

If $\eta=1$, we now have $r_n = r$, so
\[ \begin{split} 
\tn\cM_n1_{\{n\le N_0\}}\tn_{p,r}^2
& \le Cc^2\sum_{j=0}^{n-1}e^{-2cj}\frac r{(r-1)^3}
+ \frac{C c^{2-2/p}}{r^2} \left( \frac r{r-1} \right)^{4-2/p} 
\\ & \le\frac{Cc}{r^2} \left( \frac{r}{r-1} \right)^3 \bigg( 1 + c^{1-2/p} \left( \frac r{r-1} \right)^{1-2/p} \bigg) 
\end{split}\] 
and then, using that $p\geq 2$, 
$$
\tn\cM_n1_{\{n\le N_0\}}\tn_{p,r}
\le\frac{C\sqrt c}r\left(\frac r{r-1}\right)^{3/2}.
$$

\end{proof}

We now state the estimate of the second martingale term $\cW_n(z)$ in the decomposition \eqref{DECOMP} of the differentiated fluctuation process, which is given by
$$
\cW_n(z)=e^{-cn}\sum_{j=1}^nP_0^{n-j}DW_j(e^{c(1-\eta)(n-j)}z).
$$
The proof is deferred to Appendix \ref{sec:proof_W_est}.

\begin{lemma}
\label{lemma:wbound}
For all $\ve \in (0, 1/2)$ and $p\in[2,\infty)$, there is a constant $C=C(\Lambda, \eta, \epsilon, p)<\infty$ such that, for all $r\ge1+ 2c^{1/2 - \ve}$, 
\begin{align}
\label{MNESTEG}
\nonumber \tn\cW_n1_{\{n\le N_0\}}\tn_{p,r}^2
&\le Cc^2\sum_{j=1}^n e^{-2c(n-j)}
\frac{r_{n-j}}{(r_{n-j}-1)^3}\left(\tn\Psi_{j-1}1_{\{j\le N_0\}}\tn_{p,\rho_{n-j}}^2
+c\left(\frac r{r-1}\right)^2\right) \\
& \ + \frac{Ce^{2c(n-1)}c^{3-2/p}r^{4-2/p}}{(r-1)^{6-2/p}} + \frac{Ce^{2c(n-1)}c^{2-2/p}r^{2-2/p}}{(r-1)^{4-2/p}}\max_{1\leq j \leq n}\tn\Psi_{j-1}1_{\{j\le N_0\}}\tn_{p,\rho_{n-j}}^2,
\end{align}
where $r_n=re^{c(1-\eta)n}$ and $\rho_n = (1 + r_n)/2$.

It follows that, setting $\rho = (1+r)/2$, for $p \geq 1 + 1/(2 \ve)$,
$$
\tn\cW_n1_{\{n\le N_0\}}\tn_{p,r}
\le \begin{dcases}
\frac{C\sqrt c}{r-1}
\bigg(
\sup_{j\le n}\tn\Psi_{j-1}1_{\{j\le N_0\}}\tn_{p,\rho}
+\sqrt c\left(\frac r{r-1}\right) 
\bigg), &\q \eta < 1 \\
\frac{C\sqrt c}r \left(\frac r{r-1}\right)^{3/2} 
\bigg( \sup_{j\le n}\tn\Psi_{j-1}1_{\{j\le N_0\}}\tn_{p,\rho}
+\sqrt{c} \left(\frac r{r-1}\right) \bigg) &\q \eta=1.
\end{dcases}
$$
\end{lemma}

We finish this section with the estimate of the remainder term $\cR_n(z)$ in the decomposition \eqref{DECOMP} of the differentiated fluctuation process, which is given by
$$
\cR_n(z)=e^{-cn}\sum_{j=1}^nP_0^{n-j}DR_j(e^{c(1-\eta)(n-j)}z).
$$
The proof is deferred to Appendix \ref{sec:proof_R_est}.

\begin{lemma} \label{lemma:rnbound}
For all $p\in[2,\infty)$, there is a constant $C=C(\Lambda, \eta, p)<\infty$ such that, for all $r\ge1+ 2\sqrt{c}$, 
\begin{align}
\notag
&\tn\cR_n1_{\{n\le N_0\}}\tn_{p,r} \\
\notag
&\le Cc
\sum_{j=1}^n
\frac{e^{-c(n-j)}\d_0}{r_{n-j}-1}
\left(\d_0+\tn\Psi_{j-1}1_{\{j\le N_0\}}\tn_{p,\rho_{n-j}}
+\sqrt c\left(\frac{r_{n-j}}{r_{n-j}-1}\right)\right)
\left(1+\log\left(\frac{r_{n-j}}{r_{n-j}-1}\right)\right)\\
\label{RNEST}
&\q\q+Cc^{3/2}
\sum_{j=1}^n
e^{-c(n-j)}
\tn\Psi_{j-1}1_{\{j\le N_0\}}\tn_{p,\rho_{n-j}}
\left(\frac{r_{n-j}}{r_{n-j}-1}\right),
\end{align}
where we have used the same notation as in Lemma \ref{lemma:wbound}.

Now suppose that $n \leq T/c$ for some constant $T>0$. Then there is a constant $C=C(\Lambda, \eta, p, T)<\infty$ such that
\begin{align*}
\tn\cR_n1_{\{n\le N_0\}}\tn_{p,r}
&\le\frac{C\d_0}r\left(\d_0+\sup_{j\le n}\tn\Psi_{j-1}1_{\{j\le N_0\}}\tn_{p,\rho}
+\sqrt c\left(\frac r{r-1}\right)\right)\left(1+\log\left(\frac r{r-1}\right)\right)^2\\
&\q\q+
C\sqrt c\sup_{j\le n}\tn\Psi_{j-1}1_{\{j\le N_0\}}\tn_{p,\rho}
\left(1+\log\left(\frac r{r-1}\right)\right),
\end{align*}
when $\eta < 1$ and
\begin{align*}
\tn\cR_n1_{\{n\le N_0\}}\tn_{p,r}
&\le\frac{C\d_0}{r-1}\left(\d_0+\sup_{j\le n}\tn\Psi_{j-1}1_{\{j\le N_0\}}\tn_{p,\rho}
+\sqrt c\left(\frac r{r-1}\right)\right)\left(1+\log\left(\frac r{r-1}\right)\right)\\
&\q\q+
C\sqrt c\sup_{j\le n}\tn\Psi_{j-1}1_{\{j\le N_0\}}\tn_{p,\rho}
\left(\frac r{r-1}\right),
\end{align*}
when $\eta = 1$.
\end{lemma}

\section{Convergence to a disk for ALE$(\eta)$}
\label{sec:diskconv}
In this section we derive our main disk theorem. 
Recall that 
\begin{equation}
\label{NO}
N_0=\min\left\{n\ge0:\|\tilde\Phi'_n\|_{\infty,e^\s}>\d_0\right\} . 
\end{equation}
First we show that $\tn\Psi_n1_{\{n\le N_0\}}\tn_{p,r}$ is small, provided $\d_0$ is appropriately chosen. 
Then we deduce estimates on the random norms $\|\Psi_n1_{\{n\le N_0\}}\|_{\infty,r}$, valid with high probability, 
and use them to dispense with the restriction that $n\le N_0$. 
Finally, we apply these results to show that $\Phi_n(z)$ is close to $e^{cn}z$.

\subsection{$L^p$-estimates on the differentiated fluctuations}\label{LPE}
The proposition below shows that, for an appropriately chosen $\d_0$, 
the $\tn\cdot\tn_{p,r}$ norm of the differentiated fluctuation process $\Psi_n1_{\{n\le N_0\}}$ is of order $\sqrt c$, 
with quantitative control of the singularity as $r\to1$ and the decay as $r\to\infty$. 
The dependence of the estimate on $\s$ is also explicit, allowing one to consider limits in which $\s\to0$ as $c\to0$. 
For small $c$, the estimates are strongest when $\ve$ and $\nu$ are taken to be small.
A second argument, given in the next subsection, will show that the event $\{n\le N_0\}$ appearing in \eqref{PDF} and \eqref{PDG} is of high probability in the limit $c\to0$.

\begin{proposition}
\label{DFA}
For all $\eta\in[0,1)$, $T\in(0,\infty)$, $\ve\in(0,1/2)$, $\nu\in(0,\ve /2)$ and $p \in [2,\infty ) $, there is a constant $C=C(\L,\eta,T,\ve,\nu,p)\in[1,\infty)$ with the following property.
For all $c\in(0,1]$, all $r,e^\s\ge1+c^{1/2-\ve}$ and all $n\le T/c$, we have
\begin{equation}\label{PDF}
\tn\Psi_n1_{\{n\le N_0\}}\tn_{p,r}
\le\frac Cr\left(\sqrt c\left(\frac r{r-1}\right)+\frac{c^{1-3\nu}}{(e^\s-1)^2}
\right)
\end{equation}
where $N_0$ is given by \eqref{NO} with $\d_0=c^{1/2-\nu}/(e^\s-1)$.

Moreover, in the case $\eta=1$, for all $T\in(0,\infty)$, $\ve\in(0,1/5)$, $\nu\in(0,3\ve/2)$ and $p \in [2,\infty ) $, 
there is a constant $C=C(\L,T,\ve,\nu,p)\in[1,\infty)$ with the following property.
For all $c\in(0,1]$, all $r,e^\s\ge1+c^{1/5-\ve}$ and all $n\le T/c$, we have
\begin{equation}\label{PDG}
\tn\Psi_n1_{\{n\le N_0\}}\tn_{p,r}
\le\frac Cr\bigg( \sqrt{c}\left(\frac r{r-1}\right)^{3/2}+\frac{c^{1-3\nu}}{(e^\s-1)^3}\left(\frac r{r-1}\right)  \bigg)
\end{equation}
where $N_0$ is given by \eqref{NO} with $\d_0=c^{1/2-\nu}/(e^\s-1)^{3/2}$.
\end{proposition}
\begin{proof}
As before, constants referred to in the proof by the letter $C$ may change from line to line and are all assumed to lie in $[1,\infty)$. 
They may depend on $\L$, $\eta$, $T$, $\ve$, $\nu$ and $p$ but they do not depend on $c$, $n$, $\s$ and $r$.

We begin with a crude estimate which allows us to restrict further consideration to small values of $c$.
The function $e^{-cn}\Phi_n(z)$ is univalent on $\{|z|>1\}$, with $e^{-cn}\Phi_n(z)\sim z$ as $z\to\infty$. 
By  same distortion estimate used in Section \ref{sec:HL0proof}, for all $|z|=r>1$,
$$
|\tilde\Phi_n'(z)|
=|e^{-cn}\Phi_n'(z)-1|\le\frac1{r^2-1}
$$
and so
\begin{equation}\label{DIST}
\tn\Psi_n\tn_{p,r}
=r\tn\tilde\Phi_n'\tn_{p,r}
\le\frac 1{r-1}.
\end{equation}
It is straightforward to check that this implies the claimed estimates in the case where $c>1/C$, for any given constant
$C$ of the allowed dependence.
Hence it will suffice to consider the case where $c\le1/C$.

Consider first the case where $\eta<1$. 
Fix $T$, $\ve$, $p$ and $\nu$ as in the statement,
and assume that $c\le1/e$ and $r\ge1+c^{1/2-\ve/2}$ and $e^\s\ge1+c^{1/2-\ve}$ and $n\le T/c$. Set $\rho=(r+1)/2$. 
It will suffice to prove the result for $p$ large enough, so assume $p > 1+1/(2\ve )$. 

By the triangle inequality,
$$
\tn\Psi_n1_{\{n\le N_0\}}\tn_{p,r}\leq \tn\cM_n1_{\{n\le N_0\}}\tn_{p,r} + \tn\cW_n1_{\{n\le N_0\}}\tn_{p,r} + \tn\cR_n1_{\{n\le N_0\}}\tn_{p,r},
$$
where, by Lemmas \ref{lemma:mbound}, \ref{lemma:wbound} and \ref{lemma:rnbound},
\begin{equation*}
 \tn\cM_n1_{\{n\le N_0\}}\tn_{p,r} 
\le \frac{C\sqrt c}{r-1}
=\frac{C\sqrt c}r\left(\frac r{r-1}\right), 
\end{equation*}
$$
\tn\cW_n1_{\{n\le N_0\}}\tn_{p,r}
\le\frac{C\sqrt c}r
\bigg(
\sup_{j\le n}\tn\Psi_{j-1}1_{\{j\le N_0\}}\tn_{p,\rho}\left(\frac r{r-1}\right) 
+\sqrt c\left(\frac r{r-1}\right)^2 
\bigg),
$$
and
\begin{align*}
\tn\cR_n1_{\{n\le N_0\}}\tn_{p,r}
&\le\frac{C\d_0}r\left(\d_0+\sup_{j\le n}\tn\Psi_{j-1}1_{\{j\le N_0\}}\tn_{p,\rho}
+\sqrt c\left(\frac r{r-1}\right)\right)\left(1+\log\left(\frac r{r-1}\right)\right)^2\\
&\q\q+
C\sqrt c\sup_{j\le n}\tn\Psi_{j-1}1_{\{j\le N_0\}}\tn_{p,\rho}
\left(1+\log\left(\frac r{r-1}\right)\right).
\end{align*}
On combining the estimates above and substituting the chosen value of $\d_0$, we obtain, for all $r\ge1+c^{1/2-\ve/2}$,
\begin{equation}
\label{deltadeltabar}
\tn\Psi_n1_{\{n\le N_0\}}\tn_{p,r}\leq\bar\d(r)\sup_{j \leq n}\tn\Psi_{j-1}1_{\{j\le N_0\}}\tn_{p,\rho}+\d(r)
\end{equation}
where
$$
\bar\d(r)
=\frac Cr\bigg(\sqrt c\left(\frac r{r-1}\right) 
+\frac{c^{1/2-\nu}(\log\left(1/c\right))^2}{e^\s-1}\bigg)
+C\sqrt c\log(1/c)
$$
and
$$
\d(r)
=\frac Cr\left(\sqrt c\left(\frac r{r-1}\right)+\frac{c^{1-2\nu}(\log(1/c))^2}{(e^\s-1)^2}
 \right).
$$
Note that, for all $r\ge1+c^{1/2-\ve}$, we have
$$
\bar\d(r)\le Cc^{\ve}+Cc^{\ve-\nu}(\log(1/c))^2+C\sqrt c\log(1/c)\le c^{\ve/2}\le1 
$$
for all sufficiently small $c$. 
Similarly, for $r\ge1+c^{1/2-\ve/2}$, we have $\bar\d(r)\le1$ for all sufficiently small $c$. 
As noted above, it suffices to deal with the case where $c$ is sufficiently small.

A complication in the analysis is that the right hand side of the inequality \eqref{deltadeltabar} requires estimates of $\Psi_{j-1}(z)$ when $|z| = \rho$, but the left hand side only gives information about $\Psi_n(z)$ when $|z| = r > \rho$. Our approach is therefore to use the universal distortion estimate \eqref{DIST} to obtain an initial (very weak) bound and then recursively feed the bounds through the inequality. This generates stronger and stronger estimates, but at the cost of moving $r$ further away from 1. 
  
Set $C_0=1$ and for $k\ge0$ define recursively $C_{k+1}=2^{k+1}C_k+1$.
We will show that, for all $k\ge0$, all $r\ge1+2^kc^{1/2-\ve/2}$ and all $n\le T/c$,
\begin{equation}
\label{induchop}
\tn\Psi_n1_{\{n\le N_0\}}\tn_{p,r}
\le C_k\left(\frac{\bar\d(r)^k}{r-1}+\d(r)\right).
\end{equation}
The case $k=0$ is implied by \eqref{DIST}.
Suppose inductively that \eqref{induchop} holds for $k$ and that $r\ge1+2^{k+1}c^{1/2-\ve/2}$ and $n\le T/c$. 
Then $\rho=(r+1)/2\ge1+2^k c^{1/2-\ve/2}$ so, for all $j\le n$,
\[
\tn\Psi_j1_{\{j\le N_0\}}\tn_{p,\rho}
\le C_k\left(\frac{\bar\d(\rho)^k}{\rho-1}+\d(\rho)\right)\le2^{k+1}C_k\left( \frac{\bar\d(r)^k}{r-1}+\d(r)\right)
\]
where we used the inequalities $\d(\rho)\le2\d(r)$ and $\bar\d(\rho)\le2\bar\d(r)$.
Since $r\ge1+c^{1/2-\ve/2}$, we can substitute into \eqref{deltadeltabar} to obtain
\begin{align*}
\tn\Psi_n1_{\{n\le N_0\}}\tn_{p,r} 
&\le2^{k+1}C_k\left(\frac{\bar\d(r)^{k+1}}{r-1}+\bar\d(r)\d(r)\right)+\d(r)\\
&\le C_{k+1}\left(\frac{\bar\d(r)^{k+1}}{r-1}+\d(r)\right).
\end{align*}
Hence \eqref{induchop} holds for $k+1$ and the induction proceeds.

Choose now $k=\lceil1/\ve\rceil$. 
Then
$$
\frac{\bar\d(r)^k}{r-1}\le\frac{c^{\ve k/2}}{r-1}\le\frac{\sqrt c}{r-1}\le\d(r).
$$
For $c$ sufficiently small, we have $c^{\ve/2}\le 2^{-k/2}$ so, for all $r\ge1+c^{1/2-\ve}$, 
we have $r\ge1+2^kc^{1/2-\ve/2}$ and so
$$
\tn\Psi_n1_{\{n\le N_0\}}\tn_{p,r}
\le C_k\left(\frac{\bar\d(r)^k}{r-1}+\d(r)\right)
\le2C_k\d(r)
$$
giving a bound of the desired form \eqref{PDF}.

We turn to the case where $\eta=1$. 
Fix $T$, $\ve$, $p$ and $\nu$ as in the statement for $\eta=1$.
Assume that $c\le1/e$ and $n\le T/c$, 
and assume now that $r\ge1+c^{1/5}$ and $e^\s\ge1+c^{1/5-\ve}$. 
It will suffice to prove the result for $p$ sufficiently large. 
The argument follows the same pattern as the case where $\eta<1$, 
except for modifications necessary because of the different estimates in Lemmas \ref{lemma:mbound}, \ref{lemma:wbound} and \ref{lemma:rnbound} (and different choice of $\d_0$), which arose because $r_n=re^{c(1-\eta)n}=r$.

We obtain for $r\ge1+c^{1/5}$,
\begin{equation*}
\tn\Psi_n1_{\{n\le N_0\}}\tn_{p,r}
\le\bar\d(r)\sup_{j\le n}\tn\Psi_{j-1}1_{\{j\le N_0\}}\tn_{p,\rho}+\d(r)
\end{equation*}
where now
$$
\bar\d(r)=\frac Cr\bigg(\sqrt c\left(\frac r{r-1}\right)^{3/2}+\frac{c^{1/2-\nu}\log(1/c)}{(e^\s-1)^{3/2}}\left(\frac r{r-1}\right)\bigg)+C\sqrt c \left(\frac r{r-1}\right)
$$
and
$$
\d(r)=\frac Cr\bigg( \sqrt{c}\left(\frac r{r-1}\right)^{3/2}+\frac{c^{1-2\nu}\log(1/c)}{(e^\s-1)^3}\left(\frac r{r-1}\right) \bigg).
$$
Note that, 
for $r\ge1+c^{1/5-\ve}$, we have,
for all sufficiently small $c$
$$
\bar\d(r)
\le Cc^{1/5+3\ve/2}+Cc^{5\ve/2-\nu}\log(1/c)+Cc^{3/10+\ve}\le c^\ve\le1 
$$ 
Similarly, we have $\d(r)\le1$ whenever $r\ge1+c^{1/5}$, for all sufficiently small $c$.
We restrict to such $c$.
For $\rho=(r+1)/2$, we now have modified inequalities $\d(\rho)\le2^{3/2}\d(r)$ and $\bar\d(\rho)\le2^{3/2}\bar\d(r)$.
Set $C_0=1$ and for $k\ge0$ define now recursively $C_{k+1}=2^{3k/2+1}C_k+1$.
Then, by an analogous inductive argument, we obtain, for all $k\ge0$, all $n\le T/c$ and all $r\ge1+2^kc^{1/5}$,
\begin{equation*}
\label{induchop1}
\tn\Psi_n1_{\{n\le N_0\}}\tn_{p,r}
\le C_k\left(\frac{\bar\d(r)^k}{r-1}+\d(r)\right).
\end{equation*}
Choose now $k=\lceil 1/\ve\rceil$ and assume that $r\ge1+c^{1/5-\ve}$.
Then
\[ 
\frac{\bar\d(r)^k}{(r-1)}\le\frac{c^{\ve k}}{r-1}\le\frac c{r-1}\le\d(r). 
\] 
and, for $c$ sufficiently small, we have $c^\ve\le2^{-k}$, so $r\ge1+2^kc^{1/5}$ and so
$$
\tn\Psi_n1_{\{n\le N_0\}}\tn_{p,r}\le2C_k\d(r)
$$
which is a bound of the required form \eqref{PDG}.
\end{proof}

\subsection{Spatially-uniform high-probability estimates on the differentiated fluctuations}
\label{sec:highprob}
We now use the results from the previous section to obtain uniform estimates on $\Psi_n(z)$.
\begin{proposition}
\label{UDFA}
For all $\eta\in[0,1)$, $\ve\in(0,1/2)$, $\nu\in(0,\ve /4)$, $m\in\N$ and $T\in(0,\infty)$, 
there is a constant $C=C(\L,\eta,\ve,\nu,m,T)<\infty$ with the following properties.
For all $c\in(0,1]$ and all $e^\s\ge1+c^{1/2-\ve}$,
there is an event $\O_0$ of probability exceeding $1-c^m$ on which,
for all $n\le T/c$ and all $|z|=r\ge1+c^{1/2-\ve}$,
\begin{equation}
\label{SUPE}
|\Psi_n(z)|
\le\frac Cr\left(c^{1/2-\nu}\left(\frac r{r-1}\right)+\frac{c^{1-4\nu}}{(e^\s-1)^2}
 \right).
\end{equation}
Moreover, for $c\le1/C$, we have $\O_0\sse\{n\le N_0\}$ for all $n\le T/c$, 
where $N_0$ is given by \eqref{NO} with $\d_0=c^{1/2-\nu}/(e^\s-1)$. 

For $\eta=1$, $\ve\in(0,1/5)$, $\nu\in(0,\ve /2)$, $m\in\N$ and $T\in(0,\infty)$, 
there is a constant $C=C(\L,\ve,\nu,m,T)<\infty$ with the following property.
For all $c\in(0,1]$ and all $e^\s\ge1+c^{1/5-\ve}$,
there is an event $\O_0$ of probability exceeding $1-c^m$ on which,
for all $n\le T/c$ and all $|z|=r\ge1+c^{1/5-\ve}$,
$$
|\Psi_n(z)|\le\frac Cr\left(c^{1/2-\nu}\left(\frac r{r-1}\right)^{3/2}
+\frac{c^{1-4\nu}}{(e^\s-1)^3}\left(\frac r{r-1}\right) \right).
$$
Morover, for $c\le1/C$, we have $\O_0\sse\{n\le N_0\}$ for all $n\le T/c$, 
where $N_0$ is given by \eqref{NO} with $\d_0=c^{1/2-\nu}/(e^\s-1)^{3/2}$. 
\end{proposition}
\begin{proof}
We will give details for the case $\eta\in[0,1)$. 
The minor modifications needed for the case $\eta=1$ are left to the reader.
Fix $\eta,\ve,\nu,m$ and $T$ as in the statement.
It will suffice to consider the case where  $e^\s\ge1+2c^{1/2-\ve}$, and to find an event $\O_0$ 
of probability exceeding $1-c^m$ on which \eqref{SUPE} holds whenever $r\ge1+2c^{1/2-\ve}$ and $n\le T/c$.
Set
$$
K=\min\{k\ge1:2^{k}c^{1/2-\ve}\ge1\},\q
N=\lfloor T/c\rfloor.
$$ 
Then $K\le\lfloor\log(1/c)\rfloor+1$. 
For $k=1,\dots,K$, set 
$$
r(k)=1+2^{k}c^{1/2-\ve},\q\rho(k)=\frac{r (k)+1}2.
$$
Then $\rho(k)\ge1+c^{1/2-\ve }$ and $r(K)\in[2,4]$. 
Choose  $p\ge \max \{ 1+1/(2\ve ) ; (m+2)/\nu \}$ even integer,  and set
$$
R=\left(KTc^{-m}\right)^{1/p}.
$$
By Proposition \ref{DFA}, there is a constant $C=C(\L,\eta,\ve,\nu,p,T)<\infty$ such that, for all $n\le T/c$,
\begin{equation*}
\tn\Psi_n1_{\{n\le N_0\}}\tn_{p,\rho(k)}\le\mu_k
\end{equation*}
where $N_0$ is defined as in the statement and
$$
\mu_k=\frac C{r(k)}\left(\sqrt c\left(\frac{r(k)}{r(k)-1}\right)+\frac{c^{1-3\nu}}{(e^\s-1)^2} 
\right) .
$$
Set $\l=Rc^{-1/p}$ and consider the event
\begin{equation*}
\O_0=\bigcap_{n=1}^N\bigcap_{k=1}^K \{\|\Psi_n\|_{p,\rho(k)}1_{\{n\le N_0\}}\le\l\mu_k\}.
\end{equation*}
By Chebyshev's inequality,
$$
\PP(\|\Psi_n\|_{p,\rho(k)}1_{\{n\le N_0\}}>\l\mu_k)\le\l^{-p}=cR^{-p}
$$
so
$$
\PP(\O_0^c)\le KTR^{-p}=c^m.
$$ 
Fix $r\ge1+2c^{1/2-\ve}$. 
Then $r(k)\le r<r(k+1)$ for some $k\in\{1,\dots,K\}$, where we set $r(K+1)=\infty$.
Note that $z\Psi_n(z)$ is a bounded analytic function on $\{|z|>\rho(1)\}$.
We use the inequality \eqref{PINF} to see that, on the event $\O_0$, for $n\le N_0\wedge N$,
$$
r\|\Psi_n\|_{\infty,r}
\le r(k)\|\Psi_n\|_{\infty,r(k)}
\le\left(\frac{r(k)+1}{r(k)-1}\right)^{1/p}r(k)\|\Psi_n\|_{p,\rho(k)}
\le(2c^{-1/2})^{1/p}r(k)\l\mu_k
$$
so, using that $r(k) \ge (r+1) /2$, we get 
$$
\|\Psi_n\|_{\infty,r}
\le(2c^{-1/2})^{1/p}\l\mu_k
\le\frac{\g_k}{2r}\left(c^{1/2-\nu}\left(\frac r{r-1}\right)+\frac{c^{1-4\nu}}{(e^\s-1)^2}\right)
$$
where
$$
\g_k=8C(2\log(1/c)Tc^{-m-1-1/2+p\nu})^{1/p}.
$$
By our choice of $p$, we have $\g_k\le1$ for all sufficiently small $c$.
We can restrict to such $c$, since the desired estimate follows from the distortion inequality \eqref{DIST} otherwise.
Then, on the event $\O_0$, for $n\le N_0\wedge N$, 
$$
\|\Psi_n\|_{\infty,r}
\le\frac1{2r}\left(c^{1/2-\nu}\left(\frac r{r-1}\right)+\frac{c^{1-4\nu}}{(e^\s-1)^2}\right)
$$
and in particular, since $e^\s\ge1+c^{1/2-\ve}$ and $\nu < \ve /4$, we have
$$
\|\tilde\Phi_n'\|_{\infty,e^\s}
\le\|\Psi_n\|_{\infty,e^\s}
\le\frac1{2e^\s}\left(c^{1/2-\nu}\left(\frac{e^\s}{e^\s-1}\right)+\frac{c^{1-4\nu}}{(e^\s-1)^2}\right)
\le\frac{c^{1/2-\nu}}{e^\s-1}=\d_0
$$
which forces $N_0>N$ on $\O_0$.
\end{proof}

\def\j{
\bb
\subsection{Details for $\eta=1$ omitted from the proof of Proposition \ref{UDFA}}
Fix now $\eta=1$, and take $\eta,\ve,\nu,m$ and $T$ as in the statement. 
It will suffice to consider the case where $\nu<5\ve/6$ and $e^\s\ge1+2c^{1/5-\ve}$, and to find an event $\O_0$ 
of probability exceeding $1-c^m$ on which the claimed bound holds whenever $r\ge1+2c^{1/5-\ve}$ and $n\le T/c$.
Set 
$$
K=\min\{k\ge1:2^{k}c^{1/5-\ve}\ge1\},\q
N=\lfloor T/c\rfloor.
$$ 
Then $K\le\lfloor\log(1/c)\rfloor+1$. 
For $k=1,\dots,K$, set 
$$
r(k)1+2^{k}c^{1/5-\ve },\q \rho(k)=\frac{r(k)+1}2.
$$
Then $\rho(k)\ge1+c^{1/5-\ve }$ and $r(K)\in[2,4]$. 
Choose an even integer $p\ge(m+2)/\nu$ and set
$$
R=\left(KTc^{-m}\right)^{1/p}.
$$
By Proposition \ref{DFA}, there is a constant $C<\infty$, depending on $\eta,\ve,\nu,T$ and $p$ such that, for all $n\ge0$,
\begin{equation}\label{PSS1}
\tn\Psi_n1_{\{n\le N_0\}}\tn_{p,\rho(k)} 
\le\frac{C\sqrt c}{r(k)-1}\left(\Big(\frac{r(k)}{r(k)-1}\Big)^{1/2} 
+\frac{c^{1/2-3\nu}}{(e^\s-1)^3}\right) , 
\end{equation}
where
$$
N_0=\min\{n\ge0:\|\tilde\Phi'_n\|_{\infty,e^\s}>\d_0\},\q
\d_0=\frac{c^{1/2-\nu}}{(e^\s-1)^{3/2}}.
$$
Consider the event
\begin{equation*}
\O_0
=
\bigcap_{n=1}^N
\bigcap_{k=1}^K
\{\|\Psi_n\|_{p,\rho(k)}1_{\{n\le N_0\}}\le\a_k\}
\end{equation*}
where
$$
\a_k=\frac{C\sqrt c}{r(k)-1}\left(\Big(\frac{r(k)}{r(k)-1}\Big)^{1/2} 
+\frac{c^{1/2-3\nu}}{(e^\s-1)^3}\right)\l,\q
\l=Rc^{-1/p}
$$
for $C$ as in \eqref{PSS1}. 
By Chebyshev's inequality,
$$
\PP(\|\Psi_n\|_{p,\rho(k)}1_{\{n\le N_0\}}>\a_k)\le\l^{-p}=cR^{-p}
$$
so
$$
\PP(\O_0^c)\le KTR^{-p}=c^m.
$$ 
Fix $r\ge1+2c^{1/5-\ve}$. 
Then $r(k)\le r<r(k+1)$ for some $k\in\{1,\dots,K\}$, where we set $r(K+1)=\infty$.
Note that $z\Psi_n(z)$ is a bounded analytic function on $\{|z|>\rho_1\}$.
On the event $\O_0$, for all $n\le N_0\wedge N$,
\begin{align*}
r\|\Psi_n\|_{\infty,r}
&\le r(k)\|\Psi_n\|_{\infty,r(k)}
\le\left(\frac{r(k)+1}{r(k)-1}\right)^{1/p}r(k)\|\Psi_n\|_{p,\rho(k)}
\le(6c^{-1/5+\ve})^{1/p}r(k)\a_k\\
&\le6^{1/p}\cdot 8C(KTc^{-m-1-1/5+\ve})^{1/p} 
 \frac{\sqrt cr}{r-1}\left(\Big(\frac r{r-1}\Big)^{1/2} 
 +\frac{c^{1/2-3\nu}}{(e^\s -1)^3}\right), 
\end{align*}
for $C$ as in \eqref{PSS1}. 
We can choose $c_0\in(0,1]$ so that
$$
6(16C)^pT(1+\log(1/c_0))c_0^{p\nu-m-1-1/5+\ve}=1.
$$
For $c\in(c_0,1]$, the desired estimates follow from the distortion inequality \eqref{DIST}.
On the other hand, for $c\in(0,c_0]$, on the event $\O_0$, for all $n\le N_0 \wedge N $ and all $r\ge1+2c^{1/5-\ve }$,
\begin{equation}
\label{DES1}
\|\Psi_n\|_{\infty,r}
\le\frac{c^{1/2-\nu}}{ 2(r-1) }\left( \Big( \frac r{r-1}\Big)^{1/2} +\frac{c^{1/2-3\nu}}{(e^\s -1)^3}\right).
\end{equation}
Since $e^\s\ge 1+ 2c^{1/5-\ve}$, this implies in particular that, on $\O_0$, for $n=N_0 \wedge N $,
$$
\|\tilde\Phi_n'\|_{\infty,e^\s}
= e^{-\s} \|\Psi_n\|_{\infty,e^\s}
\le\frac{c^{1/2-\nu}}{(e^\s -1)^{3/2} } \left( \frac12 + \frac{c^{1/2-3\nu}}{2(e^\s -1)^{5/2}}\right)
\leq \d_0 , 
$$
which forces $N_0>N$, so \eqref{DES1} holds for all $n\le N$.
\eb
}\jj

\subsection{$L^p$-estimates on the fluctuations}
In this section we prove a result analogous to Proposition \ref{DFA} for the undifferentiated dynamics. 
This allows us to prove Theorem \ref{UFA}. 

\begin{proposition}
\label{DFAun}
For all $\eta\in[0,1)$, $T\in(0,\infty)$, $\ve\in(0,1/2)$, $\nu\in(0,\ve /4)$ and $p  \in [2,\infty ) $, there is a constant $C=C(\L,\eta,T,\ve,\nu,p)\in[1,\infty)$ with the following property.
For all $c\in(0,1]$, all $r,e^\s\ge1+c^{1/2-\ve}$ and all $n\le T/c$, we have
\begin{equation}\label{PDFun}
\tn(\tilde\Phi_n-\tilde\Phi_n(\infty))1_{\{n\le N_0\}}\tn_{p,r}
\le\frac Cr\bigg(\sqrt c\left(1+\log\left(\frac r{r-1}\right)\right)^{1/2}+\frac{c^{1-3\nu}}{(e^\s-1)^2}\bigg)
\end{equation}
and
\begin{equation}
\label{PDFuni}
\|\tilde\Phi_n(\infty)1_{\{n\le N_0\}}\|_p\le C\left(\sqrt c+\frac{c^{1-2\nu}}{(e^\s-1)^2}\right)
\end{equation}
where $N_0$ is given by \eqref{NO} with $\d_0=c^{1/2-\nu}/(e^\s-1)$.

Moreover, in the case $\eta=1$, for all $T\in(0,\infty)$, $\ve\in(0,1/5)$, $\nu\in(0,\ve/2)$ and $p \in [2,\infty ) $, 
there is a constant $C=C(\L,T,\ve,\nu,p)\in[1,\infty)$ with the following property.
For all $c\in(0,1]$, all $r,e^\s\ge1+c^{1/5-\ve}$ and all $n\le T/c$, we have
\begin{equation}\label{PDGun}
\tn(\tilde\Phi_n-\tilde\Phi_n(\infty))1_{\{n\le N_0\}}\tn_{p,r}
\le\frac {C }{r} \left( \sqrt{c} \left( \frac r{r-1}\right)^{1/2} 
+\frac{c^{1-3\nu}}{(e^\s-1)^3}\right)
\end{equation}
and
\begin{equation}\label{PDGuni}
\|\tilde\Phi_n(\infty)1_{\{n\le N_0\}}\|_{p}
\le C\bigg(\sqrt c+\frac{c^{1-2\nu}}{(e^\s-1)^3}\bigg)
\end{equation}
where $N_0$ is given by \eqref{NO} with $\d_0=c^{1/2-\nu}/(e^\s-1)^{3/2}$.
\end{proposition}
\begin{proof}
Let us first consider $\eta\in[0,1)$. It suffices to prove the result for $p> 1+1/(2\ve )$.  
We modify the estimates leading to \eqref{MNEST}, 
by deleting $D$ and $k^2$ and considering separately the constant term of the Laurent expansion, to obtain
\begin{equation*}
\label{MNESTBB}
\begin{split} 
\tn(\tilde\cM_n-\tilde\cM_n(\infty))1_{\{n\le N_0\}}\tn_{p,r}^2 
& \le\frac{Cc^2}{r^2}\sum_{j=1}^ne^{-2c(n-j)}\left(\frac{r_{n-j}}{r_{n-j}-1}\right)
+ \frac{C c^{2-2/p}}{r^2} \left( \frac{r}{r-1} \right)^{2-2/p} 
\\ & \le\frac{Cc}{r^2}\left(1+\log\left(\frac r{r-1}\right)\right)
\end{split} 
\end{equation*}
and
$$
\|\tilde\cM_n(\infty)1_{\{n\le N_0\}}\|_p^2\le Cc^2\sum_{j=1}^ne^{-2c(n-j)}\le Cc.
$$
Similarly, for the second martingale term we obtain
\begin{align*}
\notag
\tn\tilde\cW_n1_{\{n\le N_0\}}\tn_{p,r}^2
&\le\frac{Cc^2}{r^2}\sum_{j=1}^ne^{-2c(n-j)}
\bigg(\tn\Psi_{j-1}1_{\{j\le N_0\}}\tn_{p,\rho_{n-j}}^2\left(\frac{r_{n-j}}{r_{n-j}-1}\right)
+c\left(\frac{r_{n-j}}{r_{n-j}-1}\right)^3\bigg)\\
& \quad + \frac{C e^{2c(n-1)} c^{3-2/p} r^{2-2/p} }{(r-1)^{4-2/p} } + 
\frac{C e^{2c(n-1)} c^{2-2/p} r^{2-2/p}}{r^2 (r-1)^{2-2/p}} 
\sup_{j\leq n } \tn \Psi_{j-1}1_{\{j\le N_0\}}\tn_{p,\rho_{n-j}}^2 
\\ &\le \frac{Cc^2}{r^2}\left(\frac r{r-1}\right)^2 
+\frac{Cc^{3-7\nu}}{r^4(e^\s-1)^4} , 
\end{align*}
where we used the bound on $\Psi_n$ from Proposition \ref{DFA}.  
Finally, for the remainder term, we find  
\begin{align*}
\notag
&\tn(\tilde\cR_n-\tilde\cR_n(\infty))1_{\{n\le N_0\}}\tn_{p,r}
\le e^{-cn}\sum_{j=1}^n\tn P_0^{n-j}(R_j-R_j(\infty ))1_{\{j\le N_0\}}\tn_{p,r_{n-j}}\\
\notag
&\q\le\frac{Cc\d_0}r
\sum_{j=1}^n
e^{-c(n-j)}
\left(\d_0+\tn\Psi_{j-1}1_{\{j\le N_0\}}\tn_{p,\rho_{n-j}}  
+ \sqrt c\left(\frac{r_{n-j}}{r_{n-j}-1}\right)\right)
\left(1+\log\Big(\frac{r_{n-j}}{r_{n-j}-1}\Big)\right)\\
&\q\q\q\q+Cc^{3/2}
\sum_{j=1}^n
e^{-c(n-j)}
\tn\Psi_{j-1}1_{\{j\le N_0\}}\tn_{p,\rho_{n-j}}\\
\label{RNESTBB}
&\q\q\le\frac{Cc^{1-3\nu}}{r(e^\s-1)^2}. 
\end{align*}
and
$$
\|\tilde\cR_n(\infty)1_{\{n\le N_0\}}\|_p
\le\frac{Cc\d_0^2}{(e^\s-1)^2}\sum_{j=1}^ne^{-c(n-j)}
\le\frac{Cc^{1-2\nu}}{(e^\s-1)^2}.
$$
On assembling these bounds, and simplifying using our constraints on $r$ and $\s$, we obtain \eqref{PDFun} and \eqref{PDFuni}. 

As in the proof of Proposition \ref{DFA}, in the case $\eta=1$, we do not benefit from the push-out of $r_n=re^{c(1-\eta)n}$,
and the bound on $\Psi_n$ is weaker.
After some straightforward modifications, for $p$ sufficiently large we obtain
\begin{align*} 
\tn(\tilde\cM_n-\tilde\cM_n(\infty))1_{\{n\le N_0\}}\tn_{p,r}^2 
&\le\frac{Cc}{r^2}\left(\frac r{r-1}\right),\\
\tn\tilde\cW_n1_{\{n\le N_0\}}\tn_{p,r}^2 
&\le\frac{Cc^2}{r^2}\left(\frac r{r-1}\right)^4
+\frac{Cc^{3-6\nu}}{r^4(e^\s-1)^6}\left(\frac r{r-1}\right)^3,  \\ 
\tn(\tilde\cR_n-\tilde\cR_n(\infty))1_{\{n\le N_0\}}\tn_{p,r} 
&\le\frac{Cc^{1-3\nu}}r
\left(\frac1{(e^\s-1)^3}
+\frac1{(e^\s-1)^{3/2}}
\left(\frac r{r-1}\right)^{3/2}\right).
\end{align*}  

On assembling these bounds, and simplifying using our constraints on $r$ and $\s$, we obtain \eqref{PDGun}.
Similarly
\begin{align*}
\| \tilde\cM_n(\infty) 1_{\{n\le N_0\}}\|_p^2 
&\le Cc, \q \|\tilde\cR_n(\infty)1_{\{n\le N_0\}}\|_p
\le\frac{Cc^{1-2\nu}}{(e^\s-1)^3}, 
\end{align*}
giving \eqref{PDGuni}. 
\end{proof}

\begin{proof}[Proof of Theorem \ref{UFA}]
The argument is a variation of that for Proposition \ref{UDFA}.  We do it when $\eta<1$; the $\eta=1$ case is similar. 
Let $\O_0$, $N$, $K$, $r(k)$, $\rho(k)$ and $\l$ be as in the proof of Proposition \ref{UDFA}.
Define
$$
\O_1=\O_0\cap
\bigcap_{n=1}^N
\bigcap_{k=1}^K
\{\|\tilde\Phi_n\|_{p,\rho(k)}1_{\{n\le N_0\}}\le\l\b_k\}
$$
where 
$$
\b_k=2C\bigg(\sqrt c\bigg(1+\log\bigg(\frac{r(k)}{r(k)-1}\bigg)\bigg)^{1/2}+\frac{c^{1-3\nu}}{(e^\s-1)^2}\bigg)
$$
and $C$ is the larger of the constant in \eqref{PDFun} and that in \eqref{PDFuni}. 
Then $\PP(\O_1)\le2c^m$ and the desired uniform estimate on $\Phi_n$  holds on $\O_1$, by the argument used in the proof of Proposition \ref{UDFA}. 
In arriving at this estimate we use the fact that for $r\ge1+c^{1/2}$ we have $(1+\log(r/(r-1)))^{1/2}\le c^{-\ve}$ for all sufficiently small $c>0$, for all $\ve >0$. 
\end{proof}

\section{Fluctuation scaling limit for ALE$(\eta)$} 
\label{sec:fluctuations}
In this section, we show that the fluctuations of ALE$(\eta)$ for $\eta\in(-\infty,1]$ are of order $\sqrt c$,
and we determine the distribution of the rescaled fluctuations.

Let $(\Phi_n)_{n\ge0}$ be an ALE$(\eta)$ process with basic map $F$ and regularization parameter $\s$.
Assume that $F$ has capacity $c\in(0,1]$ and regularity bound $\L\in[0,\infty)$.
We consider the limit $c\to0$ with $\s\to0$, 
and will show weak limits which are otherwise uniform in $F$, subject to the given regularity bound.
We embed in continuous time by setting $n(t)=\lfloor t/c\rfloor$ and defining 
\[ 
\Phi(t,z)=\Phi_{n(t)}(z),\q \tilde\Phi(t,z)=e^{-cn(t)}\Phi_{n(t)}(z)-z.
\]
We will show that the process of analytic functions $(\tilde\Phi(t,.)/\sqrt c)_{t\ge0}$ converges weakly to a Gaussian limit.

Let us define the metric spaces our processes will live in. 
To start with, let $D[0,\infty)$ denote the space of complex-valued \cadlag processes equipped with the Skorohod metric $\mathbf{d}$. 
To discuss weak convergence of sequences of Laurent coefficients, 
it is convenient to introduce the product space $D[0,\infty)^{\Z^+}$ of sequences of complex-valued \cadlag processes, 
with the metric of coordinate-wise convergence, given by 
\begin{equation*} 
\label{productM}
\mathbf{d}^{\Z^+}((a(k))_{k\ge0},(b(k))_{k\ge0})
=\sum_{k\ge0}2^{-k}\left(1\wedge\mathbf{d}(a(k),b(k))\right).
\end{equation*}
Finally, to talk about convergence of functions, 
let $\cH$ denote the space of analytic functions on $D_0=\{|z|>1\}$ with limits at $\infty$, 
equipped  with the metric of uniform convergence on compacts in $D_0\cup\{\infty\}$, given by
\[ 
d_\cH(f,g)
=\sum_{m\ge0}2^{-m}\left(1\wedge\sup_{|z|\ge1+2^{-m}}|f(z)-g(z)|\right). 
\]
We let $D_\cH[0,\infty)$ denote the space of $\cH$-valued \cadlag processes equipped with the associated Skorohod metric $\mathbf{d}_\cH$. 
Then all the above spaces are complete separable metric spaces \cite{B}, 
and $(\tilde\Phi(t,.)/\sqrt c)_{t\ge0}$ lies in $D_\cH[0,\infty)$.

To state our main fluctuation result, we now define the limiting fluctuation field on $C_\cH[0,\infty)$, 
the space of continuous processes with values in $\cH$. 
Let $(A(\cdot,k))_{k\ge0}$ denote a sequence of independent complex Ornstein--Uhlenbeck processes, solutions to 
\begin{equation}
\label{OU}
\begin{cases}
dA(t,k) & \!\!\!\! =-(1+(1-\eta)k)A(t,k)dt+\sqrt2dB_k(t),\\
\;A(0,k) & \!\!\!\! =0 
\end{cases} 
\end{equation}
where $(B_k)_{k\ge0}$ are independent complex Brownian motions. 
Thus $(A(\cdot,k))_{k\ge0}$ is a zero-mean Gaussian process, with covariance given for $s,t\in[0,\infty)$ by
$$
\E(A(s,k)\otimes A(t,k))=\int_{|s-t|}^{s+t}e^{-(1+(1-\eta)k)u}du
\left(\begin{matrix}1&0\\0&1\end{matrix}\right). 
$$
Here, on the left, we use the tensor product from $\R^2$.
Thus 
$$
(x+iy)\otimes(x'+iy')=\left(\begin{matrix} xx' & xy' \\ yx' & yy' \end{matrix}\right). 
$$
By standard estimates, the following series both converge almost surely, uniformly on compacts in 
$(t,z)\in[0,\infty)\times(D_0\cup\{\infty\})$
\[ 
\cF(t,z)=\sum_{k\ge0}A(t,k)z^{-k},\q \xi(t,z)=\sqrt 2\sum_{k\ge0}B_k(t)z^{-k}.
\]
Hence $\cF=(\cF(t,.):t\ge0)$ and $\xi=(\xi(t,.):t\ge0)$ are continuous random processes in $\cH$.
It is straightforward to check that
$$
\cF(t,z)=(1-\eta)\int_0^tD\cF(s,z)ds-\int_0^t\cF(s,z)ds+\xi(t,z),
$$
and $\xi$ is the analytic extension in $D_0$ of space-time white noise on the unit circle, so $\cF$ satisfies the stochastic PDE \eqref{PDEfluct}. 
In this section we prove Theorem 
\ref{th:fluctuations_intro} 
by showing that $\tilde\Phi/\sqrt c\to\cF$ in distribution on $D_\cH[0,\infty)$.

\subsection{Discarding lower order fluctuations}
Our analysis is based on the decomposition \eqref{barDECOMP}, which we rewrite in continuous time, with obvious notation as 
\[ 
\tilde\Phi(t,z)=\tilde\cM(t,z)+\tilde\cW(t,z)+\tilde\cR(t,z). 
\]
Define $\tilde{\cM}^0(t,z) = \b^{-1} \tilde{\cM}(t,z)$, where $\b$ is defined in Proposition \ref{PPC}, and recall that $|\b -1| \leq \Lambda \sqrt{c}$.  
In a first step, we will show that $\tilde{\cM}^0$ is the only term that matters in the limiting fluctuations. 

\begin{lemma} 
\label{le:neglect}
Under the hypotheses of Theorem \ref{th:fluctuations_intro}, for all $t\ge0$, we have 
$$
\sup_{s\le t}d_\cH\left( \frac{(\tilde\Phi-\tilde{\cM}^0)(s,.)}{\sqrt c} ,0\right)\to 0
$$
in probability as $c\to0$, uniformly in $\s$ and $F$.
\end{lemma}
\begin{proof}
Fix $\ve\in(0,1/6)$ as in the statement of Theorem \ref{th:fluctuations_intro} and set
\[ 
\d_0=
\begin{cases}
c^{1/2-\ve/3}/(e^\s-1),&\text{ if $\eta <1$},\\
c^{1/2-\ve/6}/(e^\s-1)^{3/2},&\text{ if $\eta=1$}.
\end{cases}
\] 
We first consider the case $\eta \in (-\infty,1)$. Recall that in the proof of Proposition \ref{DFAun} we showed that, for all $T\in[0,\infty)$, $p > 1+1/(2\ve )$ and $r>1$, there is a constant $C=C(\L,\eta,T,\ve,p,r)<\infty$ such that for all 
 $c\le1/C$, $e^\s\ge1+c^{1/2-\ve}$ and $n\le T/c$, we have 
\begin{equation*}
\begin{split} 
  &\tn (\tilde\cM_n - \tilde{\cM}^0 )1_{\{n\le N_0\}}\tn_{p,r} = \frac{|\beta-1|}{|\beta|}\tn \tilde\cM_n 1_{\{n\le N_0\}}\tn_{p,r} \le Cc, \\ 
&\tn\tilde\cW_n1_{\{n\le N_0\}}\tn_{p,r}\le Cc+\frac{Cc^{3/2-\ve}}{(e^\s-1)^2},\q
\tn\tilde\cR_n1_{\{n\le N_0\}}\tn_{p,r}\le\frac{Cc^{1-\ve}}{(e^\s-1)^2}.
\end{split} 
\end{equation*}
Here we have used that $|\b-1|\le\L\sqrt c$. 
Note that under the further restriction $\s\ge c^{1/4-\ve}$,
$$
c^{1-\ve}/(e^\s-1)^2\le c^{1/2 +\ve}.
$$
By arguments from the proof of Proposition \ref{UDFA}, it follows that
\begin{equation*}
(\tilde\cM - \tilde{\cM}^0 )(t,z)1_{\{t/c\le N_0\}}/\sqrt c\to0,\q
\tilde\cW(t,z)1_{\{t/c\le N_0\}}/\sqrt c\to0,\q
\tilde\cR(t,z)1_{\{t/c\le N_0\}}/\sqrt c\to0
\end{equation*}
in probability as $c\to0$,
uniformly on compacts in $(t,z)\in[0,T]\times(D_0\cup\{\infty\})$,
and uniformly in $\s$ and $F$ subject to the given constraints.
On the other hand, by Proposition \ref{UDFA}, we know that $\PP(N_0<T/c)\to0$ in the same limiting regime.
The claim of the lemma follows.

The case $\eta=1$ is handled by the same argument with straightforward modifications.
\end{proof}

\subsection{Covariance structure}
We now focus on the leading order fluctuations, coming from the martingale term 
\begin{equation} 
\label{Mfluct2}
\tilde\cM^0(t,z)
=\sum_{j=1}^{n(t)}e^{-cj}P^{n(t)-j}M^0_j(z), 
\end{equation}
where
\[  
M_n^0(z) = \b^{-1} M_n(z) 
=\frac{2ce^{cn}z}{e^{-i\Th_n}z-1}-\E\bigg(\frac{2ce^{cn}z}{e^{-i\Th_n}z-1}\bigg|\cF_{n-1}\bigg).
\]
Let $(\Theta_n^u)_{n\ge1}$ be a sequence of independent uniform random variables in $[0,2\pi)$.
Define for $|z|>1$
\[ 
M_n^u(z)
=\frac{2ce^{cn}z}{ze^{-i\Th^u_n}-1} - \E \left( \frac{2ce^{cn}z}{ze^{-i\Th^u_n}-1} \bigg| \cF_{n-1}^u \right) 
=\frac{2ce^{cn}z}{ze^{-i\Th^u_n}-1}, 
\]
where $\cF_{n-1}^u $ is the $\sigma$-algebra generated by $\{\Theta_k^u : k\leq n-1\}$. 
Expanding in Laurent series, we find 
\[  
M_n^0(z)=\sum_{k\ge0}\hat M^0_n(k)z^{-k},\q
M_n^u(z)=\sum_{k\ge0}\hat M^u_n(k)z^{-k}
\]
where
\[ 
\begin{split} 
\hat M^0_n(k)
=2ce^{cn}\left(e^{i\Th_n(k+1)}-\E(e^{i\Th_n(k+1)}|\cF_{n-1})\right),\q
\hat M_n^u(k)=2ce^{cn}e^{i\Th^u_n (k+1)}. 
\end{split}
\]
Recalling that the operator $P$ acts diagonally on Laurent coefficients, set 
\[ 
\begin{split} 
a_{j,n}(k)
=\frac{e^{-cj}p(k)^{n-j}\hat M^0_j(k)}{\sqrt c},\q
 u_{j,n}(k) 
=\frac{e^{-cj}p(k)^{n-j}\hat M^u_j(k)}{\sqrt c}
\end{split} , 
\]
where 
	\[ p(k) = e^{-c(k+1)} + c\eta k e^{-\s (k+1)} ,\]
and define for $t\ge0$
\begin{equation*} 
\label{AU}
\tilde A(t,k)
=\sum_{j=1}^{n(t)}a_{j,n(t)}(k),\q
U(t,k) 
=\sum_{j=1}^{n(t)}u_{j,n(t)}(k).
\end{equation*}
Let $\tilde\cM^u(t,z)$ be defined as in \eqref{Mfluct2} with $M^0_j$ replaced by $M_j^u$. 
Then we have 
\[ 
\begin{split} 
\frac{\tilde\cM^0(t,z)}{\sqrt c}  
=\sum_{k\ge0}\tilde A(t,k)z^{-k},\q
\frac{\tilde\cM^u(t,z)}{\sqrt c} 	
=\sum_{k\ge0}U(t,k)z^{-k}. 
\end{split} 
\]
By an elementary calculation, we obtain
\[
\E(\hat{M}^u_j(k)\otimes\hat{M}^u_j(k'))
=2c^2e^{2cj}\d_{kk'}
\left(\begin{matrix} 1 & 0 \\ 0 & 1 \end{matrix}\right) 
\]
from which 
\[ 
\E(u_{j,n}(k)\otimes u_{j,n'}(k')) 
=2cp(k)^{n+n'-2j}\d_{k k'}
\left(\begin{matrix}1&0\\0&1\end{matrix}\right). 
\]
Recall that for $\eta \in [0,1]$
\[ 
p_0(k)= e^{c(1+(1-\eta)k)} p(k) .
\]
By some straightforward estimation, recalling that $\s \to 0$, we have
$$
0\le1-p_0(k)^{2j}\le Cc\s j k(k+1).
$$
Note that if $j \leq t/c$ for some $t>0$, and $k$ is fixed, then the right hand side converges to $0$ as $c\to 0$. 
In the case $\eta <0$, define $p_0(k)$ exactly as above (note that this differs from the definition in \eqref{eq:p0}). Provided $c$ is taken sufficiently small that $\s-c-c|\eta|>0$, 
 we have 
\[ 
1+c\eta ke^{-(\s-c)(k+1)}\ge e^{c\eta k} , 
\]
and hence $p_0(k) \geq 1$. A straightforward estimation therefore gives 
	\[ 0\leq p_0(k)^{2j} -1 \leq C c\s j k (k+1) . \]
Hence 
\begin{align} 
\notag
&\sum_{j=1}^{n(s)}\E(u_{j,n(s)}(k)\otimes u_{j,n(t)}(k'))\\ 
&\q\q=2c\d_{kk'}\sum_{j=1}^{n(s)}p(k)^{n(s)+n(t)-2j}\left(\begin{matrix}1&0\\0&1\end{matrix}\right)
\to\d_{kk'}\int_{t-s}^{t+s}e^{-(1+(1-\eta)k)u}du\left(\begin{matrix}1&0\\0&1\end{matrix}\right).
\label{ts}
\end{align}
Now, for any $k,k'\ge0$ and $s,t\in[0,\infty)$ with $s\le t$, the following limit holds in probability as $c\to 0$, uniformly in $\s$ and $F$,
\begin{equation}
\label{AUE}
\sum_{j=1}^{n(s)} 
\left|\E(a_{j,n(s)}(k)\otimes a_{j,n(t)}(k')|\cF_{j-1})-\E(u_{j,n(s)}(k)\otimes u_{j,n(t)}(k'))\right|\to 0. 
\end{equation}
To see this, recall that by Proposition \ref{DFA} for all $m\in\N$ 
there exists a constant $C=C(\L,\eta , \ve ,m,T)<\infty$ such that, for $c\le1/C$ and
$\d_0$ defined as in the proof of Lemma \ref{le:neglect}, 
there exists an event $\O_0$ of probability at least $1-c^m$ on which, for all $n\le T/c$ and all $\th\in[0,2\pi)$, 
\[ 
|\tilde\Phi_n'(e^{\s+i\th})|\le\d_0\le1 ,
\]
and hence, by \eqref{h_uniform}, $|h_n(\th)-1|\le63\d_0$.
Then, on $\O_0$, for $c\le1/C$ and $t\le T$,
\[ 
\begin{split} 
\sum_{j=1}^{n(s)} 
&\left|\E(a_{j,n(s)}(k)\otimes a_{j,n(t)}(k')|\cF_{j-1})-\E(u_{j,n(s)}(k)\otimes u_{j,n(t)}(k'))\right|\\ 
&\le\frac{e^{-c(n(s)+n(t))}}c\sum_{j=1}^{n(s)}\left|\E(\hat{M}^0_j(k)\otimes\hat{M}^0_j(k')|\cF_{j-1})- 
	\E(\hat{M}^u_j(k)\otimes\hat{M}^u_j(k'))\right|\\ 
&\le4c\sum_{j=1}^{n(s)}e^{-c(n(s)+n(t)-2j)}
	\bigg|\fint_0^{2\pi}(e^{i\th(k+1)}\otimes e^{i\th(k'+1)})(h_j(\th)-1)d\th-\\ 
&\q\q\q\q\q\q\q\q\left(\fint_0^{2\pi}e^{i\th(k+1)}(h_j(\th)-1)d\th\right)\otimes\left(\fint_0^{2\pi}e^{i\th(k'+1)}(h_j(\th)-1)d\th\right)\bigg|\\ 
&\le Cc\d_0\sum_{j=1}^{n(s)}e^{-c(n(s)+n(t)-2j)}.
\end{split} 
\]
Since $c\d_0n(s)\to0$ as $c\to0$, this shows the claimed limit in probability.

\subsection{Convergence of Laurent coefficients}
We now show that the processes of rescaled Laurent coefficients $(\tilde A(\cdot,k))_{k\ge0}$ of $\tilde\cM^0(t,z)$ converge weakly to
those of the limiting process $\cF$. 

\begin{theorem}\label{th:Laurent}
Under the hypotheses of Theorem \ref{th:fluctuations_intro}, 
in the limit $c\to0$ and $\s\to0$ and uniformly in the basic map $F$, we have 
\[ 
\left(\tilde A(.,k)\right)_{k\ge0}\to\left(A(.,k)\right)_{k\ge0}
\]
in distribution in $(D[0,\infty )^{\Z^+},\mathbf{d}^{\Z^+})$. 
\end{theorem}
\begin{proof}
It will suffice to show that the finite-dimensional distributions of $(\tilde A(\cdot,k))_{k\ge0}$ converge to those of $(A(\cdot,k))_{k\ge0}$, 
and that for each fixed $k$ the processes $\tilde A(\cdot,k)$ are tight in $(D[0,\infty),\mathbf{d})$.

We start by proving convergence of finite-dimensional distributions. 
Fix positive integers $K$ and $m$ and pick arbitrary $0\le t_1<t_2<\dots<t_m$. 
We aim to show the following convergence in distribution
\begin{equation*} 
\label{fdd}
 \left(
 \begin{matrix} 
\tilde A(t_1,1) & \tilde A(t_1,2) & \cdots & \tilde A(t_1,K) \\
\vdots & \vdots & & \vdots \\
\tilde A(t_m ,1) & \tilde A(t_m,2) & \cdots & \tilde A(t_m,K)
 \end{matrix} 
 \right) 
\longrightarrow
 \left( 
 \begin{matrix} 
A(t_1,1) & A(t_1,2) & \cdots & A(t_1,K) \\
\vdots & \vdots & & \vdots \\
A(t_m,1) & A(t_m,2) & \cdots & A(t_m ,K)
\end{matrix} 
\right). 
\end{equation*}
Write $n_i$ in place of $n(t_i)$ for brevity. 
Fix real-linear maps $\a_{k,l}:\C\to\R$, for $k=1,\dots,K$ and $l=1,\dots,m$
and consider the real-valued random variables given by 
\[
X_{j,n_m}
=\sum_{k=1}^K\sum_{l=1}^m\a_{k,l}a_{j,n_l}(k)1_{\{j\le n_l\}}.  
\]
Then
$$
\sum_{k=1}^K\sum_{l=1}^m\a_{k,l}\tilde A(t_l,k)  
=\sum_{j=1}^{n_m}X_{j,n_m}.
$$
It is readily verified that $(X_{j,n_m}:j=1,\dots,n_m)$ is a martingale difference sequence with respect to the filtration 
$(\cF_j:j=1,\dots,n_m)$. 
Set
$$
\Sigma
=\sum_{k=1}^K\sum_{l,l'=1}^m\<\a_{k,l},\a_{k,l'}\>\int_{|t_l-t_{l'}|}^{t_l+t_{l'}}e^{-(1+(1-\eta)k)u}du 
$$
and note that $\Sigma$ is the variance of
$$
\sum_{k=1}^K\sum_{l=1}^m\a_{k,l} A(t_l,k).
$$
We will use the following martingale central limit theorem \cite[Theorem 18.1]{B}. 
\begin{theorem}
\label{ThCLT}
Suppose given, for each $n\in\N$, a martingale difference array $(X_{j,n}:j=1,\dots,n)$ with filtration $(\cF_{j,n}:j=1,\dots,n)$. 
Assume that, for some $\Sigma\in[0,\infty)$ and for all $\ve>0$, 
the following two conditions hold in the limit $n\to\infty$:
\begin{itemize}
\item[{\rm (i)}]\label{i} $\displaystyle \sum_{j=1}^n\E\left(X_{j,n}^2|\cF_{j-1,n}\right)\to\Sigma$ in probability, 
\item[{\rm (ii)}] 
$\displaystyle\sum_{j=1}^n\E(|X_{j,n}|^2\,1_{\{|X_{j,n}|>\ve\}})\to0$.
\end{itemize}
Then $\displaystyle\sum_{j=1}^nX_{j,n}\to\cN(0,\Sigma)$ in distribution as $n\to\infty$. 
\end{theorem}
We can apply this theorem to the limit $c\to0$ and the martingale difference array $(X_{j,n_m}:j=1,\dots,n_m)$, with 
$n_m=n(t_m)=\lfloor t_m/c\rfloor$.
We have
\[ 
\sum_{j=1}^{n_m}\E(X_{j,n_m}^2|\cF_{j-1}) 
=\sum_{k,k'=1}^K\sum_{l,l'=1}^m\left\<\a_{k,l}
\sum_{j=1}^{n_l\wedge n_{l'}}\E(a_{j,n_l}(k)\otimes a_{j,n_{l'}}(k')|\cF_{j-1}),\a_{k',l'}\right\>\to\Sigma
\]
in probability as $c\to0$ by \eqref{AUE} and \eqref{ts}, which proves (i).
To see (ii) note that 
\begin{equation*} 
\label{ajk}
|a_{j,n}(k)|\le4\sqrt c\quad\mbox{ for all }k\le K,j\le n, 
\end{equation*} 
from which, for arbitrary $\ve>0$ and a constant $C$ allowed to depend on the constants $\a_{k,l}$, $K$ and $m$, 
for all sufficiently small $c$,
\[ 
\sum_{j=1}^{n_m}\E(|X_{j,n_m}|^2\,1_{\{|X_{j,n_m}|>\ve\}}) 
\le Cc\sum_{j=1}^{n_m}\PP(|X_{j,n_m}|>\ve) 
\le Ct_m\PP\left(\max_{j\le n_m}|X_{j,n_m}|>\ve\right)=0. 
\]
Since the linear maps $\a_{k,l}$ were arbitrary, 
this shows convergence of the finite-dimensional distributions of $(\tilde A(t,k))_{k\ge0}$ to those of $(A(t,k))_{k\ge0}$. 

It remains to prove tightness. 
We will show that, for all $p\in[2,\infty)$, all $k\ge0$ and all $T\in[0,\infty)$, there is a constant $C=C(p,\eta,k,T)<\infty$
such that, for all $s,t\in[0,T]$,
\begin{equation} 
\label{UHC}
\limsup_{c,\s\to0}\|\tilde A(s,k)-\tilde A(t,k)\|_p\le C|t-s|^{1/2}.
\end{equation} 
Since we may choose $p>2$, this implies tightness, by a standard criterion.

Recall that
$$
\tilde A(t,k)=\sum_{j=1}^{n(t)}a_{j,n(t)}(k)=\frac1{\sqrt c}\sum_{j=1}^{n(t)}e^{-cj}p(k)^{n(t)-j}\hat M^0_j(k)
$$
and that $(\hat M^0_j(k):j\ge0)$ is a martingale difference sequence with $|\hat M^0_j(k)|\le2ce^{cj}$.
Also $0\le p(k)\le1$ and, estimating as above,
$$
1-p(k)^j\le C[\s k(k+1)+(1+(1-\eta)k)]cj.
$$
Fix $s,t\in[0,T]$ with $s\le t$ and note that $n(t)-n(s)\le1+(t-s)/c$.
Then
\begin{align*}
\tilde A(t,k)-\tilde A(s,k)
= & \ \frac1{\sqrt c}\sum_{j=1}^{n(s)}e^{-cj}p(k)^{n(s)-j}(p(k)^{n(t)-n(s)}-1)\hat M^0_j(k) \\
& \ +\frac1{\sqrt c}\sum_{j=n(s)+1}^{n(t)}e^{-cj}p(k)^{n(t)-j}\hat M^0_j(k)
\end{align*}
and so, by Burkholder's inequality, for some constant $C=C(p,\eta,k,T)<\infty$,
\begin{equation}
\label{AHE}
\|\tilde A(t,k)-\tilde A(s,k)\|_p^2\le C\left(\s^2k^2(k+1)^2+(1+(1-\eta)k)^2(t-s+c)^2+t-s+c\right).
\end{equation}
The asymptotic H\"older condition \eqref{UHC} follows.
\end{proof}

\subsection{Convergence as an analytic function}\label{sec:Hfluct}
In this section we deduce the convergence of $\tilde\cM^0(t,z)$ from that of the Laurent coefficients, 
thus concluding the proof of Theorem \ref{th:fluctuations_intro}. 
To this end, set 
\[ 
\tilde\cF(t,z)
=\frac{\tilde\cM^0(t,z)}{\sqrt c} 
=\sum_{k\ge0}\tilde A(t,k)z^{-k},\q
\cF(t,z)
=\sum_{k\ge0}A(t,z)z^{-k}. 
\]
These define processes in $D_\cH[0,\infty)$. 
For any $T>0$ let $D_\cH[0,T]$ denote the space of $\cH$-valued \cadlag processes on $[0,T]$. 
Then $\tilde\cF,\cF$ define processes in $D_\cH [0,T]$ by restriction, for all $T>0$.  
For any $r>1$ let $\cH_r$ denote the space of analytic functions on $\{|z|\ge r\}$ with limits at $\infty$, 
equipped  with the metric 
\[ 
d_{r}(f,g)=\sup_{|z|\ge r}|f(z)-g(z)|. 
\]
We let $D_{\cH_r}[0,T]$ denote the space of \cadlag processes with values in $\cH_r$ equipped with the associated Skorohod metric $\mathbf{d}_{T,r}$. 
To show that $\tilde\cF$ converges to $\cF$ in distribution on $(D_\cH[0,\infty),\mathbf{d}_\cH)$,
it suffices to show that, for any $T>0$ and $r>1$, 
the process $\tilde\cF$ converge to $\cF$ in distribution on $(D_{\cH_r}[0,T],\mathbf{d}_{T,r})$ as $c\to 0$ 
(see Billingsley \cite{B}). 
This in turn follows from the lemma below. 

\begin{lemma}\label{le:Kn}
For any $T>0$, $r>1$ and $\d=\d(r) \in [0,1]$ such that $e^{-2\d}r>1$, we have that for any $\ve>0$
\[ 
\lim_{K\to\infty}\sup_{c\in(0,\d]}\PP\left(\mathbf{d}_{T,r}\bigg(\sum_{k=K}^\infty\tilde A(.,k)z^{-k},0\bigg)>\ve\right)=0. 
\]
\end{lemma}
\begin{proof}
Fix $\ve,T,r,\d$ as in the statement, 
and partition the interval $[0,T]$ into sub-intervals $I_l=[(l-1)\d,l\d)$ for $1\le l\le \lceil T/\d\rceil$. 
Then  
\[ 
\mathbf{d}_{T,r} 
\bigg(\sum_{k=K}^\infty\tilde A(.,k)z^{-k},0\bigg) 
\le\sum_{k\ge K}\sup_{t\in[0,T]}|\tilde A(t,k)|r^{-k} 	
\]
and so 
\[
\PP\left(\mathbf{d}_{T,r} 
\bigg(\sum_{k=K}^\infty\tilde A(.,k)z^{-k},0\bigg)>\ve\right) 
\le\frac1\ve\sum_{l=1}^{\lceil T/\d\rceil}\sum_{k\ge K}\E\bigg(\sup_{t\in I_l}|\tilde A(t,k)|^2\bigg)^{1/2}r^{-k}.
\]
Recall that 
\[ 
\tilde A(t,k)=\frac1{\sqrt c}\sum_{j=1}^{n(t)}e^{-cj}p(k)^{n(t)-j}\hat M^0_j(k)
\]
which shows that the process $(p(k)^{-n(t)}\tilde{A}(t,k))_{t\ge0}$ is a martingale for each $k\ge0$, with 
\[
\E\big(|p(k)^{-n(t)}\tilde{A}(t,k)|^2\big) 
\le\frac Cc\sum_{j=1}^{n(t)}e^{-2cj}p(k)^{-2j}\E(|\hat M^0_j(k)|^2|\cF_{j-1}) 
\le16Cc\sum_{j=1}^{n(t)}p(k)^{-2j}.
\]
Doob's $L^2$ inequality then gives 
\begin{align}
\label{Il}
\E\bigg(\sup_{t\in I_l}|\tilde A(t,k)|^2\bigg)
&\le p(k)^{2n((l-1)\d)}\E\bigg(\sup_{t\in I_l}|p(k)^{-n(t)}\tilde A(t,k)|^2 \bigg) \nonumber
\\ 
&\le4p(k)^{2n((l-1)\d)}\E\big(| p(k)^{-n(l\d )}\tilde A(l\d,k)|^2\big) \nonumber \\ 
&\le Ccp(k)^{2n((l-1)\d)}\sum_{j=1}^{n(l\d)}p(k)^{-2j} \nonumber \\ 
&\le Cp(k)^{-2(n(l\d)-n((l-1)\d)} ,
\end{align}
for some positive constant $C$, depending on $T$, changing from line to line. 
In the last inequality we have used that $p(k)\le1$ and $cn(l\d)\le T+1$. 
Noting that $n(l\d)-n((l-1)\d)\le1+\d/c$, and that $p(k)\ge e^{-c(k+1)}$ for $\eta\in[0,1]$, we find
\[ 
p(k)^{-2(n(l\d)-n((l-1)\d)}\le p(k)^{-2(1+\d/c)}\le e^{4\d(k+1)} 
\] 
for $\d\ge c$. 
Plugging this into \eqref{Il} gives 
\[ 
\E\bigg(\sup_{t\in I_l}|\tilde A(t,k)|^2\bigg) 
\le C e^{4\d k} , 
\]
and hence 
\[ 
\sup_{c\in(0,\d]}\frac1\ve\sum_{l=1}^{\lfloor T/\d \rfloor} 
\sum_{k\ge K}\E\bigg(\sup_{t\in I_l}|\tilde{A}(t,k)|^2\bigg)^{1/2}r^{-k}
\le\frac{C}{\ve\d}\sum_{k\ge K}(e^{-2\d}r)^{-k} 
\longrightarrow0 
\]
as $K\to\infty$ since $e^{-2\d}r>1$. 
If $\eta<0$, the result follows from the same argument using that,  for $c$ small enough that $\s-c-c|\eta|>0$, we have 
\[
p(k) =e^{-c(k+1)}(1+c\eta ke^{-(\s-c)(k+1)}) 
\ge e^{-c(k+1)-c|\eta|k}. 
\]
\end{proof}

\def\j{
\bb
\subsubsection{Showing that Lemma \ref{le:Kn} is enough}
Before proving this result, let us explain how it can be used to conclude. 
To show convergence in distribution on $(D_{\cH_r}[0,T],\mathbf{d}_{T,r})$,
it suffices to show that for all closed sets $\cC$ in this space it holds 
\[ 
\limsup_{c\to 0}\PP(\tilde\cF\in\cC)\le\PP(\cF\in\cC). 
\]
For $\ve>0$, let $\cC_\ve=\{g\in D_{\cH_r}[0,T]:\mathbf{d}_{T,r}(g,\cC)\le\ve\}$ denote the $\ve$-enlargement of $\cC$. 
Then, by Lemma \ref{le:Kn}, for any fixed $\ve>0$ there exists an integer $K(\ve)<\infty$ such that, for all $K\ge K(\ve)$,
we have 
\[ 
\sup_{c\in(0,\d]}\PP\left(\mathbf{d}_{T,r}\bigg(\sum_{k=K}^\infty\tilde A(t,k)z^{-k},0\bigg)>\ve\right)\le\ve. 
\]
At cost of increasing $K(\ve)$, we can assume that the same holds true with $A(t,k)$ replacing $\tilde A(t,k)$ for all $k$. 
It follows that, for all $K\ge K(\ve)$,
\[ 
\PP(\tilde\cF\in\cC) 
\leq\ve+\PP\left(\sum_{k=0}^{K-1}\tilde A(t,k)z^{-k}\in\cC_\ve\right). 
\]
Moreover, by convergence in distribution of the sequence $(\tilde A(t,k))_{k\ge0}$, for any fixed $K$ we have 
\[ 
\sum_{k=1}^{K-1}\tilde A(t,k)z^{-k}\to\sum_{k=1}^{K-1}A(t,k)z^{-k} 
\]
in distribution on $(D_{\cH_r}[0,T],\mathbf{d}_{T,r})$ as $c\to 0$. 
It follows that 
\[
\begin{split} 
\limsup_{c\to0}\PP(\tilde\cF\in\cC)
&\le\ve+\PP\left(\sum_{k=0}^{K-1}A(t,k)z^{-k}\in\cC_\ve\right)\\ 
&\le2\ve+\PP\left(\sum_{k=0}^{K-1}A(t,k)z^{-k}\in\cC_\ve,\,\mathbf{d}_{T,r}\bigg(\sum_{k\ge K}A(t,k)z^{-k},0\bigg)\le\ve\right)\\ 
&\le2\ve+\PP\left(\cF\in\cC_{2\ve}\right) 
\end{split} 
\]
which gives the desired conclusion by sending $\ve\to0$. 
\eb 
}\jj

\begin{appendices}
\section*{Appendices}

\section{Particle estimates} 
\label{sec:particle}
Let $c\in(0,\infty)$ and $\L\in[0,\infty)$.
Recall that we say a univalent function $F$ from $D_0=\{|z|>1\}$ into $D_0$ has capacity $c$ and regularity $\L$ 
if it satisfies condition \eqref{partcond}, that is to say, for all $z\in D_0$,
$$
\left|\log\left(\frac{F(z)}z\right)-c\frac{z+1}{z-1}\right|\le\frac{\L c^{3/2}|z|}{|z-1|(|z|-1)}.
$$
We show that this in fact implies a similar condition for $F$ but with better decay as $z\to\infty$.
Then we will give some explicit examples of suitable maps $F$.
Finally, we will show that \eqref{partcond} holds whenever the corresponding particle is not too flat.
Only Subsection \ref{PPH} is used in the paper.

\subsection{Precise form of the particle hypothesis}
\label{PPH}
Our particle hypothesis \eqref{partcond} can be reformulated more precisely in terms of the coefficient
$a_0$ in the Laurent expansion \eqref{LAUR}.

\begin{proposition}
\label{PPC}
Suppose that $F$ satisfies \eqref{partcond} and set $\b=a_0/(2c)$.
Then $|\b-1|\le\L\sqrt c/2$ and, for all $z\in D_0$,
\begin{equation}
\label{PREC}
\left|\log\left(\frac{F(z)}z\right)-c-\frac{2c\b}{z-1}\right|\le\frac{6\L c^{3/2}}{|z-1|(|z|-1)}.
\end{equation}
\end{proposition}
\begin{proof}
Set
$$
f(z)=\log\left(\frac{F(z)}z\right),\q
g(z)=(z-1)\left(f(z)-c-\frac{2c}{z-1}\right).
$$
Then $g$ is analytic in $D_0$ and $g(z)\to a_0-2c=2c(\b-1)$ as $z\to\infty$.
Condition \eqref{partcond} implies
$$
|g(z)|\le\L c^{3/2}\frac{|z|}{|z|-1}.
$$
On letting $z\to\infty$, we see that $2c|\b-1|\le\L c^{3/2}$ so $|\b-1|\le\L\sqrt c/2$.
Consider
$$
h(z)=z(g(z)-g(\infty))=z(z-1)
\left(f(z)-c-\frac{2c\b}{z-1}\right).
$$
Then $h$ is analytic in $D_0$ and bounded at $\infty$.
We have
\begin{equation}
\label{FBR}
\left|f(z)-c-\frac{2c\b}{z-1}\right|
\le\frac{|g(z)|+|g(\infty)|}{|z-1|}
\le\L c^{3/2}\frac{2|z|-1}{|z-1|(|z|-1)}
\end{equation}
so
$$
|h(z)|\le\L c^{3/2}\frac{|z|(2|z|-1)}{|z|-1}=6\L c^{3/2}
$$ 
whenever $|z|=2$.
Then, by the maximum principle, for all $|z|\ge2$, we have $|h(z)|\le6\L c^{3/2}$ and hence
$$
\left|f(z)-c-\frac{2c\b}{z-1}\right|
\le\frac{6\L c^{3/2}}{|z-1|(|z|-1)}.
$$
On the other hand \eqref{FBR} implies the same inequality for $1<|z|<2$.
\end{proof}

Note that \eqref{PREC} with $|\b-1|\le\L\sqrt c/2$ implies \eqref{partcond} with $\L$ replaced by $7\L$.
Thus the two conditions are equivalent up to adjustment of the constant by a universal factor.

\subsection{Spread out particles}
Consider for $\g\in\C$ the map on $D_0$ given by
$$
F(z)
=F_{c,\g}(z)
=z\exp\left(c\frac{\g z+1}{\g z-1}\right)
=e^cz\exp\left(\frac{2c}{\g z-1}\right).
$$
It is straightforward to check that $F_{c,\g}$ is univalent into $D_0$ if and only if
\begin{equation*}
\label{GAM}
|\g|\ge\g(c)=1+c+\sqrt{2c+c^2}.
\end{equation*}
Then $F_{c,\g}$ has capacity $c$ and, since
$$
\log\left(\frac{F_{c,\g}(z)}z\right)
=c\frac{\g z+1}{\g z-1}
$$
and
$$
\left|\frac{\g z+1}{\g z-1}-\frac{z+1}{z-1}\right|
=\frac{2|\g-1||z|}{|z-1||\g z-1|}
$$
we see that $F_{c,\g}$ has regularity $\L=2|\g-1|/\sqrt c$.
The corresponding particles $P_{c,\g}$ are spread all around the unit circle, 
as illustrated in the rightmost particle in Figure 1.
When $\g=\g(c)$ we find $F'(1)=0$ so $P_{c,\g(c)}$ has the form of a cusp with endpoint $F(1)$.
Moreover, in the limit $c\to0$ with $\g=\g(c)$, the regularity constant $\L$ stays bounded and $\log F(1)\sim\sqrt{2c}$, 
so the endpoint lies at distance $F(1)-1\sim\sqrt{2c}$ from the unit circle.

 
\subsection{Small particles of a fixed shape}
The following proposition shows that our condition \eqref{partcond} holds generically for particles attached near $1$
which are not too flat.
In particular, it shows that, 
for particles of a fixed shape, such as slits or disks, attached to the unit circle at $1$, in the small diameter limit $\d\to0$,
the capacity $c\to0$ while the regularity constant $\L$ stays bounded,
which is the regime in which our limit theorems apply.

\begin{proposition}
\label{PEST}
There is a constant $C<\infty$ with the following property.
Let $P$ be a basic particle such that, for some $\d_0,\d\in(0,1]$,
\begin{itemize}
\item[{\rm (a)}]$|z|=1+\d_0$ for some $z\in P$,
\item[{\rm (b)}]$|z-1|\le\d$ for all $z\in P$.
\end{itemize}
Then $P$ has capacity $c$ satisfying
$
\d_0^2/C\le c\le C\d^2.
$
Moreover, if $\d\le1/C$, then $P$ has regularity 
$
\L\le C\d/\d_0.
$
\end{proposition}
\begin{proof}
The bounds on $c$ are well known. 
The lower bound relies on Beurling's projection theorem and a comparison with the case of a slit particle.
The upper bound follows from a comparison with the case $P_\d=S_\d\cap D_0$,
where $S_\d$ is the closed disk whose boundary intersects the unit circle orthogonally at $e^{\pm i\th_\d}$
with $\th_\d\in[0,\pi]$ is determined by $|e^{i\th_\d}-1|=\d$.
See Pommerenke \cite{MR1217706}.

We turn to the bound on $\L$.
First we will show, for $a=15\d\le\pi$, we have
\begin{equation}
\label{USCO}
|F(e^{i\th})|=1\q\text{whenever $|\th|\in[a,\pi]$}.
\end{equation}
Then we will show that, if $c\in(0,1]$ and \eqref{USCO} holds with $a\in(0,\pi/2)$, then, for all $|z|>1$, 
\begin{equation}
\label{USCT}
\left|\log\left(\frac{F(z)}z\right)-c\frac{z+1}{z-1}\right|
\le\frac{76ac|z|}{|z-1|(|z|-1)}.
\end{equation}
The desired bound on $\L$ then follows from \eqref{USCO} and \eqref{USCT} and the lower bound on $c$.

We can write
$$
\log\left(\frac{F(z)}{z}\right)=u(z)+iv(z)
$$
where $u$ and $v$ are harmonic functions in $D$ with $u(z)\to c$ and $v(z)\to0$ as $z\to\infty$.
Since $F$ maps into $D_0$, we have $u(e^{i\th})\ge0$ for all $\th \in [0,2\pi )$.
We have to show that $u(e^{i\th})=0$ whenever $|\th|\in[a,\pi]$.
Set
$$
p_\d=\PP_\infty(B\text{ hits }S_\d\text{ before leaving }D_0)
$$
where $B$ is a complex Brownian motion.
Consider the conformal map $f$ of $D_0$ to the upper half-plane $\H$ given by 
$$
f(z)=i\frac{z-1}{z+1}.
$$
Set $b=f(e^{-i\th_\d})=\sin\th_\d/(1+\cos\th_\d)$.
Since $\d\le1$, we have $\th_\d\le\d\pi/3$ and then $b\le2\pi\d/9$.
By conformal invariance,
$$
p_\d
=\PP_i(B\text{ hits $f(S_\d)$ before leaving $\H$})
=2\int_0^{2b/(1-b^2)}\frac{dx}{\pi(1+x^2)}.
$$
Hence $p_\d\le4b/\pi\le8\d/9$.

Now $e^{i\pi}$ is not a limit point of $P$ so $e^{i\pi}=F(e^{i(\pi+\a)})$ for some $\a\in\R$.
Then $u(e^{i(\pi+\a)})=0$ and we can and do choose $\a$ so that $\a+v(e^{i(\pi+\a)})=0$.
Set
$$
\th^+=\sup\{\th\le\pi+\a:u(e^{i\th})>0\},\q
\th^-=\inf\{\th\ge\pi+\a:u(e^{i\th})>0\}-2\pi.
$$
Then $\th^-\le\th^+$.
It will suffice to show that $|\th^\pm|\le15\d$.
For $\th\in[\th^-,\th^+]$, we have $F(e^{i\th})\in S_\d$ so $|\th+v(e^{i\th})|\le\th_\d$.
Set $P^*=\{F(e^{i\th}):\th\in[\th^-,\th^+]\}$. 
Then $P^*\sse S_\d$ so, by conformal invariance,
$$
\frac{\th^+-\th^-}{2\pi}
=\PP_\infty(B\text{ hits $P^*$ on leaving $D_0\sm P$})
\le p_\d.
$$
On the other hand, for $\th,\th'\in[\th^+,\th^-+2\pi]$ with $\th\le\th'$, by conformal invariance,
$$
\frac{\th'-\th}{2\pi}
=\PP_\infty\big(B\text{ hits $\big[e^{i(\th+v(e^{i\th}))},e^{i(\th'+v(e^{i\th'}))}\big]$ on leaving $D_0\sm P$}\big)
\le\frac{\th'+v(e^{i\th'})-\th-v(e^{i\th})}{2\pi}
$$
so $v$ is non-decreasing on $[\th^+,\th^-+2\pi]$, and so
$$
\a+v(e^{i\th^+})
\le\a+v(e^{i(\pi+\a)})=0
\le\a+v(e^{i\th^-}).
$$
Hence
$$
\th^+-\a
\le2\pi p_\d+\th^--\a
\le2\pi p_\d+\th_\d-v(e^{i\th^-})-\a
\le2\pi p_\d+\th_\d
$$
and similarly $\th^--\a\ge-2\pi p_\d-\th_\d$.
So we obtain, for all $\th\in[\th^-,\th^+]$,
$$
|\a+v(e^{i\th})|\le2\th_\d+2\pi p_\d.
$$
Since $v$ is continuous and is non-decreasing on the complementary interval, this inequality then holds for all $\th$.
Now $v$ is bounded and harmonic in $D_0$ with limit $0$ at $\infty$,
so
$$
\int_0^{2\pi}v(e^{i\th})d\th=0.
$$
Hence
$$
|\a|
=\left|\fint_0^{2\pi}(\a+v(e^{i\th}))d\th\right|
\le2\th_\d+2\pi p_\d
$$
and so $|\th^\pm|\le3\th_\d+4\pi p_\d\le41\pi\d/9\le15\d$, as required.

We turn to the proof of \eqref{USCT}.
Assume now that $u(e^{i\th})=0$ whenever $|\th|\in[a,\pi]$.
Since $u$ is harmonic, we have
$$
\fint_0^{2\pi}u(e^{i\th})d\th=c
$$
and, for all $|z|>1$,
$$
u(z)
=\fint_0^{2\pi}u(e^{i\th})\re\left(\frac{z+e^{i\th}}{z-e^{i\th}}\right)d\th
=c+\fint_0^{2\pi}u(e^{i\th})\re\left(\frac{2e^{i\th}}{z-e^{i\th}}\right)d\th.
$$
Hence, using that $u(e^{i\th}) \geq 0$ for all $\th \in [0,2\pi )$,   
$$
|u(z)-c|\le\frac2{|z|-1}\fint_0^{2\pi}u(e^{i\th})d\th=\frac{2c}{|z|-1}
$$
and, since $v(z)\to0$ as $z\to\infty$, a standard argument using the Cauchy--Riemann equations then shows that 
$$
\left|\log\left(\frac{F(z)}z\right)-c\frac{z+1}{z-1}\right|\le\frac{24c}{|z|-1}.
$$
This gives the claimed estimate in the case where $|z-1|\le2a$.
It remains to consider the case where $|z-1|>2a$.
Let $\alpha , \rho $ be defined by 
$$
\fint_0^{2\pi}u(e^{i\th})e^{i\th}d\th=c\rho e^{i\a}.
$$
Then $|\a|\le a$ and $\rho\in[\cos a,1)$.
Now
$$
u(e^{i\a}z)-c-\re\left(\frac{2\rho c}{z-1}\right)
=\fint_0^{2\pi}u(e^{i(\th+\a)})\re\left(\frac{2e^{i\th}}{z-e^{i\th}}-\frac{2e^{i\th}}{z-1}\right)d\th
$$
so
$$
\left|u(e^{i\a}z)-c-\re\left(\frac{2\rho c}{z-1}\right)\right|
\le\frac2{|z-1|(|z|-1)}\fint_0^{2\pi}u(e^{i(\th+\a)})|e^{i\th}-1|d\th
\le\frac{4ac}{|z-1|(|z|-1)}.
$$
The standard argument mentioned above now allows us to deduce that
$$
\left|v(e^{i\a}z)-\im\left(\frac{2\rho c}{z-1}\right)\right|
\le\frac{27}\pi\frac{4ac}{|z-1|(|z|-1)}.
$$
Hence, by a simple calculation,
\begin{equation}\label{LOGFF}
\left|\log\left(\frac{F(z)}z\right)-c-\frac{2\rho c}{e^{-i\a}z-1}\right|
\le\frac{35ac}{|e^{-i\a}z-1|(|z|-1)}.
\end{equation}
Note that
$$
\left|\frac{z+1}{z-1}-1-\frac{2\rho}{e^{-i\a}z-1}\right|
\le\frac{2(1-\rho+|\rho e^{i\a}-1||z|)}{|z-1|(|z|-1)}
\le\frac{6a|z|}{|z-1|(|z|-1)}.
$$
Since $|z-1|>2a$, we have $|e^{-i\a}z-1|\ge|z-1|/2$, so we can deduce from \eqref{LOGFF} that
$$
\left|\log\left(\frac{F(z)}z\right)-c\frac{z+1}{z-1}\right|
\le\frac{76ac|z|}{|z-1|(|z|-1)}.
$$
\end{proof}

\section{Preliminary estimates}\label{sec:preliminaries}
In this section, we gather together some standard results which are used in our proofs.
\subsection{Martingale estimates}\label{sec:martingales}
We recall the following martingale inequality, due to Burkholder. 
\begin{theorem}[\cite{MR365692}, Theorem 21.1]\label{thB}
Let $(X_n)_{n\geq 0}$ be a martingale with respect to the filtration $(\cF_n)_{n\geq 0}$. For $n\geq 1 $ write $\Delta_n = X_n - X_{n-1}$ for the increment process, and define 
	\[ \Delta_n^* = \max_{0\leq k \leq n } |\Delta_k| , 
	\qquad 
	Q_n = \sum_{k=1}^n \E (|\Delta_k  |^2 | \cF_{k-1} ) . 
	\] 
Then for any $p \in [2,\infty ) $ there exists a constant $C=C(p)$ such that for all $n\geq 1$ 
	\[ \Big\| \max_{k\leq n } |X_k |  \Big\|_p ^2 \leq C \big(  \| Q_n \|_{p/2}
	+ \| \D_n ^* \|_p ^2 \big)  . \]
\end{theorem}

\subsection{Operator estimates}\label{sec:operators}
We note some $L^p$-estimates on operators which act on the set of analytic functions $f$ on $\{|z|>1\}$ which are bounded at $\infty$, and hence have a Laurent expansion
$$
f(z)=\sum_{k=0}^\infty f_kz^{-k}.
$$
Firstly, for the operator $Df(z)=zf'(z)$, by a standard argument using Cauchy's integral formula,
there is an absolute constant $C<\infty$ such that,
for all $p\in\N$ and $1<\rho<r$,
\begin{equation}\label{DEST}
\|Df\|_{p,r}\le\frac{C\rho}{r-\rho}\|f\|_{p,\rho}.
\end{equation}
Secondly, let $L$ be an operator which acts as multiplication by $m_k$ on the the $k$th Laurent coefficient.
Thus
$$
Lf(z)=\sum_{k=0}^\infty m_k f_kz^{-k}.
$$
Assume that there exists a finite constant $M>0$ such that, for all $k\ge0$, 
$$
|m_k|\le M
$$
and, for all integers $K\ge0$, 
$$
\sum_{k=2^K}^{2^{K+1}-1}|m_{k+1}-m_k|\le M.
$$
The Marcinkiewicz multiplier theorem \cite[Vol. II, Theorem 4.14]{MR933759} then asserts that, for all $p\in(1,\infty)$, 
there is a constant $C=C(p)<\infty$ such that, for all $r>1$,
$$
\|Lf\|_{p,r}\le CM\|f\|_{p,r}.
$$
We will use also the following estimate.
Write $\|L\|_{p,\rho\ra r}$ for the smallest constant $K$ such that
$$
\|Lf\|_{p,r}\le K\|f\|_{p,\rho}
$$
for all analytic functions $f$ on $\{|z|>1\}$ bounded at $\infty$.

\begin{proposition}
Let $f$ and $g$ be analytic in $\{|z|>1\}$ and bounded at $\infty$.
Set $f_\th(z)=f(e^{-i\th}z)$.
Let $L$ be a multiplier operator and define
$$
h(z)=\fint_0^{2\pi}|L(f_\th.g)|^2(z)d\th.
$$
Then, for all $r,\rho>1$, we have
\begin{equation}\label{LFG}
\|h\|_{p/2,r}\le\|L\|^2_{p,\rho\ra r}\|g\|^2_{p,\rho}\|f\|^2_{2,\rho}.
\end{equation}
\end{proposition}
\begin{proof}
We can write
$$
f(z)=\sum_{k=0}^\infty f_kz^{-k},\q
g(z)=\sum_{k=0}^\infty g_kz^{-k},\q
Lf(z)=\sum_{k=0}^\infty m_k f_kz^{-k}.
$$
Then
$$
L(f_\th.g)(z)=\sum_{k=0}^\infty\sum_{j=0}^\infty m_{j+k}f_kg_je^{i\th k}z^{-(k+j)}
$$
so
$$
h(z)=\sum_{k=0}^\infty|f_k|^2|L(\tau_kg)(z)|^2
$$
where $\tau_kg(z)=z^{-k}g(z)$.
Hence
\begin{align*}
\|h\|_{p/2,r}
&\le\sum_{k=0}^\infty|f_k|^2\|L(\tau_kg)\|^2_{p,r}
\le\sum_{k=0}^\infty|f_k|^2\|L\|^2_{p,\rho\ra r}\|\tau_kg\|^2_{p,\rho}\\
&=\sum_{k=0}^\infty|f_k|^2\rho^{-2k}\|L\|^2_{p,\rho\ra r}\|g\|^2_{p,\rho}
=\|L\|^2_{p,\rho\ra r}\|f\|^2_{2,\rho}\|g\|^2_{p,\rho}.
\end{align*}
\end{proof}

\def\j{
\bb
\subsection{Bound on $D$ from Cauchy's integral formula}
We have
$$
|f'(re^{it})|\le\fint_0^{2\pi}\frac{|f(\rho e^{i(\th+t)})|}{|\rho e^{i\th}-r|^2}\rho d\th
$$
so
\begin{align*}
\fint_0^{2\pi}|f'(re^{it})|^pdt
&\le\rho^p
\fint_0^{2\pi}
\dots
\fint_0^{2\pi}
\prod_{m=1}^p
\frac1{|\rho e^{i\th_m}-r|^2}
\fint_0^{2\pi}
\prod_{m=1}^p
|f(\rho e^{i(\th_m+t)})|dtd\th_1\dots d\th_p\\
&\le\rho^p\|f\|_{p,\rho}^p
\fint_0^{2\pi}
\dots
\fint_0^{2\pi}
\prod_{m=1}^p
\frac1{|\rho e^{i\th_m}-r|^2}
d\th_1\dots d\th_p
\le\left(\frac{14\rho\|f\|_{p,\rho}}{r(r-\rho)}\right)^p.
\end{align*}
\eb
}\jj

\section{Computations of first order and error estimates}
\label{sec:comp_est}
In this section, we provide the detailed calculations behind the estimates stated in Section \ref{sec:estimates}. We only explicitly state estimates for $\eta \in [0,1]$, taking advantage of the fact that certain constants can be chosen uniformly over such values of $\eta$. Similar estimates hold for $\eta \in (-\infty , 0)$, and we leave the necessary adjustments to the reader. Furthermore, throughout this section we assume that $c, \s  \leq 1$. This assumption can be relaxed at the cost of the absolute constants. 

\subsection{Estimates on the attachment measure $h_n(\th)$}
We begin by obtaining estimates on $h_n(\th)$, defined in \eqref{hn}.
By an elementary calculation, for all $\eta\in (-\infty,1]$ and $w\in\C\sm\{0\}$, we can write
$$
|w|^{-\eta}=1-\eta\re(w-1)+\ve_1(w)
$$
with
$$
|\ve_1(w)|\le C (|w|^{-\eta}\vee1)|w-1|^2
$$
for some constant $C < \infty$ depending only on $\eta$. We will see below that $C \leq 24$ for all $\eta \in [0,1]$; the case $\eta<0$ requires minor adjustments to take into account the dependence of $C$ on $\eta$, which we leave to the reader. 
Take 
$$
w=e^{-c(n-1)}\Phi_{n-1}'(e^{\s+i\th})=\tilde\Phi_{n-1}'(e^{\s+i\th})+1
$$ 
to obtain
$$
e^{c(n-1)\eta}|\Phi_{n-1}'(e^{\s+i\th})|^{-\eta}
=1-\eta\re\tilde\Phi_{n-1}'(e^{\s+i\th})+\ve_2(\th)
$$
where
$$
\ve_2(\th)
=\ve_1(w).
$$
(Here and throughout the remainder of this section, $n$ is fixed and the dependence of error terms on $n$ is suppressed in the notation).
Then
\begin{equation}
\label{e3e4}
e^{c(n-1)\eta}Z_{n}
=\fint_0^{2\pi}e^{c(n-1)\eta}|\Phi_{n-1}'(e^{\s+i\th})|^{-\eta}d\th
=1+\ve_3
=\frac1{1+\ve_4}
\end{equation}
where
$$
\ve_3
=\fint_0^{2\pi}\ve_2(\th)d\th,
\q
\ve_4=-\frac{\ve_3}{1+\ve_3}.
$$
Here we used the fact that 
$$
\int_0^{2\pi}\re\tilde\Phi'_{n-1}(e^{\s+i\th})d\th=0
$$
which holds because $\tilde\Phi'_{n-1}(z)$ is analytic in $\{|z|>1\}$ and vanishes as $z\to\infty$.
Hence
\begin{equation}\label{HN}
h_{n}(\th)=1-\eta\re\tilde\Phi'_{n-1}(e^{\s+i\th})+\ve_5(\th)
\end{equation}
where
$$
\ve_5(\th)
=\ve_2(\th)
+(1-\eta\re\tilde\Phi'_{n-1}(e^{\s+i\th}))\ve_4+\ve_2(\th)\ve_4.
$$
Recall the definition of $N_0$ from \eqref{eq:N0_def}. Then, for all $n\le N_0$ and all $\th\in[0,2\pi)$,
$$
|\tilde\Phi'_{n-1}(e^{\s+i\th})|\le\frac18,\q 
\frac78 \le e^{-c(n-1)}|\Phi'_{n-1}(e^{\s+i\th})|\le\frac98
$$
so
$$
|\ve_2(\th)|\le
\frac{192}7
|\tilde\Phi'_{n-1}(e^{\s+i\th})|^2
$$
and
$$
|\ve_3|
=\left|\fint_0^{2\pi}\ve_2(\th)d\th\right|
\le\frac{192}7\|\tilde\Phi'_{n-1}\|^2_{2,e^\s}
\le\frac37.
$$
Using \eqref{e3e4} to bound $|1+\ve_3|$ directly,
$$
|\ve_4|\le\frac98|\ve_3|\le \frac{216}{7}\|\tilde\Phi'_{n-1}\|^2_{2,e^\s}
$$
and
\begin{equation}
\label{eq:e5_est}
|\ve_5(\th)|
\le42|\tilde\Phi'_{n-1}(e^{\s+i\th})|^2+35\|\tilde\Phi'_{n-1}\|_{2,e^\s}^2
\le 77 \delta_0^2.
\end{equation}

\subsection{Estimates on the increment $\Delta_n(\th,z)$}

We now move to analysing the increment $\Delta_n(\th,z)$, defined in \eqref{eq:Delta_def}. Recall from \eqref{eq:wn_def} that
$$
\Delta_n(\th,z) = m_n(\th,z) + w_n(\th,z) 
$$
where $m_n$ is defined in\eqref{eq:mn_def}. By \eqref{PFMB} we can write
\begin{align*}
w_n(\th,z)&=\ve_6(\th,z)+\ve_7(\th,z)
\end{align*}
where
\[\begin{split} 
\ve_6(\th,z)
&=\left(\log\frac{F(e^{-i\th}z)}{e^{-i\th}z}-c\right)
\int_0^1(D\Phi_{n-1}(F_{s,\th}(z))-e^{cn}z)ds\\
& = \left(\log\frac{F(e^{-i\th}z)}{e^{-i\th}z}-c\right)
\int_0^1 e^{c(n-1)} \big( \Psi_{n-1} (F_{s,\th}(z)) + F_{s,\th} (z) -e^c z \big) ds , \\
\ve_7(\th,z)
&=e^{cn}z\left(\log\frac{F(e^{-i\th}z)}{e^{-i\th}z}-c-\frac{2c\b}{e^{-i\th}z-1}\right). 
\end{split} \]
Note that $\ve_6(\th,\infty)=\ve_7(\th,\infty)=0$ and
\begin{equation}\label{E7B}
|\ve_7(\th,z)|\le\frac{6\L e^{cn}c^{3/2}|z|}{|e^{-i\th}z-1|(|z|-1)}. 
\end{equation}
By some straightforward estimation, we obtain a constant $C=C(\L)<\infty$ such that, for all $c\in(0,1]$ and all $|z|>1$,
$$
|z|\le|F_{s,\th}(z)|\le e^{C\sqrt c}|z|
$$
and, for $|z|\ge1+\sqrt c$,
\begin{equation}\label{GTHB}
|F_{s,\th}(z)-e^cz|\le\frac{Cc|z|}{|e^{-i\th}z-1|}. 
\end{equation}
Hence, for $|z|\ge1+\sqrt c$, 
\begin{equation}\label{E6B}
|\ve_6(\th,z)|\le\frac{Cce^{cn}}{|e^{-i\th}z-1|}\int_0^1|\Psi_{n-1}(F_{s,\th}(z))|ds+\frac{Cc^2e^{cn}|z|}{|e^{-i\th}z-1|^2}.
\end{equation}
We combine \eqref{HN} and \eqref{PFMB} to obtain
$$
A_n(z)
=\fint_0^{2\pi}
\bigg(-\eta\re\tilde\Phi'_{n-1}(e^{\s+i\th})+\ve_5(\th)\bigg)
\left(\frac{2c\b e^{cn}z}{ze^{-i\th}-1}+\ve_6(\th,z)+\ve_7(\th,z)\right)d\th.
$$
By Cauchy's integral formula
$$
\fint_0^{2\pi}\eta\re\tilde\Phi'_{n-1}(e^{\s+i\th})\frac{2c\b e^{cn}z}{ze^{-i\th}-1}d\th
=c\b\eta e^{cn}z\tilde\Phi'_{n-1}(e^\s z).
$$
So we obtain
\begin{equation}\label{AN}
A_n(z)
=-c\eta e^{cn}z\tilde\Phi'_{n-1}(e^\s z)+R_n(z)
\end{equation}
where
\begin{align*}
R_n(z)
&=\fint_0^{2\pi}
\left(
\frac{2c\b e^{cn}z}{ze^{-i\th}-1}\ve_5(\th) 
+\left(-\eta\re\tilde\Phi'_{n-1}(e^{\s+i\th})+\ve_5(\th)\right)
(\ve_6(\th,z)+\ve_7(\th,z))
\right)d\th\\
&\q\q\q\q\q\q\q\q-c(\b-1)\eta e^{cn}z\tilde\Phi'_{n-1}(e^\s z).
\end{align*}

Now suppose $n \leq N_0$. Then, using \eqref{E7B} and \eqref{E6B}, for $|z|=r$ with $r\ge1+\sqrt c$, we obtain \eqref{eq:Rn_est}.

\subsection{A more refined decomposition}

The estimate \eqref{eq:wn_bound}, is not sufficiently tight for all our needs. In this section, we give a decomposition of $w_n$, which can be used for more refined estimates.

Set
$$
l(z)=\log\frac{F(z)}z-c,\q
q(z)=l(z)-\frac{2c \b}{ z-1}.
$$
Then, for $|z|=r\ge1+\sqrt c$,
$$
|l(z)|\le\frac{Cc}{|z-1|},\q
|q(z)|\le\frac{Cc^{3/2}}{(r-1)|z-1|}.
$$
It follows that we can write
\[
F(z) = e^c z + \frac{2c \beta e^c z}{z-1} + \tilde{q}(z)
\]
where
\[
|\tilde{q}(z)|\le\frac{Cc^{3/2}|z|}{(r-1)|z-1|}.
\]
We will write 
$$
l(\th,z)=l(e^{-i\th}z),\q
q(\th,z)=q(e^{-i\th}z),\q 
\tilde q(\th,z)=e^{i\th}\tilde q(e^{-i\th}z),\q
F(\th,z)=e^{i\th}F(e^{-i\th}z).
$$
Recall the interpolation \eqref{Gdef}, which we can write as
$$
F_{s,\th}(z)=e^cz\exp(sl(\th,z)).
$$
We will use the following Taylor expansion
\[
\tilde\Phi_{n-1}(F(\th,z))=\sum_{k=0}^m\frac{l(\th,z)^k}{k!}(D^k\tilde\Phi_{n-1})(e^cz)+l(\th,z)^{m+1}\int_0^1\frac{(1-s)^m}{m!}(D^{m+1}\tilde\Phi_{n-1})(F_{s,\th}(z))ds
\]
where $m\in\N$.
Hence
\begin{align*}
\Delta_{n}(\th, z) &= e^{c(n-1)}\left ( \tilde\Phi_{n-1}(F(\th,z))+F(\th,z) - \tilde\Phi_{n-1}(e^c z) - e^{c}z \right ) \\
 &= e^{c(n-1)} \left ( F(\th,z)-e^{c}z + \sum_{k=1}^m \frac{l(\th,z)^k}{k!}(D^k \tilde \Phi_{n-1})(e^cz) \right ) \\
 &\q\q + e^{c(n-1)} l(\th,z)^{m+1} \int_0^1 \frac{(1-s)^m}{m!}(D^{m+1}\tilde \Phi_{n-1})(F_{s,\th}(z)) ds \\
 &= e^{c(n-1)} \left ( \frac{2c \beta e^c z}{e^{-i\th}z-1} + \tilde{q}(\th,z) + \sum_{k=1}^m \frac{l(\th,z)^k}{k!}(D^{k-1} \Psi_{n-1})(e^cz) \right ) \\
 &\q\q+e^{c(n-1)} l(\th,z)^{m+1} \int_0^1 \frac{(1-s)^m}{m!}(D^{m}\Psi_{n-1})(F_{s,\th}(z)) ds
\end{align*}
and so
\begin{align}
\label{eq:wn_refined}
\notag w_n(\th,z) = & \ e^{c(n-1)}\tilde{q}(\th,z) + e^{c(n-1)} \sum_{k=1}^m \frac{l(\th,z)^k}{k!}(D^{k-1} \Psi_{n-1})(e^cz)  \\
& \ +e^{c(n-1)} l(\th,z)^{m+1} \int_0^1 \frac{(1-s)^m}{m!}(D^{m}\Psi_{n-1})(F_{s,\th}(z)) ds.
\end{align}

\def\j{
\bb
\subsection{Details of error estimate for $|w|^{-\eta}$, $\eta \in [0,1]$}
Set $h=x+iy=w-1$.
Consider first the case $|h|\le1/4$.
Note that
$$
|w|^2=1+t+s
$$
where 
$$
s=x^2+y^2=|h|^2\le1/16
$$
and
$$
t=2x,\q |t|\le1/4+4x^2\le1/2.
$$
Now
$$
|w|^{-\eta}-1+\eta x
=(1+s+t)^{-\eta/2}
-(1+t)^{-\eta/2}
+(1+t)^{-\eta/2}-1+(\eta/2)t
$$
so
\begin{align*}
||w|^{-\eta}-1+\eta x|
&\le\frac\eta2s\left(\frac12\right)^{-\eta/2-1}
+\frac12(4x^2)\frac\eta2\left(\frac\eta2+1\right)\left(\frac12\right)^{-\eta/2-2}\\
&\le\left(\frac\eta22^{\eta/2+1}+\frac12.4\frac\eta2\left(\frac\eta2+1\right)2^{\eta/2+2}\right)|h|^2
\le14|h|^2.
\end{align*}
On the other hand, we have also
$$
||w|^{-\eta}-1+\eta x|
\le|w|^{-\eta}\vee1+|x|
$$
and
$$
|x|\le\frac14\vee\left(2|h|\right)^2
$$
so, if $|h|>1/4$, then
$$
||w|^{-\eta}-1+\eta x|
\le\frac54\left(|w|^{-\eta}\vee1\right)+4|h|^2
\le24(|w|^{-\eta}\vee1)|h|^2.
$$

\subsection{Next order terms for $\eta=1$} 
For $w=1+x+iy$, 
$$
\ve_1(w)=x^2-\frac{y^2}2+O(|w-1|^3)
$$
which we apply with 
$$
x+iy=x(\th)+iy(\th)=
\tilde\Phi_{n-1}'(e^{\s+i\th})
$$ 
to obtain $\ve_2(\th)$.
So
$$
\ve_3
=\fint_0^{2\pi}\ve_2(\th)d\th
=\fint_0^{2\pi}\left(x(\th)^2-\frac12y(\th)^2\right)d\th+\tilde\ve_3
=\frac14\fint_0^{2\pi}|\tilde\Phi_{n-1}'(e^{\s+i\th})|^2d\th.
+\tilde\ve_3
$$
where 
$$
|\tilde\ve_3|
\le C\fint_0^{2\pi}|\tilde\Phi_{n-1}'(e^{\s+i\th})|^3d\th.
$$
Here we used, for $f$ analytic in $\{|z|>1\|$, 
$$
\fint_0^{2\pi}f(e^{\s+i\th})^2d\th
=\frac1{2\pi i}\int_{|z|=\s}\frac{f(z)^2}zdz=0.
$$
So we obtain (in the tube)
$$
h_{n-1}(\th)
=1-\eta x(\th)+x(\th)^2-\frac{y(\th)^2}2-\fint_0^{2\pi}\left(x(\th)^2-\frac{y(\th)^2}2\right)d\th
+O(\|\tilde\Phi'_{n-1}\|^3_{\infty,e^\s}).
$$

\subsection{Integral formula for the real part of an analytic function}
For $f$ analytic in $\{|z|>1\}$ vanishing at $\infty$, and for $|z|>1$,
$$
\fint_0^{2\pi}\frac{f(e^{\s+i\th})}{1-ze^{-i\th}}d\th
=\frac1{2\pi i}\int_{|w|=1}\frac{f(e^\s w)}{w-z}dw=-f(e^\s z)
$$
while
$$
\fint_0^{2\pi}\frac{\overline{f(e^{\s+i\th})}}{1-ze^{-i\th}}d\th
=\fint_0^{2\pi}\frac{f^*(e^{-\s+i\th})}{1-ze^{-i\th}}d\th
=\frac1{2\pi i}\int_{|w|=1}\frac{f^*(e^{-\s}w)}{w-z}dw=0
$$
where $f^*(w)=\overline{f(1/\bar w)}$, which is analytic in $\{|w|<1\}$.
Hence
$$
\fint_0^{2\pi}\frac{\re f(e^{\s+i\th})}{ze^{-i\th}-1}d\th=\frac12f(e^\s z).
$$
\eb
}\jj

\def\j{
\bb
\subsection{Additional notes -- unmodified, containing $\g$}
Set
$$
l(z)=\log\frac{F(z)}z-c,\q
q(z)=l(z)-\frac{2c}{\g z-1}.
$$
Then, for $|z|=r\ge1+\sqrt c$,
$$
|l(z)|\le\frac{Cc}{|\g z-1|},\q
|q(z)|\le\frac{Cc^{3/2}}{(r-1)|\g z-1|}.
$$
We will write $l(\th,z)=l(e^{-i\th}z)$ and $q(\th,z)=q(e^{-i\th}z)$ and $F(\th,z)=e^{i\th}F(e^{-i\th}z)$.
As we showed above, 
$$
\Delta_n(\th,z)
=l(\th,z)\int_0^1D\Phi_{n-1}(\g_{\th,z}(s))ds
=e^{c(n-1)}\left(l(\th,z)\int_0^1\Psi_{n-1}(\g_{\th,z}(s))ds+F(\th,z)-e^cz\right).
$$
Let us use the interpolation trick again to write
$$
\Psi_{n-1}(\g_{\th,z}(s))
=\Psi_{n-1}(e^cz)
+l(\th,z)\int_0^sD\Psi_{n-1}(\g_{\th,z}(t))dt
$$
and hence
$$
\Delta_n(\th,z)
=e^{c(n-1)}\left(l(\th,z)\Psi_{n-1}(e^cz)+F(\th,z)-e^cz
+l(\th,z)^2\int_0^1D\Psi_{n-1}(\g_{\th,z}(s))(1-s)ds\right).
$$
Also, split the first term on the right using
$$
l(\th,z)
=\frac{2c}{\g e^{-i\th}z-1}
+q(\th,z).
$$
We obtain
\begin{align*}
\Delta_n(\th,z)
&=e^{c(n-1)}\left(\frac{2c}{\g e^{-i\th}z-1}\Psi_{n-1}(e^cz)
+q(\th,z)\Psi_{n-1}(e^cz)\right.\\
&\q\q\q\q\q\q\left.+F(\th,z)-e^cz
+l(\th,z)^2\int_0^1D^2\Phi_{n-1}(\g_{\th,z}(s))(1-s)ds\right).
\end{align*}
Recall that
$$
A_n(z)
=\fint_0^{2\pi}\Delta_n(\th,z)(h_{n}(\th)-1)d\th
$$
and we define $\ve_5(\th)$ and $\ve_8(z)$ by
$$
h_{n}(\th)-1
=-\eta\re\tilde\Phi_{n-1}'(e^{\s+i\th})
+\ve_5(\th)
$$
and
$$
A_n(z)
=-c\eta e^{cn}z\tilde\Phi_{n-1}'(e^\s\g z)
+\ve_8(z).
$$
On substituting the new expression for $\Delta_n(\th,z)$, we find
\begin{align*}
\ve_8(z)
&=-c\eta e^{cn}\tilde\Phi_{n-1}'(e^\s\g z)e^{-c}\Psi_{n-1}(e^cz)\\
&\q\q+D\Phi_{n-1}(e^cz)\fint_0^{2\pi}\frac{2c}{\g e^{-i\th}z-1}\ve_5(\th)d\th\\
&\q\q+\fint_0^{2\pi}\left(q(\th,z)D\Phi_{n-1}(e^cz)+l(\th,z)^2\int_0^1D^2\Phi_{n-1}(\g_{\th,z}(s))(1-s)ds\right)(h_{n}(\th)-1)d\th
\end{align*}
and
\begin{align*}
\ve_8(z)-\ve_8(\infty)
&=-c\eta e^{cn}\tilde\Phi_{n-1}'(e^\s\g z)e^{-c}\Psi_{n-1}(e^cz)\\
&\q\q
+2ce^{c(n-1)}\Psi_{n-1}(e^cz)\fint_0^{2\pi}\frac1{\g e^{-i\th}z-1}\ve_5(\th)d\th\\
&\q\q+2ce^{cn}\fint_0^{2\pi}\frac{e^{i\th}}{\g(\g e^{-i\th}z-1)}\ve_5(\th)d\th\\
&\q\q+\fint_0^{2\pi}\left(q(\th,z)D\Phi_{n-1}(e^cz)+l(\th,z)^2\int_0^1D^2\Phi_{n-1}(F_{s,\th}(z))(1-s)ds\right)(h_{n}(\th)-1)d\th
\end{align*}
\eb
}\jj

\section{Proofs of second order bounds} \label{sec:proof_est}

\subsection{Estimation of the second martingale term}\label{sec:proof_W_est}

In this section, we give the proof of Lemma \ref{lemma:wbound}, which bounds the second martingale term $\cW_n(z)$ in the decomposition \eqref{DECOMP} of the differentiated fluctuation process, which is given by
$$
\cW_n(z)=e^{-cn}\sum_{j=1}^nP_0^{n-j}DW_j(e^{c(1-\eta)(n-j)}z).
$$

By Burkholder's inequality, for all $p\in[2,\infty)$, there is a constant $C=C(p)<\infty$ such that
\begin{align*}
&\|\cW_n(z)1_{\{n\le N_0\}}\|_p^2 \\
&\q \q \q \le \ Ce^{-2cn}\left ( \| \max_{1 \leq j \leq n} X^W_{j,n}(e^{c(1-\eta)(n-j)}z)1_{\{j\le N_0\}}\|_p^2 + \sum_{j=1}^n\|Q^W_{j,n}(e^{c(1-\eta)(n-j)}z)1_{\{j\le N_0\}}\|_{p/2} \right )
\end{align*}
where
$$
X^W_{j,n}(z)
=|P_0^{n-j}DW_j(e^{c(1-\eta)(n-j)}z)| \quad \mbox{ and } \quad Q_{j,n}^W(z)
=\E(|P_0^{n-j}DW_j(e^{c(1-\eta)(n-j)}z)|^2|\cF_{j-1}).
$$
Then, on taking the $\|\cdot\|_{p/2,r}$-norm, we deduce that
\begin{equation}\label{MEE}
\tn\cW_n1_{\{n\le N_0\}}\tn_{p,r}^2
\le Ce^{-2cn}\left (  \tn \max_{1 \leq j \leq n } X^W_{j,n}1_{\{j\le N_0\}}\tn_{p,r}^2 +  \sum_{j=1}^n
\tn Q^W_{j,n}1_{\{j\le N_0\}}\tn_{p/2,r_{n-j}} \right ).
\end{equation}
While it is possible to use the estimate \eqref{eq:wn_bound} to bound this expression, the bound is  only sufficient to prove our final result for $\s \gg c^{1/3}$. In order to obtain a bound that works all the way down to $\s \gg c^{1/2}$, we need the refined decomposition \eqref{eq:wn_refined}, for some $m \in \N$ which we will choose later. 
Define
$$
U^i_n(z)=u^i_n(\Th_n,z)-\fint_0^{2\pi}u^i_n(\th,z)h_{n}(\th)d\th
$$
where
\begin{align*}
u^0_n(\th,z)&=e^{c(n-1)}\tilde{q}(\th,z),\\
u^1_n(\th,z)&=e^{c(n-1)}\sum_{k=1}^m\frac{l(\th,z)^k}{k!}D^{k-1}\Psi_{n-1}(e^cz),\\
u^2_n(\th,z)&=e^{c(n-1)}l(\th,z)^{m+1}\int_0^1\frac{(1-s)^m}{m!}D^m\Psi_{n-1}(F_{s,\th}(z))ds.
\end{align*}
Then $W_n=U^0_n+U^1_n+U_n^2$ so, with obvious notation,
$$
Q_{j,n}^W\le3(Q^0_{j,n}+Q^1_{j,n}+Q^2_{j,n})
$$
and so
$$
\|Q_{j,n}^W\|_{p/2,r}\le3\left(\|Q^0_{j,n}\|_{p/2,r}+\|Q^1_{j,n}\|_{p/2,r}+\|Q^2_{j,n}\|_{p/2,r}\right).
$$

We estimate the terms on the right.
First, for $j\le N_0$, we have
\begin{align*}
Q_{j,n}^i(z)
&=\E(|P_0^{n-j}DU^i_j(z)|^2|\cF_{j-1})
\le\E(|P_0^{n-j}Du^i_j(\Th_j,z)|^2|\cF_{j-1})\\
&=\fint_0^{2\pi}|P^{n-j}_0Du^i_j(\th,z)|^2h_{j}(\th)d\th
\le3\fint_0^{2\pi}|P^{n-j}_0Du^i_j(\th,z)|^2d\th.
\end{align*}

We start with $i=0$.
Then for all $|z|=r$ and $j\le N_0$,
\begin{align*}
Q_{j,n}^0(z)
&\le3\fint_0^{2\pi}|P^{n-j}_0D (e^{c(j-1)} \tilde{q}(e^{-i\th}z))|^2d\th \\
&=3e^{2c(j-1)}\|P_0^{n-j}D\tilde{q}\|_{2,r}^2\\
&\le Cc^3e^{2cj}\frac {r^3}{(r-1)^5}.
\end{align*}
Here, and in what follows, $C<\infty$ is a constant, which may only depend on $\Lambda$, $\eta$, $m$ and $p$, and which may change from line to line.

Next, consider $i=1$.
Note that
\[
Q^1_{j,n} \leq Ce^{2c(j-1)} \sum_{k=1}^m \fint_0^{2 \pi}|P_0^{n-j}D\left (l(\th,z)^kD^{k-1}\Psi_{j-1}(e^cz) \right )|^2 d\th.
\]
\def\j{
Suppose that 
\[
D^{k-1}\Psi_{j-1}(e^cz) = \sum_{x=1}^{\infty} a_x z^{-x}
\]
and
\[
l(z)^k = \sum_{y=1}^{\infty} b_yz^{-y}.
\]
Then 
\begin{align*}
P_0^{n-j}D\left (l(\th,z)^kD^{k-1}\Psi_{j-1}(e^cz) \right )
= -\sum_{x=1}^{\infty} \sum_{y=1}^{\infty} (x+y)p_0(x+y)^{n-j}a_x b_y e^{i \th y} z^{-(x+y)}.
\end{align*}
Hence
\begin{align*}
&\fint_0^{2 \pi}|P_0^{n-j}D\left (l(\th,z)^kD^{k-1}\Psi_{j-1}(e^cz) \right )|^2 d\th \\
&= \sum_{x',x=1}^{\infty} \sum_{y=1}^{\infty} (x+y)p_0(x+y)^{n-j}a_x b_y z^{-(x+y)} (x'+y)p_0(x'+y)^{n-j}\bar{a}_{x'} \bar{b}_y \bar{z}^{-(x'+y)} \\
&= \sum_{y=1}^{\infty} |b_y|^2 |P_0^{n-j}D\left (z^{-y} D^{k-1}\Psi_{j-1}(e^cz) \right )|^2
\end{align*}
and so, setting $\rho=(r+1)/2$ and $\tilde{\rho}=(3r+1)/4$,
\begin{align*}
\|Q^1_{j,n}\|_{p/2,r}
&\le C e^{2c(j-1)} \sum_{k=1}^m \sum_{y=1}^{\infty} |b_y|^2 \| |P_0^{n-j}D\left (z^{-y} D^{k-1}\Psi_{j-1}(e^cz) \right )|^2 \|_{p/2,r} \\
&= C e^{2c(j-1)} \sum_{k=1}^m \sum_{y=1}^{\infty} |b_y|^2 \| P_0^{n-j}D\left (z^{-y} D^{k-1}\Psi_{j-1}(e^cz) \right ) \|_{p,r}^2 \\
&\le C e^{2c(j-1)} \left (\frac{r}{r-1}\right )^2\sum_{k=1}^m \sum_{y=1}^{\infty} |b_y|^2 \| z^{-y} D^{k-1}\Psi_{j-1}(e^cz)  \|_{p,\tilde{\rho}}^2 \\
&\le Ce^{2c(j-1)}\left(\frac{r}{r-1}\right)^2\sum_{k=1}^m\sum_{y=1}^\infty|b_y|^2\tilde{\rho}^{-2y}\left(\frac{r}{r-1}\right)^{2(k-1)}\| \Psi_{j-1} \|_{p,\rho}^2 \\
&= C e^{2c(j-1)} \| \Psi_{j-1} \|_{p,\rho}^2 \sum_{k=1}^m \left ( \frac{r}{r-1}\right )^{2k} \|l(z)^k \|_{2,\tilde{\rho}}^2.
\end{align*}
}
We use the estimates \eqref{DEST},\eqref{LFG} and \eqref{PIE} to see that, for $\rho=(r+1)/2$ and $\tilde{\rho}=(3r+1)/4$, 
$$
\|Q^1_{j,n}\|_{p/2,r}
\le Ce^{2c(j-1)}\|\Psi_{j-1}\|_{p,\rho}^2\sum_{k=1}^m\left(\frac r{r-1}\right)^{2k}\|l^k\|_{2,\tilde{\rho}}^2.
$$
It follows from \eqref{partcond} that, for $|z|\ge1+\sqrt c$,
$$
|l(z)|
\le\frac{2(\L+1)c}{|z-1|}.
$$
Hence
\begin{align*}
\|Q^1_{j,n}\|_{p/2,r}
&\le Ce^{2cj}\|\Psi_{j-1}\|_{p,\rho}^2\sum_{k=1}^m\left(\frac r{r-1}\right)^{2k}\frac{c^{2k}}{r(r-1)^{2k-1}}\\
&\le Ce^{2cj}\|\Psi_{j-1}\|_{p,\rho}^2c^2\frac r{(r-1)^3}
\end{align*}
where in the last line we used that $r\ge1+\sqrt c$.

Finally we turn to $i=2$.
Then, for $|z|=r\ge1+\sqrt c$ and $\tilde\rho=(3r+1)/4$, using Jensen's inequality and the inequalities \eqref{DEST} and \eqref{PIE},
\begin{align*}
\|Q^2_{j,n}\|_{p/2,r}^{p/2}
&\le 3^{p/2}\fint_0^{2\pi}\left(\fint_0^{2\pi}|P^{n-j}_0Du^2_j(\th,re^{it})|^2d\th\right)^{p/2}dt\\
&\le3^{p/2}\fint_0^{2\pi}\fint_0^{2\pi}|P^{n-j}_0Du^2_j(\th,re^{it})|^pdtd\th\\
&\le C\left(\frac r{r-1}\right)^p \fint_0^{2\pi}\fint_0^{2\pi}|u^2_j(\th,\tilde\rho e^{it})|^pdtd\th\\
&=C\left(\frac r{r-1}\right)^p \fint_0^{2\pi}\fint_0^{2\pi}|u^2_j(\th+t,\tilde\rho e^{it})|^pdtd\th.
\end{align*}
Note that
$$
u^2_j(\th+t,\tilde\rho e^{it})=e^{c(j-1)} \left(\log\frac{F(\tilde\rho e^{-i\th})}{\tilde\rho e^{-i\th}}-c\right)^{m+1}
\int_0^1\frac{(1-s)^m}{m!}(D^{m}\Psi_{j-1})(e^{i t}F_{s,\th}(\tilde\rho ))ds.
$$
We use the inequalities \eqref{partcond}, $|F_{s,\th}(\tilde\rho)|\ge\tilde\rho$ and \eqref{GTHB} to see that, for $\rho=(r+1)/2\ge1+\sqrt c$,
$$
\|D^{m}\Psi_{j-1}\|_{p,|F_{s,\th}(\tilde\rho)|}\le C\left(\frac r{r-1}\right)^m\|\Psi_{j-1}\|_{p,\rho}.
$$
Hence we obtain
\begin{align*}
\fint_0^{2\pi}|u^2_j(\th+t,\tilde\rho e^{it})|^pdt
&\le C e^{c(j-1)p} \left|\frac{c}{\tilde\rho e^{-i\th}-1}\right|^{p(m+1)}
\left( \frac{r}{r-1}\right )^{mp}\|\Psi_{j-1}\|_{p,\rho}^p
\end{align*}
and then
\begin{align*}
&\fint_0^{2\pi}\fint_0^{2\pi}|u^2_j(\th+t,\tilde{\rho} e^{it})|^pdtd\th\\
&\q\q\le C \left ( \frac{c^{m+1}e^{c(j-1)}\|\Psi_{j-1}\|_{p,\rho} r^m}{(r-1)^m} \right )^p
\fint_0^{2\pi}\frac1{|\rho e^{-i\th}-1|^{(m+1)p}}d\th.
\end{align*}
Hence
\begin{align*}
\|Q^2_{j,n}\|_{p/2,r}^{p/2}
\le C \left ( \frac{c^{m+1}e^{c(j-1)}\|\Psi_{j-1}\|_{p,\rho} r^{m+1}}{(r-1)^{m+1}} \right )^p
\frac1{\rho(\rho-1)^{(m+1)p-1}}
\end{align*}
and hence
$$
\|Q^2_{j,n}\|_{p/2,r}
\le Cc^{2(m+1)}e^{2cj}\frac {r^{2(m+1)-2/p}}{(r-1)^{4(m+1)-2/p}}
\|\Psi_{j-1}\|^2_{p,\rho}.
$$
Fix $\ve\in(0,1/2)$ and assume that $r\ge1+c^{1/2-\ve}$.
Then, on choosing $m=\lceil 1/(8\ve) \rceil$, we obtain, for all $p\ge2$,
$$
\|Q^2_{j,n}\|_{p/2,r}\le Cc^{2}e^{2cj}\frac r{(r-1)^3}\|\Psi_{j-1}\|^2_{p,\rho}
$$
where $C$ now depends on $\ve$, in places where before it depended on $m$.

On combining our estimates, we have shown that, for $j\le N_0$, we have
$$
\|Q^W_{j,n}\|_{p/2,r}
\le Cc^2e^{2cj}\frac r{(r-1)^3}
\left(\|\Psi_{j-1}\|^2_{p,\rho}
+c\left(\frac r{r-1}\right)^2\right).
$$
We take the $L^{p/2}(\PP)$-norm to deduce that
\begin{equation*}
\tn Q^W_{j,n}1_{\{j\le N_0\}}\tn_{p/2,r}
\le Cc^2e^{2cj}\frac r{(r-1)^3}
\left(\tn\Psi_{j-1}1_{\{j\le N_0\}}\tn^2_{p,\rho}
+c\left(\frac r{r-1}\right)^2\right).
\end{equation*}
When bounding $\tn \max_{1 \leq j \leq n } X^W_{j,n}1_{\{j\le N_0\}}\tn_{p,r}^2$, it is sufficient to take $m=0$ in the decomposition above. In this case $u_n^1(\th,z)=0$, so
\[
\tn \max_{1 \leq j \leq n } X^W_{j,n}1_{\{j\le N_0\}}\tn_{p,r}^2
\leq 2 \left ( \tn \max_{1 \leq j \leq n } X^0_{j,n}1_{\{j\le N_0\}}\tn_{p,r}^2 + \tn \max_{1 \leq j \leq n } X^2_{j,n}1_{\{j\le N_0\}}\tn_{p,r}^2 \right ).
\]
By \eqref{DEST} and \eqref{PIE}, 
now, similarly to above,
\[
\tn P_0^{n-j}DU^i_j(z)\tn_{p,r} \leq C \left ( \frac{r}{r-1} \right ) \tn U^i_j(z) \tn_{p,\rho}.
\]
For $i=0$,
\begin{align*}
\|u^0_j(\Theta_j, z)\|_{p,r}^p 
&=\fint_0^{2\pi}|u^0_j(\Theta_j,re^{i t})|^p dt 
=\fint_0^{2\pi}|u^0_j(0,re^{i (t-\Theta_j)})|^p dt \\
&=\| u^0_j(0,z) \|_{p,r}^p 
\leq \frac{Ce^{pc(j-1)}c^{3p/2}r^{2p-1}}{(r-1)^{3p-1}}
\end{align*}
and hence, by the same argument as in the proof of Lemma \ref{lemma:mbound},
\[ 
\tn \max_{1 \leq j \leq n } X^0_{j,n}1_{\{j\le N_0\}}\tn_{p,r}^2
\leq \frac{Ce^{2c(n-1)}c^{3-2/p}r^{4-2/p}}{(r-1)^{6-2/p}}.
\]
For $i=2$,
\[
\tn U^2_j(z)\tn_{p,r}^p = \E \left ( \E(\| U^2_j(z) \|_{p,r}^p | \cF_{j-1}) \right ). 
\]
Now, by the computation above,
\begin{align*}
\E ( \|u^2_j(\Theta_j, z)\|_{p,r}^p |\cF_{j-1})
&=\fint_0^{2 \pi} \fint_0^{2\pi}|u^2_j(\th,re^{i t})|^p h_j(\th) dt d\th \\
&\leq \frac{Cc^pe^{cp(j-1)}\|\Psi_{j-1}\|_{p,\rho}^p r^{p-1}}{(r-1)^{2p-1}}.
\end{align*}
Hence
\[ 
\tn \max_{1 \leq j \leq n } X^2_{j,n}1_{\{j\le N_0\}}\tn_{p,r}^2
\leq \frac{Ce^{2c(n-1)}c^{2-2/p}r^{2-2/p}}{(r-1)^{4-2/p}}\max_{1\leq j \leq n}\tn\Psi_{j-1}1_{\{j\le N_0\}} \tn_{p,\rho}^2.
\]
Then, on using this estimate in \eqref{MEE}, we obtain \eqref{MNESTEG}.

Now suppose $\eta < 1$. It suffices to prove the result for $p$ sufficiently large, so assume $p > 1+1/(2\ve )$. We use our constraint on $r$, the monotonicity of norms \eqref{PINF}, and same integral comparison as in \eqref{MPA2}
to deduce from \eqref{MNESTEG} the estimate
\begin{align*}
\tn\cW_n1_{\{n\le N_0\}}\tn_{p,r}^2
&\le \frac{Cc}{r^2} \left( \frac r{r-1} \right)^2 \bigg( 1 
+ c^{1-2/p} \left( \frac r{r-1} \right)^{2-2/p} \bigg) 
\sup_{j \leq n}\tn\Psi_{j-1}1_{\{j\le N_0\}}\tn_{p,\rho}^2
\\ & \quad +\frac{Cc^2}{r^2} \left( \frac r{r-1}\right)^{4} \bigg( 1 
+ c^{1-2/p}  \left( \frac r{r-1} \right)^{2-2/p} \bigg). 
\end{align*}
The desired result follows, using our assumption on $p$.  

The case when $\eta=1$ is similar.

\subsection{Estimation of the remainder term}\label{sec:proof_R_est}
The remainder term in the decomposition \eqref{DECOMP} of the differentiated fluctuation process is given by
$$
\cR_n(z)=e^{-cn}\sum_{j=1}^nP_0^{n-j}DR_j(e^{c(1-\eta)(n-j)}z).
$$
In this section, we give the proof of Lemma \ref{lemma:rnbound}, which bounds this quantity.

By the triangle inequality
$$
\tn\cR_n1_{\{n\le N_0\}}\tn_{p,r}
\le e^{-cn}\sum_{j=1}^n\tn P_0^{n-j}DR_j1_{\{j\le N_0\}}\tn_{p,r_{n-j}}.
$$
For $n\le N_0$ and $|z|=r>1+\sqrt c$, we obtained in \eqref{eq:Rn_est} the estimate
\begin{align*}
|R_n(z)-R_n(\infty)|
&\le 
\frac{Cce^{cn}\d_0}r\left(1+\log\left(\frac r{r-1}\right)\right)\left(\d_0+\sqrt c\left(\frac r{r-1}\right)\right)\\
&\q\q
+Cc^{3/2}e^{cn}|\Psi_{n-1}(e^\s z)|
+Cce^{cn}\d_0\int_0^1\fint_0^{2\pi}\frac{|\Psi_{n-1}(F_{s,\th}(z))|}{|ze^{-i\th}-1|}d\th ds.
\end{align*}
We bound the $\|\cdot\|_{p,r}$-norm of the final term on the right as follows:
\begin{align*}
&\fint_0^{2\pi}
\left(\int_0^1\fint_0^{2\pi}
\frac{|\Psi_{n-1}(F_{s,\th}(re^{iu}))|}{|re^{iu}e^{-i\th}-1|}
d\th ds\right)^pdu\\
&\q\q=\int_0^1\dots\int_0^1
\fint_0^{2\pi}
\dots
\fint_0^{2\pi}
\left(
\fint_0^{2\pi}
\prod_{m=1}^p
\frac{|\Psi_{n-1}(F_{s_m,\th_m}(re^{iu}))|}{|re^{iu}e^{-i\th_m}-1|}du
\right)
d\th_1\dots d\th_pds_1\dots ds_p\\
&\q\q=\int_0^1\dots\int_0^1
\fint_0^{2\pi}
\dots
\fint_0^{2\pi}
\left(
\fint_0^{2\pi}
\prod_{m=1}^p
\frac{|\Psi_{n-1}(e^{iu}F_{s_m,\t_m}(r))|}{|re^{-i\t_m}-1|}du
\right)
d\t_1\dots d\t_pds_1\dots ds_p\\
&\q\q\le
\|\Psi_{n-1}\|_{p,r}^p
\fint_0^{2\pi}
\dots
\fint_0^{2\pi}
\prod_{m=1}^p
\frac1{|re^{-i\t_m}-1|}
d\t_1\dots d\t_p\\
&\q\q\le
\left(\frac Cr\right)^p\left(
1+
\log
\left(\frac r{r-1}\right)
\right)^p
\|\Psi_{n-1}\|_{p,r}^p.
\end{align*}
We used the change of variable $\t_m=\th_m-u$ and the identity
$$
F_{s_m,\t_m+u}(re^{iu})=e^{iu}F_{s_m,\t_m}(r)
$$
in the second equality.
Then we used H\"older's inequality and the fact that $|F_{s,\t_m}(r)|\ge r$ for the first inequality, and we used
$$
\fint_0^{2\pi}\frac 1{|re^{-i\t}-1|}d\t
\le\frac Cr\left(1+\log\left(\frac r{r-1}\right)\right)
$$
for the second inequality.
Hence, for all $p\in\N$,
\begin{align*} 
\notag
\|R_n-R_n(\infty)\|_{p,r}
&\le\frac{Cce^{cn}\d_0}r\left(\d_0+\|\Psi_{n-1}\|_{p,r}+\sqrt c\left(\frac r{r-1}\right)\right)
\left(1+\log\left(\frac r{r-1}\right)\right)\\
&\q\q\q\q\q\q\q\q\q\q+Cc^{3/2}e^{cn}\|\Psi_{n-1}\|_{p,r}.
\end{align*}
We use \eqref{DEST} and \eqref{PIE} to obtain \eqref{RNEST}.

Now suppose $n \leq T/c$ for some constant $T$. If $\eta<1$, the result follows, using the integral comparison
\begin{align*}
\sum_{j=1}^{n}\frac c{r_j-1}
&\le\frac c{r-1}+\int_0^n\frac c{re^{c(1-\eta)\t}-1}d\t
\le\frac c{r-1}+\int_0^n\frac{ce^{c(1-\eta)\t}}{re^{c(1-\eta)\t}-1}d\t\\
&\le\frac c{r-1}+\frac1r\left(\frac1{1-\eta}\log\left(\frac r{r-1}\right)+T\right)
\le\frac Cr\left(1+\log\left(\frac r{r-1}\right)\right).
\end{align*}
The argument when $\eta=1$ is similar.

\end{appendices}


\section*{Acknowledgements}
We would like to thank Alan Sola for helpful discussions and for suggestions on an earlier draft of this paper. Amanda Turner would like to thank the University of Geneva for a visiting position in 2019-20 during which time this work was completed.

\bibliography{p}
\end{document}